\documentclass{amsart}

 \usepackage[foot]{amsaddr}
\let\oldtocsection=\tocsection
\let\oldtocsubsection=\tocsubsection 
\let\oldtocsubsubsection=\tocsubsubsection
\renewcommand{\tocsection}[2]{\vspace{0.5em}\hspace{0em}\oldtocsection{#1}{#2}}
\renewcommand{\tocsubsection}[2]{\vspace{0.5em}\hspace{1em}\oldtocsubsection{#1}{#2}}
\renewcommand{\tocsubsubsection}[2]{\vspace{0.5em}\hspace{2em}\oldtocsubsubsection{#1}{#2}}
\usepackage{graphicx,url,etoolbox}
\usepackage{lipsum}
\usepackage{lastpage}
\usepackage{graphicx,cancel,xcolor,comment,graphicx,geometry}
\usepackage{graphicx,comment,dsfont}
\usepackage[utf8]{inputenc}
\usepackage[T1]{fontenc}
\usepackage[english]{babel}
\everymath{\displaystyle}
\usepackage{fancyhdr}
\def\R{\mathbb R}

\def\Q{\mathbb Q}

\newcounter{dummy} 
\numberwithin{dummy}{section}
\newtheorem{Theorem}[dummy]{Theorem}

\newtheorem{defi}[dummy]{Definition}

\newtheorem{lemma}[dummy]{Lemma}

\newtheorem{Proposition}[dummy]{Proposition}
\newtheorem{Remark}[dummy]{Remark}

\numberwithin{equation}{section}

\newcommand\xqed[1]{\leavevmode\unskip\penalty9999 \hbox{}\nobreak\hfill\quad\hbox{#1}}
\RequirePackage[bookmarks,%
                colorlinks,%
                urlcolor=blue,%
                citecolor=blue,%
                linkcolor=blue,%
                hyperfigures,%
                pagebackref=true,%
                pdfcreator=LaTeX,%
                breaklinks=true,%
                pdfpagelayout=SinglePage,%
                bookmarksopen=true,%
                bookmarksopenlevel=2]{hyperref}
\def\nline{\\ \noalign{\medskip}}
\usepackage{tikz}
\usetikzlibrary{decorations.pathmorphing,patterns,scopes,intersections,calc}
\usepackage{caption}
\title[Wave equation with Kelvin-Voigt  damping and time delay]{Stability results for an elastic-viscoelastic waves interaction systems with localized Kelvin-Voigt  damping and with an internal or boundary  time delay} 

\author {Mouhammad Ghader$^{1}$}
\author{Rayan Nasser$^{1,2}$}
\author{Ali Wehbe$^{1}$}
\address{$^1$Lebanese University, Faculty of sciences 1, Khawarizmi Laboratory of  Mathematics and Applications-KALMA, Hadath-Beirut, Lebanon.}
\address{$^2$Université de Bretagne-Occidentale, France.}
\email{mhammadghader@hotmail.com, rayan.nasser94@hotmail.com,  ali.wehbe@ul.edu.lb}
\keywords{Wave equation; Kelvin-Voigt damping; Time delay; Semigroup; Stability.}
\subjclass{35L05; 35B35;   93D15;   93D20}
\begin{document}

\begin{abstract}
We investigate the stability of a one-dimensional wave equation with non smooth localized internal viscoelastic damping of Kelvin-Voigt type and with boundary or localized internal delay feedback. The main novelty in this paper is that the Kelvin-Voigt and the delay damping are both localized via non smooth coefficients. In the case that the Kelvin-Voigt damping is localized faraway from the tip and the wave is subjected to a locally distributed internal or boundary delay feedback, we prove that the energy of the system decays polynomially of type $t^{-4}$. However, an exponential decay of the energy of the system is established provided that the Kelvin-Voigt damping is localized near a part of the boundary and a time delay damping acts on the second boundary. While, when the Kelvin-Voigt and the internal delay damping are both localized via non smooth coefficients near the tip, the energy of the system decays polynomially of type $t^{-4}$. Frequency domain arguments combined with piecewise multiplier techniques are employed. 
\end{abstract}
         
\maketitle
\pagenumbering{roman}
     \vspace{-0.5cm} 
\maketitle
\tableofcontents
\clearpage
\pagenumbering{arabic}
\setcounter{page}{1}
\newpage
\section{Introduction}\label{NG-S-1} 
Viscoelastic materials feature intermediate characteristics between purely elastic and purely viscous behaviors, {\it i.e.} they display both behaviors when undergoing deformation. In wave equations, when the viscoelastic controlling parameter is null, the viscous property vanishes and the wave equation becomes a pure elastic wave equation. However, time delays arise in many applications and practical problems like physical, chemical, biological, thermal and economic phenomena, where an arbitrary small delay may destroy the well-posedness of the problem and destabilize it. Actually, it is well-known that simplest delay equations  of parabolic type,
$$u_t (x,t) = \Delta u (x, t - \tau), $$
\noindent or hyperbolic type
$$u_{tt} (x,t) = \Delta u (x, t - \tau), $$
\noindent with a delay parameter $\tau > 0,$ are not well-posed. Their instability is due to the existence of a sequence of initial data remaining bounded, while the corresponding solutions go to infinity in an exponential manner at a fixed time (see \cite{dreher2009,JORDAN2008414}). \\[0.1in]
\indent   The stabilization of a wave equation with Kelvin-Voigt type damping and internal or boundary time delay has attracted the attention of many authors in the last five years. Indeed, in 2016 Messaoudi {\it et al.} studied the stabilization of a wave equation with global Kelvin-Voigt damping and internal time delay in the multidimensional case (see \cite{messaoudi2016}), and they obtained an exponential stability result. In the same year, Nicaise {\it et al.} in \cite{Nicaise-Pignotti16} considered the multidimensional wave equation with localized Kelvin-Voigt damping and mixed boundary condition with time delay. They obtained an exponential decay of the energy regarding that the damping is acting on a neighborhood of part of the boundary via a smooth coefficient. Also, in 2018, Anikushyn {\it et al.} in \cite{Anikushyn2018} considered the stabilization of a wave equation with global viscoelastic material subjected to an internal strong time delay where a global exponential decay rate was obtained. Thus, it seems to us that there are no previous results concerning the case of wave equations with internal localized  Kelvin-Voigt type damping  and boundary or internal time delay, especially in the absence of smoothness of the damping coefficient even in the one dimensional case. So, we are interested in studying the stability of elastic wave equation with local   Kelvin–Voigt damping and with boundary or internal time delay (see Systems \eqref{NG-E(1.1)} and  \eqref{NG-E(1.2)}). \\[0.1in]
\indent This paper investigates the study of the stability of a string with Kelvin-Voigt type damping localized via non-smooth coefficient and subjected to a localized internal or boundary time delay. Indeed, in the first part of this paper,  we study the stability of elastic wave equation with local   Kelvin–Voigt damping,  boundary feedback and time delay term  at the boundary,  {\it i.e.} we consider the following system
\begin{equation}\label{NG-E(1.1)}
\left\{
\begin{array}{ll}
 {U}_{tt}(x,t) -  \left[ \kappa\, {U}_x(x,t) +\delta_1 \chi_{(\alpha,\beta)}\,  {U}_{xt}(x,t)\right]_x = 0, & (x,t)\in  (0,L) \times (0,+\infty), \nline
    {U}(0,t)=0, & t\in   (0,+\infty), \nline
    {U}_x(L,t)=-\delta_3 U_t(L,t)-\delta_2 U_t(L,t-\tau), & t\in   (0,+\infty), \nline
    \left({U}(x,0),{U}_t(x,0)\right)= \left(U_0(x),U_1(x)\right),   & x\in  (0,L),\nline
 \displaystyle{U_t(L,t) = f_0(L,t) },&\displaystyle{t\in (-\tau,0) }, 
\end{array}
\right.
\end{equation}
where $L,\ \tau,\ \delta_1$ and $\delta_3$  are strictly positive constant numbers, $\delta_2$ is a non zero real number and the initial data $(U_0,U_1,f_0)$ belongs to a suitable space.  
 Here $0\leq\alpha <\beta<L$ and $U=u\chi_{(0,\alpha)}+v\chi_{(\alpha,\beta)}+w\chi_{(\beta,L)}$, with $\chi_{(a,b)}$ is the characteristic function of the interval $(a,b).$ We assume that there exist strictly positive constant numbers $\ \kappa_1,\ \kappa_2,\ \kappa_3$, such that $\kappa=\kappa_1\chi_{(0,\alpha)}+\kappa_2\chi_{(\alpha,\beta)}+\kappa_3\chi_{(\beta,L)}.$
In fact, here we will consider two cases. In the first case, we divide the bar into 3 pieces; the first piece is an elastic part, the second piece is the viscoelastic part and in the third piece, the time delay feedback is effective at the ending point of the piece,  {\it i.e.} we consider the case $\alpha>0$ (see Figure \ref{Fig1}). While, in the second case, we divide the bar into 2 pieces;  the first piece is the viscoelastic part and in the second piece the time delay feedback is effective at the ending point of the piece, {\it i.e.} we consider the case $\alpha=0$ (see Figure \ref{Fig2}).  Remark, here,  in both cases, the Kelvin–Voigt damping   is effective on a part of the piece  and the time delay is effective at $L$.
\newpage
\begin{figure}[h!]
     \begin{center}
     {
\begin{tikzpicture}

\draw[thick,-,blue](-5,-1)--(5,-1);
\draw[thick,-,blue](-5,1)--(5,1);
\draw[-,blue](-2,0.6)--(2,0.6);
\draw[-,blue](-2,0.2)--(2,0.2);
\draw[-,blue](-2,-0.2)--(2,-0.2);
\draw[-,blue](-2,-0.6)--(2,-0.6);
\draw[-,blue](-2,-1)--(2,-1);
\draw[-,blue](-2,-1)--(-2,1);
\draw[-,blue](-1.6,-1)--(-1.6,1);
\draw[-,blue](-1.2,-1)--(-1.2,1);
\draw[-,blue](-0.8,-1)--(-0.8,1);
\draw[-,blue](-0.4,-1)--(-0.4,1);
\draw[-,blue](0,-1)--(0,1);
\draw[-,blue](0.4,-1)--(0.4,1);
\draw[-,blue](0.8,-1)--(0.8,1);
\draw[-,blue](1.2,-1)--(1.2,1);
\draw[-,blue](1.6,-1)--(1.6,1);
\draw[-,blue](2,-1)--(2,1);
\draw[|-|>,blue](-5,-1.5)--(5,-1.5);

\draw[-,blue](2,-1)--(2,1);

\node at (-2,-1.5) {${\color{blue}\bullet}$};
\node at (2,-1.5) {${\color{blue}\bullet}$};

\draw[->] (5,1)--(4.6,1.2);

\coordinate [label=below:\scalebox{0.5}{$ \alpha$}](a1)  at (-2,-1.5);
\coordinate [label=below:\scalebox{0.5}{$ \beta$}](a2)  at (2,-1.5);
\coordinate [label=below:\scalebox{0.5}{$ v(x)$}](a3)  at (0,-1);
\coordinate [label=below:\scalebox{0.5}{$ u(x)$}](a4)  at (-3.5,-1);
\coordinate [label=below:\scalebox{0.5}{$ w(x)$}](a5)  at (3.5,-1);
\coordinate [label=below:\scalebox{0.5}{$0$}](a6)  at (-5,-1.5);
\coordinate [label=below:\scalebox{0.5}{$L$}](a7)  at (5,-1.5);

\node at (-3.5,1.2) {\scalebox{0.6}{Elastic part}};
\node at (0,1.2) {\scalebox{0.6}{Viscoelastic part}};
\node at (3.4,1.2) {\scalebox{0.6}{Boundary delay feedback}};

 { [color=blue,pattern color=blue]
\draw[thick,blue] (-5,1.2) -- (-5,-1.2);
\fill [blue, pattern = north east lines] (-5,1.2) rectangle (-5.2,-1.2);
}
 { [color=blue,pattern color=blue]
\draw[ultra thick,red] (5,1.2) -- (5,-1.2);
\fill [blue, pattern = north west lines] (5,1.2) rectangle (5.2,-1.2);
}

\end{tikzpicture}
 }
     \caption{K-V damping  is acting localized in the internal of the body and time delay feedback is effective at $L$}

     \label{Fig1}
      \end{center}
      \end{figure}
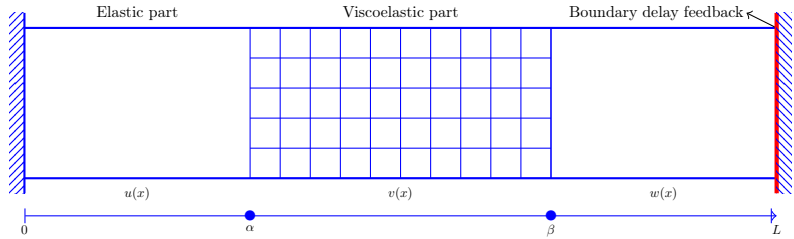
\begin{figure}[h!]
     \begin{center}
     {
\begin{tikzpicture}

\draw[thick,-,blue](-2,-1)--(5,-1);
\draw[thick,-,blue](-2,1)--(5,1);
\draw[-,blue](-2,0.6)--(2,0.6);
\draw[-,blue](-2,0.2)--(2,0.2);
\draw[-,blue](-2,-0.2)--(2,-0.2);
\draw[-,blue](-2,-0.6)--(2,-0.6);
\draw[-,blue](-2,-1)--(2,-1);
\draw[-,blue](-2,-1)--(-2,1);
\draw[-,blue](-1.6,-1)--(-1.6,1);
\draw[-,blue](-1.2,-1)--(-1.2,1);
\draw[-,blue](-0.8,-1)--(-0.8,1);
\draw[-,blue](-0.4,-1)--(-0.4,1);
\draw[-,blue](0,-1)--(0,1);
\draw[-,blue](0.4,-1)--(0.4,1);
\draw[-,blue](0.8,-1)--(0.8,1);
\draw[-,blue](1.2,-1)--(1.2,1);
\draw[-,blue](1.6,-1)--(1.6,1);
\draw[-,blue](2,-1)--(2,1);
\draw[|-|>,blue](-2,-1.5)--(5,-1.5);

\draw[-,blue](2,-1)--(2,1);

\draw[->] (5,1)--(4.6,1.2);
\node at (2,-1.5) {${\color{blue}\bullet}$};

\coordinate [label=below:\scalebox{0.5}{$ 0$}](a1)  at (-2,-1.5);
\coordinate [label=below:\scalebox{0.5}{$ \beta$}](a2)  at (2,-1.5);
\coordinate [label=below:\scalebox{0.5}{$ v(x)$}](a3)  at (0,-1);
\coordinate [label=below:\scalebox{0.5}{$ w(x)$}](a5)  at (3.5,-1);
\coordinate [label=below:\scalebox{0.5}{$L$}](a7)  at (5,-1.5);

\node at (0,1.2) {\scalebox{0.6}{Viscoelastic part}};
\node at (3.4,1.2) {\scalebox{0.6}{Boundary delay feedback}};

 { [color=blue,pattern color=blue]
\draw[thick,blue] (-2,1.2) -- (-2,-1.2);
\fill [blue, pattern = north east lines] (-2,1.2) rectangle (-2.2,-1.2);
}
 { [color=blue,pattern color=blue]
\draw[ultra thick,red] (5,1.2) -- (5,-1.2);
\fill [blue, pattern = north west lines] (5,1.2) rectangle (5.2,-1.2);
}

\end{tikzpicture}
 }
     \caption{K-V damping  is acting localized near the boundary of the body and time delay feedback is effective at $L$}
     \label{Fig2}
      \end{center}
      \end{figure}
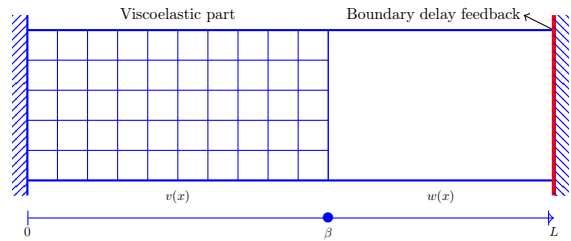
$\vspace{-0.4cm}\\$      
\noindent In the second part of this paper, we study the stability of elastic wave equation with local  Kelvin–Voigt damping and local internal time delay. This system takes the following form
\begin{equation}\label{NG-E(1.2)}
\left\{
\begin{array}{ll}
 {U}_{tt}(x,t) -  \left[ \kappa\, {U}_x(x,t) + \chi_{(\alpha,\beta)}\left(\delta_1\, {U}_{xt}(x,t)+\delta_2\, {U}_{xt}(x,t-\tau)\right)\right]_x = 0, & (x,t)\in  (0,L) \times (0,+\infty), \nline
    {U}(0,t)=U(L,t)=0, & t\in   (0,+\infty), \nline
    \left({U}(x,0),{U}_t(x,0)\right)= \left(U_0(x),U_1(x)\right),   & x\in  (0,L),\nline
 \displaystyle{U_t(x,t) = f_0(x,t) },&\displaystyle{(x,t)\in (0,L)\times (-\tau,0) }, 
\end{array}
\right.
\end{equation}
where $L,\ \tau$ and $\delta_1$  are strictly positive constant numbers, $\delta_2$ is a non zero real number  and the initial data $(U_0,U_1,f_0)$ belongs to a suitable space.  Here $0<\alpha <\beta<L$ and $U=u\chi_{(0,\alpha)}+v\chi_{(\alpha,\beta)}+w\chi_{(\beta,L)}$, with $\chi_{(a,b)}$ is the characteristic function of the interval $(a,b).$ We assume that there exist strictly positive constant numbers $\ \kappa_1,\ \kappa_2,\ \kappa_3$, such that $\kappa=\kappa_1\chi_{(0,\alpha)}+\kappa_2\chi_{(\alpha,\beta)}+\kappa_3\chi_{(\beta,L)}.$ In fact, here we will divide the bar into 3 pieces; the first piece is an elastic part, in the second piece the Kelvin–Voigt damping and  the time delay  are effective and the third piece is an elastic part (see Figure \ref{Fig3}). So, the  Kelvin–Voigt damping and  the time delay  are effective on $(\alpha,\beta)$.   
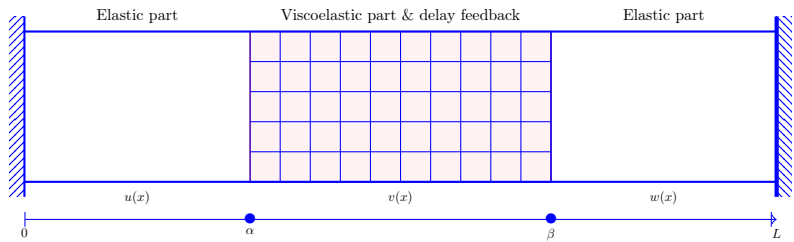
\begin{figure}[h!]
     \begin{center}
     {
\begin{tikzpicture}
  
\filldraw[
        draw=red,%
        fill=red!5,%
    ]         
           (-2,-1)--(2,-1)--(2,1)--(-2,1)
            -- cycle; 

\draw[thick,-,blue](-5,-1)--(5,-1);
\draw[thick,-,blue](-5,1)--(5,1);
\draw[-,blue](-2,0.6)--(2,0.6);
\draw[-,blue](-2,0.2)--(2,0.2);
\draw[-,blue](-2,-0.2)--(2,-0.2);
\draw[-,blue](-2,-0.6)--(2,-0.6);
\draw[-,blue](-2,-1)--(2,-1);
\draw[-,blue](-2,-1)--(-2,1);
\draw[-,blue](-1.6,-1)--(-1.6,1);
\draw[-,blue](-1.2,-1)--(-1.2,1);
\draw[-,blue](-0.8,-1)--(-0.8,1);
\draw[-,blue](-0.4,-1)--(-0.4,1);
\draw[-,blue](0,-1)--(0,1);
\draw[-,blue](0.4,-1)--(0.4,1);
\draw[-,blue](0.8,-1)--(0.8,1);
\draw[-,blue](1.2,-1)--(1.2,1);
\draw[-,blue](1.6,-1)--(1.6,1);
\draw[-,blue](2,-1)--(2,1);
\draw[|-|>,blue](-5,-1.5)--(5,-1.5);

\draw[-,blue](2,-1)--(2,1);

\node at (-2,-1.5) {${\color{blue}\bullet}$};
\node at (2,-1.5) {${\color{blue}\bullet}$};

\coordinate [label=below:\scalebox{0.5}{$ \alpha$}](a1)  at (-2,-1.5);
\coordinate [label=below:\scalebox{0.5}{$ \beta$}](a2)  at (2,-1.5);
\coordinate [label=below:\scalebox{0.5}{$ v(x)$}](a3)  at (0,-1);
\coordinate [label=below:\scalebox{0.5}{$ u(x)$}](a4)  at (-3.5,-1);
\coordinate [label=below:\scalebox{0.5}{$ w(x)$}](a5)  at (3.5,-1);
\coordinate [label=below:\scalebox{0.5}{$0$}](a6)  at (-5,-1.5);
\coordinate [label=below:\scalebox{0.5}{$L$}](a7)  at (5,-1.5);

\node at (-3.5,1.2) {\scalebox{0.6}{Elastic part}};
\node at (0,1.2) {\scalebox{0.6}{Viscoelastic part $\&$ delay feedback}};
\node at (3.5,1.2) {\scalebox{0.6}{Elastic part}};

 { [color=blue,pattern color=blue]
\draw[thick,blue] (-5,1.2) -- (-5,-1.2);
\fill [blue, pattern = north east lines] (-5,1.2) rectangle (-5.2,-1.2);
}
 { [color=blue,pattern color=blue]
\draw[ultra thick,blue] (5,1.2) -- (5,-1.2);
\fill [blue, pattern = north west lines] (5,1.2) rectangle (5.2,-1.2);
}

\end{tikzpicture}
 }
     \caption{Local internal K-V damping and  Local internal delay feedback}
     \label{Fig3}
      \end{center}
      \end{figure}
$\vspace{-0.1cm}\\$
\indent In the literature, Datko {\it et al.}  studied in 1985 the one dimensional wave equation which models the vibrations of a string fixed at one end and free at the other one (see \cite{datko1985}). The system is given by the following:
\begin{equation}\label{EQ Datko 1985}
\left\{
\begin{array}{ll}
 {u}_{tt}(x,t) -  u_{xx}(x,t) + 2 a u_t (x,t) + a^2 u(x,t) = 0, & (x,t)\in  (0,1) \times (0,+\infty), \nline
    {u}(0,t)=0, & t\in   (0,+\infty), \nline
    {u}_x(1,t)=-\kappa u_t(1,t-\tau), & t\in   (0,+\infty), 
\end{array}
\right.
\end{equation}
where the delay parameter $\tau$ is strictly positive, $a>0$ and $\kappa > 0$. So, the above system models a string having a boundary feedback with delay at the free end. They showed that if $ \kappa \left({e^{2a} + 1}\right) < { e^{2a} -1},$ then System \eqref{EQ Datko 1985} is strongly stable for all small enough delays. However, if  $\kappa \left({e^{2a} + 1}\right) > { e^{2a} -1},$ then there exists an open  set $D$ dense in $(0, +\infty)$, such that for all $\tau$ in $D$, System \eqref{EQ Datko 1985} admits exponentially unstable solutions.  Moreover, in the absence of delay in System \eqref{EQ Datko 1985} ({\it i.e} $\tau = 0$)  and $a \geq 0, \kappa \geq 0$, its energy decays exponentially to zero under the condition $a^2 + \kappa^2 > 0$ (see \cite{Chen1979}). In 1990, Datko in \cite{DATKO1990} considered  the boundary feedback stabilization of a one-dimensional wave equation with time delay (see Example 3.5 in \cite{DATKO1990}). The system is given by the following:
\begin{equation}\label{EQ Datko 1990}
\left\{
\begin{array}{ll}
 {u}_{tt}(x,t) -  u_{xx}(x,t) - \delta u_{xxt}(x,t) = 0, & (x,t)\in  (0,1) \times (0,+\infty), \nline
    {u}(0,t)=0, & t\in   (0,+\infty), \nline
    {u}_x(1,t)=-\kappa u_t(1,t-\tau), & t\in   (0,+\infty),
\end{array}
\right.
\end{equation}
where $\tau>0,$ $\kappa>0$ and $\delta>0$. He proved that System \eqref{EQ Datko 1990} is unstable for an arbitrary small value of $\tau$. In 2006, Xu {\it et al.}  in \cite{Xu2006} investigated the following closed loop system with homogeneous Dirichlet boundary condition at $x=0$ and delayed Neumann boundary feedback at $x=1$:
\begin{equation}\label{Eq Xu 2006)}
\left\{
\begin{array}{ll}
 {u}_{tt}(x,t) -  u_{xx}(x,t) = 0, & (x,t)\in  (0,1) \times (0,+\infty), \nline
    {u}(0,t)=0, & t\in   (0,+\infty), \nline
    {u}_x(1,t)=-\kappa u_t(1,t)-\kappa (1 - \mu)u_t(1,t-\tau), & t\in   (0,+\infty), \nline
    \left({u}(x,0),{u}_t(x,0)\right)= \left(u_0(x),u_1(x)\right),   & x\in  (0,1),\nline
 \displaystyle{u_t(1,t) = f_0(1,t) },&\displaystyle{t\in (-\tau,0) }.
\end{array}
\right.
\end{equation}
The above system represents a wave equation that is fixed at one end and subjected to a boundary control input possessing a partial time delay of weight $(1 - \mu)$ at the  other end. They proved the following stability results:
\begin{enumerate}
\item[1.]  If $\mu > 2^{-1},$ then  System \eqref{Eq Xu 2006)} is uniformly stable.
\item[2.]   If $\mu = 2^{-1}$ and  $\tau \in \Q \cap (0,1)$, then System \eqref{Eq Xu 2006)} is unstable.
\item[3.] If $\mu = 2^{-1}$ and $\tau \in (\R \setminus \Q )\cap (0,1)$, then System \eqref{Eq Xu 2006)} is asymptotically stable.
\item[4.] If $\mu < 2^{-1},$ then System \eqref{Eq Xu 2006)} is always unstable.
\end{enumerate}
 Later on, in 2008, Guo and Xu in \cite{Guo2008} studied the stabilization of a wave equation in the 1-D case where it is effected by a boundary control and output observation suffering from time delay. The system is given by the following:
\begin{equation*}\label{Eq Guo 2008)}
\left\{
\begin{array}{ll}
 {u}_{tt}(x,t) -  u_{xx}(x,t) = 0, & (x,t)\in  (0,1) \times (0,+\infty), \nline
    {u}(0,t)=0, & t\in   (0,+\infty), \nline
    {u}_x(1,t)=w(t), & t\in   (0,+\infty), \nline
    \left({u}(x,0),{u}_t(x,0)\right)= \left(u_0(x),u_1(x)\right),   & x\in  (0,1),\nline
 y(t) = u_t (1, t - \tau), & t\in   (0,+\infty),
\end{array}
\right.
\end{equation*}
where $w$ is the control and $y$ is the output observation. Using the separation principle, the authors proved that the above delayed system is exponentially stable. In 2010, Gugat  in \cite{gugat2010} studied the wave equation which models a string of length $L$ that is rigidly fixed at one end and stabilized with a boundary feedback and constant delay at the other end. The problem is described by the following system
\begin{equation*}
\left\{
\begin{array}{ll}
 {u}_{tt}(x,t) - c^2 u_{xx}(x,t) = 0, & (x,t)\in  (0,L) \times (0,+\infty), \nline
    {u}(0,t)=0, & t\in   (0,+\infty), \nline
    {u}_x(L,t)=0, & t\in   (0, 2L\, c^{-1}), \nline
    {u}_x(L,t)=c^{-1}\, \lambda u_t\left(L,t - 2L\, c^{-1}\right), & t\in   (2L\, c^{-1},+ \infty), \nline
    \left({u}(x,0),{u}_t(x,0), u(0,0)\right)= \left(u_0(x),u_1(x), 0 \right),   & x\in  (0,L),
  \end{array}
\right.
\end{equation*}
where $\lambda$ is a real number and $c>0$. Gugat  proved that the above system is exponentially stable.  In 2011, J. Wang {\it et al.} in \cite{wang2011} studied the stabilization of a wave equation under boundary control and collocated observation with time delay. The system is given by the following:
\begin{equation*}\label{Wang 2011}
\left\{
\begin{array}{ll}
 {u}_{tt}(x,t) -  u_{xx}(x,t) = 0, & (x,t)\in  (0,1) \times (0,+\infty), \nline
    {u}(0,t)=0, & t\in   (0,+\infty), \nline
    {u}_x(1,t)=\kappa u_t(1,t - \tau), & t\in   (0,+ \infty), \nline
    \left({u}(x,0),{u}_t(x,0)\right)= \left(u_0(x),u_1(x) \right),   & x\in  (0,1).
  \end{array}
\right.
\end{equation*}
They showed that if the delay is equal to even multiples of the wave propagation time, then the above closed loop system is exponentially stable under sufficient and necessary conditions for $\kappa$. Else, if the delay is an odd multiple of the wave propagation time, thus the closed loop system is unstable. In 2013, H. Wang {\it et al.} in \cite{wang2013},   studied System \eqref{Eq Xu 2006)} under the feedback control law $ u_t(1,t) = w(t) $ provided that the weight of the feedback with delay is a real  $\beta$ and that of the feedback without delay is a real $\alpha$. They found a feedback control law that stabilizes exponentially the system for any $|\alpha| \neq |\beta|$, by modifying the velocity feedback into the form $u(t) = \beta w_t(1,t) + \alpha f (w(.,t),w_t(.,t))$, where $f$ is a linear functional. Finally, in 2017, Xu {\it et al.} in \cite{Xie2017}, studied the stability problem of a one dimensional wave equation with internal control and boundary delay term
\begin{equation*}\label{Eq Xu 2017)}
\left\{
\begin{array}{ll}
 {u}_{tt}(x,t) -  u_{xx}(x,t) + 2 \alpha u_t(x,t)= 0, & (x,t)\in  (0,1) \times (0,+\infty), \nline
    {u}(0,t)=0, & t\in   (0,+\infty), \nline
    {u}_x(1,t)=\kappa u_t(1,t-\tau), & t\in   (0,+\infty), \nline
    \left({u}(x,0),{u}_t(x,0)\right)= \left(u_0(x),u_1(x)\right),   & x\in  (0,1),\nline
 \displaystyle{u_t(1,t) = f_0(1,t) },&\displaystyle{t\in (-\tau,0) },
\end{array}
\right.
\end{equation*}
where $\tau>0$, $\alpha > 0$ and $\kappa$ is real. Based on the idea of Lyapunov  functional, they proved exponential stability of the above system under a certain relationship between $\alpha$ and $\kappa$.  
 \\[0.1in]
\indent Going to the multidimensional case, the stability of wave equation with time delay has been studied in \cite{Nicaise2006,ammari2010,pignotti2012,messaoudi2016,Nicaise-Pignotti16,Anikushyn2018,ammari2017,ammari2018}. In 2006, Nicaise and Pignotti in \cite{Nicaise2006} studied the multidimensional wave equation considering two cases. The first case concerns a wave equation with boundary feedback and a delay term at the boundary 
\begin{equation}\label{EQ Nicaise 2006-1}
\left\{
\begin{array}{ll}
 {u}_{tt}(x,t) -  \Delta u (x,t) = 0, & (x,t)\in  \Omega \times (0,+\infty), \nline
    {u}(x,t)=0, & (x,t) \in  \Gamma_D \times (0,+\infty), \nline
     \frac{\partial u}{\partial \nu} (x,t) = -\mu_1 u_t(x,t)-\mu_2 u_t(x,t-\tau), & (x,t) \in  \Gamma_N \times (0,+\infty), \nline
    \left({u}(x,0),{u}_t(x,0)\right)= \left(u_0(x),u_1(x)\right),   & x\in  \Omega,\nline
 \displaystyle{u_t(x,t) = f_0(x,t) },&\displaystyle{(x,t) \in \Gamma_N \times (-\tau, 0) }.
\end{array}
\right.
\end{equation}
The second case concerns a wave equation with an internal feedback and a delayed velocity term ({\it i.e} an internal delay) and a mixed Dirichlet-Neumann boundary condition 
\begin{equation}\label{EQ Nicaise 2006-2}
\left\{
\begin{array}{ll}
 {u}_{tt}(x,t) -  \Delta u (x,t) + \mu_1 u_t (x,t) + \mu_2 u_t (x, t - \tau) = 0, & (x,t)\in  \Omega \times (0,+\infty), \nline
    {u}(x,t)=0, & (x,t) \in  \Gamma_D \times (0,+\infty), \nline
     \frac{\partial u}{\partial \nu} (x,t)= 0, & (x,t) \in  \Gamma_N \times (0,+\infty), \nline
    \left({u}(x,0),{u}_t(x,0)\right)= \left(u_0(x),u_1(x)\right),   & x\in  \Omega,\nline
 \displaystyle{u_t(x,t) = f_0(x,t) },&\displaystyle{(x,t) \in \Omega \times (-\tau,0) }. 
\end{array}
\right.
\end{equation}
In both systems, $\tau, \ \mu_1, \ \mu_2$ are strictly positive constants, $\partial u/\partial \nu$ is the partial derivative, $\Omega$ is an open bounded domain of $\R^N$  with a boundary $\Gamma$ of class $C^2$ and $\Gamma = \Gamma _ D \cup \Gamma_N$, such that $\Gamma _ D \cap \Gamma_N = \emptyset$.  Under the assumption that the weight of the feedback with delay is smaller than that without delay $(\mu_2 < \mu_1)$, they obtained an exponential decay of the energy of both Systems \eqref{EQ Nicaise 2006-1} and \eqref{EQ Nicaise 2006-2}. On the contrary, if the previous assumption does not hold $({\it i.e}\ \mu_2\geq\mu_1)$, they found a sequence of delays for which the energy of some solutions does not tend to zero (see \cite{benhassi2009} for the treatment of Problem \eqref{EQ Nicaise 2006-2} in more general abstract form). In 2009, Nicaise {\it et al.} in \cite{Nicaise2009} studied System \eqref{EQ Nicaise 2006-1} in the one dimensional case where the delay time $\tau$ is a function depending on time  and they established an exponential stability result under the condition that the derivative of the decay function is upper bounded by a constant $d<1$ and assuming that $\mu_2 < \sqrt{1 - d} \ \mu_1$. In 2010, Ammari {\it et al.} in \cite{ammari2010} studied the wave equation with interior delay damping and dissipative undelayed boundary condition in an open domain $\Omega$ of $\R^N,\, N \geq 2.$ The system is given by the following:
\begin{equation}\label{EQ Ammari 2010}
\left\{
\begin{array}{ll}
 {u}_{tt}(x,t) -  \Delta u (x,t) + a u_t (x,t - \tau)=0 , & (x,t)\in  \Omega \times (0,+\infty), \nline
    {u}(x,t)=0, & (x,t) \in  \Gamma_0\times (0,+\infty), \nline
     \frac{\partial u}{\partial \nu} (x,t)=  - \kappa u_t(x,t), & (x,t) \in  \Gamma_1\times  (0,+\infty), \nline
    \left({u}(x,0),{u}_t(x,0)\right)= \left(u_0(x),u_1(x)\right),   & x\in  \Omega,\nline
 \displaystyle{u_t(x,t) = f_0(x,t) },&\displaystyle{(x,t) \in \Omega \times (-\tau,0) }, 
\end{array}
\right.
\end{equation}
where $\tau > 0$, $a>0$ and $ \kappa > 0$. Under the condition that $\Gamma_1$ satisfies the $\Gamma$-condition introduced in \cite{Lions88}, they proved that System  \eqref{EQ Ammari 2010} is uniformly asymptotically stable whenever the delay coefficient is sufficiently small. In 2012, Pignotti in \cite{pignotti2012} considered the wave equation with internal distributed time delay and local damping in a bounded and smooth domain $\Omega \subset \R^N, N \geq 1$.  The system is given by the following:
 \begin{equation}\label{EQ pignotti 2012}
\left\{
\begin{array}{ll}
 {u}_{tt}(x,t) -  \Delta u (x,t) + a \chi_{\omega} u_t(x,t) + \kappa u_t (x , t - \tau) = 0, & (x,t)\in  \Omega \times (0,+\infty), \nline
    {u}(x,t)=0, & (x,t) \in  \Gamma \times (0,+\infty), \nline
      \left({u}(x,0),{u}_t(x,0)\right)= \left(u_0(x),u_1(x)\right), & x\in  \Omega,\nline
 \displaystyle{u_t(x,t) = f(x,t) },&\displaystyle{(x,t) \in \Omega \times (-\tau,0) }, 
\end{array}
\right.
\end{equation}
where $\kappa $ real, $\tau>0$ and $a>0$. System \eqref{EQ pignotti 2012} shows that the damping is localized, indeed, it acts on a neighborhood of a part of the boundary of $\Omega$. Under the assumption that $|\kappa| < \kappa_0 < a,$ the author established an exponential decay rate. Later, in 2016, Messaoudi {\it et al.} in \cite{messaoudi2016} considered the stabilization of the following wave equation with strong time delay
\begin{equation*}\label{EQ messaoudi 2016}
\left\{
\begin{array}{ll}
 {u}_{tt}(x,t) -  \Delta u (x,t) - \mu_1 \Delta u_t (x,t) - \mu_2 \Delta u_t (x, t - \tau) = 0, & (x,t)\in  \Omega \times (0,+\infty), \nline
    {u}(x,t)=0, & (x,t) \in  \Gamma \times (0,+\infty), \nline
    \left({u}(x,0),{u}_t(x,0)\right)= \left(u_0(x),u_1(x)\right),   & x\in  \Omega,\nline
 \displaystyle{u_t(x,t) = f_0(x,t) },&\displaystyle{(x,t) \in \Omega \times (-\tau,0) },
\end{array}
\right.
\end{equation*}
where $\mu_1>0$ and $\mu_2$ is a non zero real number.
Under the assumption that $|\mu_2| < \mu_1 $, they obtained an exponential stability result. In addition, in the same year, Nicaise {\it et al.} in \cite{Nicaise-Pignotti16} studied the multidimensional wave equation with localized Kelvin-Voigt damping and mixed boundary condition with time delay
\begin{equation}\label{EQ Nicaise 2016}
\left\{
\begin{array}{ll}
 {u}_{tt}(x,t) -  \Delta u (x,t) - {\textrm{div}} (a(x) \nabla u_t) = 0, & (x,t)\in  \Omega \times (0,+\infty), \nline
    {u}(x,t)=0, & (x,t) \in  \Gamma_0 \times (0,+\infty), \nline
     \frac{\partial u}{\partial \nu} (x,t) = -a(x) \frac{\partial u_t}{\partial \nu} (x,t)-\kappa u_t(x,t-\tau), & (x,t) \in  \Gamma_1 \times (0,+\infty), \nline
    \left({u}(x,0),{u}_t(x,0)\right)= \left(u_0(x),u_1(x)\right),   & x\in  \Omega,\nline
 \displaystyle{u_t(x,t) = f_0(x,t) },&\displaystyle{(x,t) \in \Gamma_1 \times (-\tau, 0) },
\end{array}
\right.
\end{equation}
where $\tau > 0$, $\kappa$ is real, $a(x) \in L^{\infty}(\Omega)$ and $a(x) \geq a_0 > 0$ on $\omega$ such that $\omega \subset \Omega$ is an open neighborhood of $\Gamma_1$. Under an appropriate geometric condition on $\Gamma_1$ and assuming that $a \in C^{1,1}(\overline{\Omega})$, $\Delta a \in L^{\infty}(\Omega)$, they proved an exponential decay  of the energy of System \eqref{EQ Nicaise 2016}. Finally, in 2018, Anikushyn {\it et al.} in \cite{Anikushyn2018} considered an initial boundary value problem for a viscoelastic wave equation subjected to a strong time localized delay in a Kelvin-Voigt type. The system is given by the following:
\begin{equation*}\label{Eq Anykushyn 2018}
\left\{
\begin{array}{ll}
 {u}_{tt}(x,t) -  c_1 \Delta u (x,t) - c_2 \Delta u (x,t - \tau ) - d_1  \Delta u_t (x,t)- d_2 \Delta u_t (x,t - \tau )  = 0, & (x,t)\in  \Omega \times (0,+\infty), \nline
    {u}(x,t)=0, & (x,t) \in  \Gamma_0 \times (0,+\infty), \nline
     \frac{\partial u}{\partial \nu} (x,t) = 0, & (x,t) \in  \Gamma_1 \times (0,+\infty), \nline
    \left({u}(x,0),{u}_t(x,0)\right)= \left(u_0(x),u_1(x)\right),   & x\in  \Omega,\nline
 \displaystyle{u(x,t) = f_0(x,t) },&\displaystyle{(x,t) \in \Omega \times (-\tau, 0) }.
\end{array}
\right.
\end{equation*}
Under appropriate conditions on the coefficients, a global exponential decay rate is obtained. We can also mention that Ammari {\it et al.} in \cite{Ammari2015} considered the stabilization problem for an abstract equation with delay and a Kelvin-Voigt damping in 2015. The system is given by the following:
 \begin{equation*}\label{Abstarct Ammari 2015}
\left\{
\begin{array}{ll}
 {u}_{tt}(t) +  a \mathcal{B}\mathcal{B}^* u_t(t) + \mathcal{B}\mathcal{B}^* u(t - \tau),  & t \in   (0,+\infty), \nline
    \left({u}(0),{u}_t(0)\right)= \left(u_0,u_1 \right),   \nline
 \displaystyle{\mathcal{B}^* u(t) = f_0(t) },&\displaystyle{t \in  (-\tau,0) }, 
\end{array}
\right.
\end{equation*}
for an appropriate class of operator $\mathcal{B}$ and $a > 0.$ Using the frequency domain approach, they obtained an exponential stability result.  Finally, the transmission problem of a wave equation with global or local Kelvin-Voigt damping and without any time delay was studied by many authors in the one dimensional case (see \cite{chenLiuLiu-1998,Alves2013,Huang-1988,Alves2014,HASSINE201584,F.HASSINE2015,Portillo2017,rivera2018,Liu2016}) and in the multidimensional case (see \cite{NASSER2019272,Zhang2018,Tebou2012,Liu2006}) and polynomial and exponential stability results were obtained. In addition, the stability of wave equations on tree with local Kelvin-Voigt damping has been studied in \cite{Ammari2019}. \\[0.1in]
\indent Thus, as we confirmed in the beginning, the case  of wave equations with localized  Kelvin-Voigt type damping and boundary or internal time delay; as in our  Systems \eqref{NG-E(1.1)} and \eqref{NG-E(1.2)}, where the damping is acting in a non-smooth region is still an open problem.  The aim of the present paper consists in studying the stability of the  Systems \eqref{NG-E(1.1)} and \eqref{NG-E(1.2)}. For System  \eqref{NG-E(1.1)}, we consider two cases. Case one, if $\alpha>0$ (see Figure \ref{Fig1}), then using the semigroup theory of linear operators and a result obtained by Borichev and Tomilov, we show that the energy of the System \eqref{NG-E(1.1)} has a polynomial decay rate of type $t^{-4}$. Case two, if  $\alpha=0$ (see Figure \ref{Fig2}), then using the semigroup theory of linear operators and a result obtained by Huang and Prüss, we  prove an exponential decay  of the energy of System  \eqref{NG-E(1.1)}.  For System  \eqref{NG-E(1.2)},  by using the semigroup theory of linear operators and a result obtained by Borichev and Tomilov, we show that the energy of the System \eqref{NG-E(1.2)} has a polynomial decay rate of type $t^{-4}$.\\[0.1in]
\indent   This paper is organized as follows: In Section \ref{NG-S-2}, we study the stability of System \eqref{NG-E(1.1)}. Indeed,  in Subsection \ref{NG-S-2.1}, we consider the case $\alpha>0$. First,  we prove the well-posedness   of System \eqref{NG-E(1.1)}. Next, we prove the strong stability of the system in the lack of the compactness of the resolvent of the generator. Then, we establish a polynomial energy decay rate of type $t^{-4}$ (see Theorem \ref{NG-T-2.7}). In addition, in Subsection \ref{NG-S-2.2},  we consider the case $\alpha=0$ and we prove the exponential  stability of  system \eqref{NG-E(1.1)} (see Theorem \ref{NG-T-2.14}).  In Section \ref{NG-S-3},  we study the stability of System \eqref{NG-E(1.2)}. First,  we prove the well-posedness   of System \eqref{NG-E(1.2)}. Next,  we establish a polynomial energy decay rate of type $t^{-4}$ (see Theorem \ref{NG-T-3.2}). 
\section{Wave equation with local Kelvin-Voigt damping and with boundary delay feedback}\label{NG-S-2}
\noindent This section is devoted to our first aim, that is to study the stability of a wave equation with localized Kelvin-Voigt damping and boundary delay feedback (see System \eqref{NG-E(1.1)}). For this aim, let us introduce the auxiliary unknown
\begin{equation*}
\eta(L,\rho,t)=U_t(L,t-\rho\, \tau),\quad \rho\in(0,1),\ t>0.
\end{equation*}
Thus, Problem \eqref{NG-E(1.1)} is equivalent to
\begin{equation}\label{NG-E(2.1)}
\left\{
\begin{array}{ll}
 {U}_{tt}(x,t) -  \left[ \kappa\, {U}_x(x,t) + \delta_1 \chi_{(\alpha,\beta)}  {U}_{xt}(x,t)\right]_x = 0, & (x,t)\in  (0,L) \times (0,+\infty), \nline
 \tau \eta_t(L,\rho,t)+\eta_\rho(L,\rho,t)=0,& (\rho,t)\in (0,1)\times (0, + \infty),\nline
 
    {U}(0,t)=0, & t\in   (0,+\infty), \nline
    {U}_x(L,t)=-\delta_3 U_t(L,t)-\delta_2 \eta(L,1,t), & t\in   (0,+\infty), \nline
    \left({U}(x,0),{U}_t(x,0)\right)= \left(U_0(x),U_1(x)\right),   & x\in  (0,L),\nline
\eta(L,\rho,0)=f_0(L,-\rho\, \tau),& \rho\in (0,1).  
\end{array}
\right.
\end{equation}
Let $U$ be a smooth solution of System \eqref{NG-E(2.1)}, we associate its energy defined by
\begin{equation}\label{Energy1}
E(t)=\frac{1}{2}\int_0^L \left(|U_t|^2+\kappa |U_x|^2\right) dx+\frac{\tau}{2}\int_0^1|\eta|^2 d\rho.
\end{equation}
Multiplying the first  equation  of \eqref{NG-E(2.1)} by $U_t$, integrating over $(0,L)$ with respect to $x$, then using by parts integration and the boundary conditions in \eqref{NG-E(2.1)} at $x=0$ and at $x=L$, we get
\begin{equation}\label{NG-E(2.2)}
\frac{1}{2}\, \frac{d}{dt}\int_0^L  \left(|U_t|^2+\kappa |U_x|^2\right) dx=-\delta_1\int_\alpha^\beta   |U_{xt}|^2 dx-\kappa_3\delta_3 |U_t(L,t)|^2-\kappa_3\delta_2 \eta(L,1,t)\, U_{t}(L,t).
\end{equation}
Multiplying the second  equation  of \eqref{NG-E(2.1)} by $\eta$,  integrating  over $(0,1)$ with respect to $\rho$, then using the fact that $\eta(L,0,t)=U_t(L,t)$, we get
\begin{equation}\label{NG-E(2.3)}
\frac{\tau}{2}\, \frac{d}{dt}\int_0^1|\eta|^2 d\rho =-\frac{1}{2}|\eta(L,1,t)|^2+\frac{1}{2}|U_t(L,t)|^2.
\end{equation}
Adding \eqref{NG-E(2.2)} and \eqref{NG-E(2.3)}, we get
\begin{equation}\label{NG-E(2.4)}
\frac{d\,E(t)}{dt}=-\delta_1\int_\alpha^\beta   |U_{xt}|^2 dx-\left(\kappa_3\delta_3-\frac{1}{2}\right) |U_t(L,t)|^2-\kappa_3\delta_2 \eta(L,1,t)\, U_{t}(L,t)-\frac{1}{2}|\eta(L,1,t)|^2.
\end{equation}
For all $p>0$, we have 
\begin{equation}\label{NG-E(2.5)}
-\kappa_3\delta_2 \eta(L,1,t) U_t(L,t)\leq \frac{\kappa_3|\delta_2|\, |\eta(L,1,t)|^2}{2p}+\frac{\kappa_3|\delta_2|\, p\, |U_t(L,t)|^2}{2}.
\end{equation}
Inserting \eqref{NG-E(2.5)} in \eqref{NG-E(2.4)}, we get
\begin{equation}\label{NG-E(2.6)}
\frac{d\, E(t)}{dt} \leq -\delta_1\int_\alpha^\beta  |U_{xt}|^2 dx-\left(\frac{1}{2}-\frac{\kappa_3|\delta_2|}{2p}\right)|\eta(L,1,t)|^2
-\left(\kappa_3\delta_3-\frac{1}{2}-\frac{\kappa_3|\delta_2|\, p}{2}\right)|U_t(L,t)|^2.
\end{equation}
In the sequel, the assumption on $\kappa_3,\ \delta_1,\ \delta_2$ and $\delta_3$  will ensure that 
\begin{equation}\tag{H}
\kappa_3>0,\quad \delta_1>0,\quad \delta_3>0,\quad \delta_2\in\mathbb{R}^*,\quad \delta_3>\frac{1}{2\kappa_3},\quad  |\delta_2|<\frac{1}{\kappa_3}\sqrt{2\kappa_3\delta_3-1}.
\end{equation}
In this case, we easily check that there exists a strictly positive number $p$ satisfying
\begin{equation}\label{NG-E(2.7)}
\kappa_3 |\delta_2|<p<\frac{2}{\kappa_3|\delta_2|}\left(\kappa_3\delta_3-\frac{1}{2}\right),
\end{equation}
such that 
\begin{equation*}
\frac{1}{2}-\frac{\kappa_3|\delta_2|}{2p}>0\ \ \ \text{and}\ \ \ \kappa_3\delta_3-\frac{1}{2}-\frac{\kappa_3|\delta_2|\, p}{2}>0,
\end{equation*}
so that the energies of the strong solutions satisfy $E'(t)\leq0.$ Hence, System \eqref{NG-E(2.1)}  is dissipative in the sense that its energy is non increasing with respect to the time $t.$\\[0.1in]
For studying the stability of System \eqref{NG-E(2.1)}, we consider two cases. In Subsection \ref{NG-S-2.1}, we consider the first case, when the Kelvin-Voigt damping is localized in the internal of the body, {\it i.e.} $\alpha>0$. While in Subsection \ref{NG-S-2.2}, we consider the second, when the Kelvin-Voigt damping is localized near the boundary of the body, {\it i.e.} $\alpha=0$.
\subsection{Wave equation with  local Kelvin-Voigt damping far from the boundary  and with boundary delay feedback}\label{NG-S-2.1} 
\noindent In this subsection, we assume that there exist $\alpha$ and $\beta$ such that    $0<\alpha<\beta<L$, in this case, the Kelvin-Voigt damping  is localized in the  internal of the body (see Figure \ref{Fig1}).  For this aim, we denote the longitudinal displacement by $U$ and this displacement is divided into three parts
\begin{equation*}
U(x,t)= \left\{
        \begin{array}{ll}
           u(x,t), &  (x,t) \in \ (0,\alpha)\times (0, + \infty), \nline
            v(x,t), & (x,t) \in \ (\alpha,\beta)\times (0, + \infty) , \nline
            w(x,t), &  (x,t) \in \ (\beta,L)\times (0, + \infty).
        \end{array}
    \right.
\end{equation*}
In this case, System \eqref{NG-E(2.1)} is equivalent to the following system
\begin{eqnarray}
 u_{tt} -\kappa_1  u_{xx} = 0, &(x,t) \in  (0,\alpha) \times (0, + \infty), \label{NG-E(2.8)}\nline
 v_{tt} -\left(\kappa_2  v_{x} +\delta_1  v_{xt}\right)_x = 0, &(x,t) \in (\alpha, \beta ) \times (0, + \infty),\label{NG-E(2.9)}\nline
 w_{tt} -\kappa_3  w_{xx} = 0, &(x,t) \in  ( \beta, L )  \times (0, + \infty), \label{NG-E(2.10)}\nline
\tau \eta_t(L,\rho,t)+\eta_\rho(L,\rho,t)=0,& (\rho,t)\in (0,1)\times (0, + \infty), \label{NG-E(2.11)}
\end{eqnarray}
with the following boundary and transmission conditions
\begin{eqnarray}
u(0,t) = 0, & t\in (0,+\infty),\label{NG-E(2.12)}\nline
w_x(L,t) = - \delta_3 w_{t}(L,t) - \delta_2 \eta(L,1,t),& t\in (0, + \infty),
\label{NG-E(2.13)}\nline
u(\alpha,t ) = v(\alpha,t), &t\in  (0, + \infty), \label{NG-E(2.14)} \nline
v (\beta , t) = w (\beta , t), &t\in  (0, + \infty),\label{NG-E(2.15)} \nline
\kappa_2 v_x (\alpha , t) +\delta_1  v_{xt} (\alpha,t ) = \kappa_1 u_x (\alpha ,t), &t\in  (0, + \infty),\label{NG-E(2.16)}\nline
\kappa_2 v_x (\beta , t) + \delta_1 v_{xt} (\beta,t ) =\kappa_3 w_x (\beta ,t), &t\in  (0, + \infty),\label{NG-E(2.17)}
\end{eqnarray}
and with the following initial conditions 
\begin{eqnarray}
\left(u(x, 0 ),u_t(x,0)  \right) =\left(u_0(x), u_1 (x)\right), \quad & x\in  (0,\alpha),\label{NG-E(2.18)}
\nline
 \left(v(x, 0 ),v_t(x,0)  \right) =\left(v_0(x), v_1 (x)\right), \quad &x\in  (\alpha,\beta),\label{NG-E(2.19)}
 \nline
\left(w(x, 0 ),w_t(x,0)  \right) =\left(w_0(x), w_1 (x)\right), \quad &x\in   (\beta,L),\label{NG-E(2.20)}\nline
\eta(L,\rho,0)=f_0(L,-\rho\, \tau),& \rho\in (0,1),\label{NG-E(2.21)}
\end{eqnarray}
where the initial data $(u_0,u_1,v_0,v_1,w_0,w_1,f_0)$ belongs to a suitable Hilbert space. So, using \eqref{Energy1}, the energy of System \eqref{NG-E(2.8)}-\eqref{NG-E(2.21)} is given by
\begin{equation*}
E(t)=\frac{1}{2}\int_0^\alpha  \left(|u_t|^2+\kappa_1|u_x|^2\right) dx+\frac{1}{2}\int_\alpha^\beta  \left(|v_t|^2+\kappa_2|v_x|^2\right) dx+\frac{1}{2}\int_\beta^L \left(|w_t|^2+\kappa_3|w_x|^2\right) dx+\frac{\tau}{2}\int_0^1|\eta|^2 d\rho.
\end{equation*}
Similar to \eqref{NG-E(2.4)} and \eqref{NG-E(2.6)}, we get
\begin{equation*}
\begin{array}{lll}
\displaystyle{\frac{d\, E(t)}{dt}}
&=\displaystyle{-\delta_1\int_\alpha^\beta  |v_{xt}|^2 dx-\frac{1}{2}|\eta(L,1,t)|^2
-\kappa_3\delta_2 \eta(L,1,t) w_t(L,t)-\left(\kappa_3\delta_3-\frac{1}{2}\right)|w_t(L,t)|^2,}
\nline&\leq 
\displaystyle{-\delta_1\int_\alpha^\beta  |v_{xt}|^2 dx-\left(\frac{1}{2}-\frac{\kappa_3|\delta_2|}{2p}\right)|\eta(L,1,t)|^2
-\left(\kappa_3\delta_3-\frac{1}{2}-\frac{\kappa_3|\delta_2|\, p}{2}\right)|w_t(L,t)|^2,}

\end{array}
\end{equation*}
where $p$ is defined in \eqref{NG-E(2.7)}. Thus, under hypothesis {\rm(H)}, the System \eqref{NG-E(2.8)}-\eqref{NG-E(2.21)}  is dissipative in the sense that its energy is non increasing with respect to the time
$t.$ 
Now, we are in position to prove the existence and uniqueness of the solution of our system.
\subsubsection{Well-posedness of the problem}\label{NG-S-2.1.1}
We start this part by formulating System \eqref{NG-E(2.8)}-\eqref{NG-E(2.21)} as an abstract Cauchy problem. For this aim, let us define
\begin{equation*}
\begin{array}{ll}
\displaystyle{ \mathbb{H}^m =  H^m (0,\alpha) \times H^m (\alpha, \beta) \times H^m (\beta, L), \quad m = 1 , 2, }\nline 
\displaystyle{\mathbb{L}^2 =  L^2 (0,\alpha) \times L^2 (\alpha, \beta) \times L^2(\beta,L)},\nline
\displaystyle{  \mathbb{H}^1_L = \{(u,v,w) \in \mathbb{H}^1\ | \ u (0) = 0, \ u(\alpha) = v (\alpha), \ v(\beta) = w(\beta) \}.}
\end{array}
\end{equation*}
\begin{Remark}\label{NG-R-2.100} 
The Hilbert space $\mathbb{L}^2$ is equipped  with the norm:
\begin{equation*}
\left\|(u,v,w)\right\|^2_{\mathbb{L}^2}= \int_0^\alpha |u|^2dx+ \int_\alpha^\beta |v|^2dx+\int_\beta^L |w|^2dx.
\end{equation*}
Also, it is easy to check that the space $\mathbb{H}^1_L$ is Hilbert space over $\mathbb{C}$ equipped  with the norm:
\begin{equation*}
\left\|(u,v,w)\right\|^2_{\mathbb{H}^1_L}=\kappa_1 \int_0^\alpha |u_x|^2dx+\kappa_2 \int_\alpha^\beta |v_x|^2dx+\kappa_3 \int_\beta^L |w_x|^2dx.
\end{equation*}
Moreover, by Poincaré inequality we can easily  verify that there exists $C>0$, such that
\begin{equation*}
\left\|(u,v,w)\right\|_{\mathbb{L}^2}\leq C \left\|(u,v,w)\right\|_{\mathbb{H}^1_L}.
\end{equation*}
\xqed{$\square$}
\end{Remark}
\noindent We now define the Hilbert energy space ${\mathcal{H}}$  by
\begin{equation*}
{\mathcal{H}} = \mathbb{H}^1_L  \times \mathbb{L}^2 \times L^2(0,1)
\end{equation*}
\noindent equipped with the following inner product
\begin{equation*}
\langle \mathbb{U},\tilde{\mathbb{U}}\rangle _{{\mathcal{H}}} = \kappa_1 \int_{0}^{\alpha} u_x \overline{\tilde{u}}_x dx + \kappa_2 \int_{\alpha}^{\beta} v_x \overline{\tilde{v}}_x dx + \kappa_3 \int_{\beta}^{L} w_x \overline{\tilde{w}}_x dx+
   \int_{0}^{\alpha} y \overline{\tilde{y}} dx +  \int_{\alpha}^{\beta} z \overline{\tilde{z}} dx +  \int_{\beta}^{L} \phi \overline{\tilde{\phi}} dx +\tau \int_{0}^{1} \eta(L,\rho)\, \overline{\tilde{\eta}}(L,\rho) d \rho,
\end{equation*}
where $\mathbb{U}=( u,v,w,y ,z, \phi, \eta(L,\cdot) )\in{\mathcal{H}}$ and  $\tilde{\mathbb{U}}= ( \tilde{u}, \tilde{ v}, \tilde{w}, \tilde{y}, \tilde{z}, \tilde{\phi}, \tilde{\eta}(L,\cdot)) \in{\mathcal{H}}$. We use  $\|\mathbb{U}\|_{\mathcal{H}}$ to denote the corresponding norm. We  define the linear unbounded operator $\mathcal{A}: D(\mathcal{A}) \subset {\mathcal{H}}\longrightarrow {\mathcal{H}} $ by:
\begin{equation*}
\begin{array}{lll}
\displaystyle{D(\mathcal{A}) = \bigg\{(u,v,w,y,z,\phi, \eta(L,\cdot))\in {\mathbb{H}_L^1 \times \mathbb{H}_L^1 \times H^1 (0,1) }\ |  \ (u , \kappa_2 v  +\delta_1 z , w) \in \mathbb{H}^2,} 
\\ \noalign{\medskip}\hspace{1.9cm}
\displaystyle{ \kappa_2 v_x (\alpha) +\delta_1 z_x (\alpha)=\kappa_1 u _x (\alpha), \ \kappa_2 v_x (\beta) +\delta_1  z_x (\beta)= \kappa_3 w _x (\beta),}
 \\ \noalign{\medskip}\hspace{4.8cm}
\displaystyle{ w_x (L) = - \delta_3 \eta(L,0) - \delta_2 \eta (L,1), \ \phi(L) = \eta(L,0)\bigg\}}
\end{array}
\end{equation*}
and for all $\mathbb{U}=(u,v,w, y,z, \phi, \eta(L,\cdot))\in D(\mathcal{A})$
\begin{equation*}
\mathcal{A}\mathbb{U} = \left( y , z ,\phi, \kappa_1 u_{xx}, (\kappa_2  v_{x} +\delta_1 z_{x})_x , \kappa_3 w_{xx} , - \tau^{-1} \eta_{\rho} (L,\cdot)\right).
\end{equation*}
If $\mathbb{U}=(u,v,w,u_t,v_t,w_t, \eta(L,\cdot))$ is a regular solution of System \eqref{NG-E(2.8)}-\eqref{NG-E(2.21)}, then we transform this system into the following initial value problem
\begin{equation} \label{NG-E(2.22)}
\begin{cases}
\mathbb{U}_t &= \mathcal{A} \mathbb{U},\\
\mathbb{U}(0)&= \mathbb{U}_0,
\end{cases}
\end{equation}
where $\mathbb{U}_0 = (u_0, v_0 , w_0 , u_1, v_1, w_1, f_0(L,-\cdot \tau))\in{\mathcal{H}}.$ We now use  semigroup approach to establish well-posedness result for the  System \eqref{NG-E(2.8)}-\eqref{NG-E(2.21)}. According to Lumer-Phillips theorem (see \cite{Pazy01}), we need to prove that the operator $\mathcal{A}$ is m-dissipative in ${\mathcal{H}}$. Therefore, we prove the following proposition.
\begin{Proposition}\label{NG-R-2.1} 
Under hypothesis {\rm(H)}, the unbounded linear operator $\mathcal{A}$ is m-dissipative in the energy space ${\mathcal{H}}$.
\end{Proposition}
\begin{proof}
For all $\mathbb{U} = (u,v,w,y,z, \phi,\eta(L,\cdot)) \in D(\mathcal{A}),$ we have
\begin{equation*}
\begin{array}{lll}
\displaystyle{\text{Re}\left<\mathcal{A} \mathbb{U},\mathbb{U}\right>_{\mathcal{H}} =
\kappa_1\text{Re}\int_{0}^{\alpha} \left(y_x \overline{{u}}_x+u_{xx} \overline{{y}}\right) dx +  \text{Re}\int_{\alpha}^{\beta} \left(\kappa_2 z_x \overline{{v}}_x+ (\kappa_2 v_x+\delta_1 z_x)_{x} \overline{{z}}\right) dx 
 }\nline

   \hspace{1.9cm}
 \displaystyle{  + \kappa_3 \text{Re}\int_{\beta}^{L} \left(\phi_x \overline{{w}}_x+ w_{xx}\overline{{\phi}}\right) dx - \text{Re}\int_{0}^{1} \eta_\rho(L,\rho)\, \overline{{\eta}}(L,\rho) d \rho.

 }
\end{array}
\end{equation*}
Here $\text{Re}$ is used to denote the real part of a complex number. Using by parts integration in the above equation, we get 
\begin{equation} \label{NG-E(2.23)}
\begin{array}{lll}
\displaystyle{\text{Re}\left<\mathcal{A} \mathbb{U},\mathbb{U}\right>_{\mathcal{H}} =
-\delta_1\int_\alpha^\beta|z_x|^2dx-\frac{1}{2}|\eta(L,1)|^2+\frac{1}{2}|\eta(L,0)|^2+\kappa_3\text{Re}\left(w_x(L)\overline{\phi}(L)\right)}\nline

   \hspace{2.4cm}
 \displaystyle{  -\kappa_1\text{Re}\left(u_x(0)\overline{y}(0)\right)
  +\text{Re}\left(\kappa_1 u_x(\alpha)\overline{y}(\alpha)-\kappa_2 v_x(\alpha)\overline{z}(\alpha)-\delta_1z_x(\alpha)\overline{z}(\alpha)\right) }\nline 
   \hspace{2.4cm}

 \displaystyle{
+ \text{Re}\left(\kappa_2v_x(\beta)\overline{z}(\beta)+\delta_1 z_x(\beta)\overline{z}(\beta)-\kappa_3 w_x(\beta)\overline{\phi}(\beta)\right).

 }
\end{array}
\end{equation}
On the other hand, since $\mathbb{U}\in D(\mathcal{A})$, we have
\begin{equation}\label{NG-E(2.24)}
\left\{
\begin{array}{lll}
\displaystyle{y(0)=0,\ y(\alpha)=z(\alpha),\ z(\beta)=\phi(\beta),}\nline
  \displaystyle{ \kappa_1 u_x(\alpha)-\kappa_2 v_x(\alpha)-\delta_1z_x(\alpha)=0,\ \kappa_2v_x(\beta)+\delta_1 z_x(\beta)-\kappa_3w_x(\beta)=0,}\nline
  \displaystyle{w_x (L) = - \delta_3 \eta(L,0) - \delta_2 \eta (L,1), \ \phi(L) = \eta(L,0).}

\end{array}
\right.
\end{equation}
Inserting \eqref{NG-E(2.24)} in \eqref{NG-E(2.23)}, we get
\begin{equation}\label{NG-E(2.25)}
\text{Re}\left<\mathcal{A} \mathbb{U},\mathbb{U}\right>_{\mathcal{H}} =
-\delta_1\int_\alpha^\beta|z_x|^2dx-\frac{1}{2}|\eta(L,1)|^2-\left(\kappa_3\delta_3-\frac{1}{2}\right)|\eta(L,0)|^2-\kappa_3\delta_2\text{Re} \left(\eta(L,0)\overline{\eta}(L,1)\right).
\end{equation}
Under hypothesis {\rm(H)}, we easily check that there exists  $p>0$ such that  
\begin{equation*}
\kappa_3|\delta_2| <p<\frac{2}{\kappa_3|\delta_2|}\left(\kappa_3\delta_3-\frac{1}{2}\right).
\end{equation*}
By Young's inequality, we get 
\begin{equation*}
-\kappa_3\delta_2\text{Re} \left(\eta(L,0)\overline{\eta}(L,1)\right)\leq \frac{\kappa_3|\delta_2|\, |\eta(L,1)|^2}{2p}+\frac{\kappa_3|\delta_2|\, p\, |\eta(L,0)|^2}{2}.
\end{equation*}
Inserting the above inequality in \eqref{NG-E(2.25)}, we get
\begin{equation}\label{NG-E(2.26)}
\text{Re}\left<\mathcal{A} \mathbb{U},\mathbb{U}\right>_{\mathcal{H}}  \leq -\delta_1\int_\alpha^\beta  |z_{x}|^2 dx-\left(\frac{1}{2}-\frac{\kappa_3|\delta_2|}{2p}\right)|\eta(L,1)|^2
-\left(\kappa_3\delta_3-\frac{1}{2}-\frac{\kappa_3|\delta_2|\, p}{2}\right)|\eta(L,0)|^2.
\end{equation}
From the construction of $p$, we have
\begin{equation*}
\frac{1}{2}-\frac{\kappa_3|\delta_2|}{2p}>0 \ \ \ \text{and}\ \ \ \kappa_3\delta_3-\frac{1}{2}-\frac{\kappa_3|\delta_2|\, p}{2}>0.
\end{equation*}
Therefore, from \eqref{NG-E(2.26)}, we get 
\begin{equation*}
\text{Re}\left<\mathcal{A} \mathbb{U},\mathbb{U}\right>_{\mathcal{H}}  \leq 0,
\end{equation*}
which implies that $\mathcal{A}$ is dissipative. Now, let us go on with maximality. Let $F= (f_1,f_2,f_3,f_4,f_5,f_6,f_7(L,\cdot)) \in {\mathcal{H}}$ we look for $U= (u,v,w, y,z , \phi,\eta(L,\cdot)) \in D(\mathcal{A})$ solution of the equation
\begin{equation}\label{0dra}
-\mathcal{A}U=F.
\end{equation}
\noindent Equivalently, we consider the following system
\begin{eqnarray}
-y &=& f_1, \label{NG-E(2.27)}\\
-z &=& f_2, \label{NG-E(2.28)}\\
-\phi &=& f_3,\label{NG-E(2.29)}\\
- \kappa_1 u_{xx} &=& f_4, \label{NG-E(2.30)}\\
- \left(\kappa_2 v_x +\delta_1  z_x\right)_{x} &=&  f_5, \label{NG-E(2.31)}\\
-\kappa_3 w_{xx} &=&  f_6, \label{NG-E(2.32)}\\
\eta_\rho (L,\rho)&=&\tau f_7(\rho).\label{NG-E(2.33)}
\end{eqnarray}
In addition, we consider the following boundary conditions  
\begin{eqnarray}
u(0) = 0, \quad u( \alpha ) = v (\alpha), \quad v (\beta) = w(\beta),\label{NG-E(2.34)}\\
\kappa_2 v_x (\alpha) +\delta_1  z_x (\alpha)=\kappa_1  u _x (\alpha), \quad  \kappa_2 v_x (\beta) +\delta_1 z_x (\beta)=\kappa_3 w _x (\beta), \label{NG-E(2.35)}\\
w_x(L)=-\delta_3 \eta(L,0)-\delta_2 \eta(L,1),\label{NG-E(2.36)}\\
\eta(L,0)=\phi(L).\label{NG-E(2.37)}
\end{eqnarray}
From \eqref{NG-E(2.27)}-\eqref{NG-E(2.29)} and the fact that $F\in \mathcal{H}$, it is clear that $(y,z,\phi)\in \mathbb{H}^1_L$. Next, from \eqref{NG-E(2.29)}, \eqref{NG-E(2.37)} and the fact that $f_3\in H^1(\beta,L)$, we get
\begin{equation*}
\eta(L,0)=\phi(L)=-f_3(L).
\end{equation*}
From the above equation and Equation \eqref{NG-E(2.33)}, we can determine
\begin{equation*}
\eta(L,\rho)=\tau \int_0^{\rho}f_7(\xi)\, d\xi  -f_3(L).
\end{equation*}
It is clear that $\eta(L,\cdot)\in H^1(0,1)$ and $\eta(L,0)=\phi(L)=-f_3(L)$. Inserting the above equation  in \eqref{NG-E(2.36)}, then System \eqref{NG-E(2.27)}-\eqref{NG-E(2.37)} is equivalent to
\begin{eqnarray}
y = -f_1,\quad z = -f_2,\quad \phi =- f_3,\  \eta(L,\rho)=\tau \int_0^{\rho}f_7(\xi)\, d\xi  -f_3(L),\label{NG-E(2.38)}\\
-\kappa_1 u_{xx} =f_4,\label{NG-E(2.39)} \\
- \left(\kappa_2 v_x +\delta_1  z_x\right)_{x} =  f_5,\label{NG-E(2.40)}\\
-\kappa_3 w_{xx} =  f_6,\label{NG-E(2.41)}\\
u(0) = 0, \quad u( \alpha ) = v (\alpha), \quad v (\beta) = w(\beta),\label{NG-E(2.42)}\\
\kappa_2 v_x (\alpha) + \delta_1 z_x (\alpha)= \kappa_1 u _x (\alpha), \quad  \kappa_2 v_x (\beta) +\delta_1 z_x (\beta)=\kappa_3 w _x (\beta),\label{NG-E(2.43)}\\
w_x(L)=\left(\delta_3+\delta_2\right) f_3(L)-\tau\delta_2\int_0^{1}f_7(\xi)\, d\xi.\label{NG-E(2.44)}
\end{eqnarray}
Let $\left(\varphi, \psi, \theta\right)\in \mathbb{H}^1_L$.  Multiplying Equations \eqref{NG-E(2.39)}, \eqref{NG-E(2.40)}, \eqref{NG-E(2.41)} by $\overline{\varphi}$, $\overline{\psi}$, $\overline{\theta}$, integrating  over $(0,\alpha),$ $(\alpha,\beta)$ and $(\beta,L)$ respectively,  taking the sum, then using by parts integration, we get
\begin{equation}\label{NG-E(2.45)}
\begin{array}{ll}
\displaystyle{
\kappa_1\int_0^{\alpha} u_x\overline{\varphi}_xdx+\int_{\alpha}^{\beta} (\kappa_2 v_x+\delta_1 z_x)\overline{\psi}_xdx+\kappa_3\int_{\beta}^L
w_x\overline{\theta}_xdx+\kappa_1u_x(0)\overline{\varphi}(0)}\nline

\displaystyle{-\kappa_1 u_x(\alpha)\overline{\varphi}(\alpha)
+(\kappa_2v_x(\alpha)+\delta_1z_x(\alpha))\overline{\psi}(\alpha)
-(\kappa_2v_x(\beta)+\delta_1z_x(\beta))\overline{\psi}(\beta)+\kappa_3 w_x(\beta)\overline{\theta}(\beta)
}\nline
\displaystyle{
=\int_0^{\alpha} f_4\overline{\varphi}dx+\int_{\alpha}^{\beta} f_5\overline{\psi}dx+\int_{\beta}^L
f_6\overline{\theta}dx+\kappa_3w_x(L)\overline{\theta}(L).}
\end{array}
\end{equation}
From the fact that $\left(\varphi, \psi, \theta\right)\in \mathbb{H}^1_L,$ we have
\begin{equation*}
\varphi(0)=0,\quad \varphi(\alpha)=\psi(\alpha),\quad \theta(\beta)=\psi(\beta).
\end{equation*}
Inserting the above equation  in \eqref{NG-E(2.45)}, then  using \eqref{NG-E(2.38)} and \eqref{NG-E(2.42)}-\eqref{NG-E(2.44)}, we get
\begin{equation}\label{NG-E(2.46)}
\begin{array}{ll}
\displaystyle{
\kappa_1\int_0^{\alpha} u_x\overline{\varphi}_xdx+\kappa_2\int_{\alpha}^{\beta} v_x\overline{\psi}_xdx+\kappa_3\int_{\beta}^L
w_x\overline{\theta}_xdx=\int_0^{\alpha} f_4\overline{\varphi}dx}\nline

\displaystyle{
+\int_{\alpha}^{\beta} \left(\delta_1(f_2)_x\overline{\psi}_x+f_5\overline{\psi}\right)dx+\int_{\beta}^L
f_6\overline{\theta}dx+\kappa_3\left(\left(\delta_3+\delta_2\right) f_3(L)-\tau\delta_2\int_0^{1}f_7(\xi)\, d\xi\right)\overline{\theta}(L).}
\end{array}
\end{equation}
We can easily verify that the left hand side of  \eqref{NG-E(2.46)} is a bilinear continuous coercive form on $\mathbb{H}^1_L \times \mathbb{H}^1_L$,  and the right hand side of   \eqref{NG-E(2.46)} is a linear continuous form on $\mathbb{H}^1_L$. Then, using Lax-Milgram theorem, we deduce that there exists $(u,v,w) \in \mathbb{H}^1_L  $ unique solution of the variational Problem \eqref{NG-E(2.46)}. Using standard arguments, we can show that $(u,\kappa_2v+\delta_1z , w) \in \mathbb{H}^2$. Finally, by seting $y = -f_1,$ $z = -f_2,$ $\phi =- f_3$ and  $\eta(L,\rho)=\tau \int_0^{\rho}f_7(\xi)\, d\xi  -f_3(L)$ and by applying the classical elliptic regularity we deduce that $U=(u,v,w, y,z,\phi,\eta(L,\cdot))\in D({\mathcal{A}})$ is solution of Equation \eqref{0dra}. To conclude, we need to show the uniqueness of $U$. So, let $U=(u,v,w, y,z,\phi,\eta(L,\cdot))\in D({\mathcal{A}})$ be a solution of \eqref{0dra} with $F=0$, then we directly deduce that $y = z = \phi =\eta(L,\rho)=0$ and that $(u,v,w) \in \mathbb{H}^1_L  $ satisfies Problem \eqref{NG-E(2.46)} with zero in the right hand side. This implies that $u=v=w=0$,   in other words, ker $\mathcal{A} =\{0\}$ and $0$ belongs to the resolvent set $\rho(\mathcal{A})$ of $\mathcal{A}$.  Then, by contraction principale, we easily deduce that $R(\lambda I -\mathcal{A} ) = {\mathcal{H}}$ for sufficiently small $\lambda>0 $. This, together with the dissipativeness of $\mathcal{A}$, imply that   $D\left(\mathcal{A}\right)$ is dense in ${\mathcal{H}}$   and that $\mathcal{A}$ is m-dissipative in ${\mathcal{H}}$ (see Theorems 4.5, 4.6 in  \cite{Pazy01}). The proof is thus complete.
\end{proof}
\noindent Thanks to Lumer-Philips theorem (see \cite{Pazy01}), we deduce that $\mathcal{A}$ generates a $C_0-$semigroup of contractions $e^{t\mathcal{A}}$ in ${\mathcal{H}}$ and therefore Problem \eqref{NG-E(2.8)}-\eqref{NG-E(2.21)} is well-posed. Then we have the following result:
\begin{Theorem}\label{NG-T-2.3} 
Under hypothesis {\rm(H)}, for any $\mathbb{U}_0 \in {\mathcal{H}},$ Problem \eqref{NG-E(2.22)} admits a unique weak solution, $\mathbb{U}(x,\rho,t)=e^{t\mathcal{A}}\mathbb{U}_0(x,\rho)$, such that 
\begin{equation*}
\mathbb{U}\in C^0(\R^+, {\mathcal{H}}).
\end{equation*}
\noindent Moreover, if $\mathbb{U}_0 \in D(\mathcal{A}),$ then
\begin{equation*}
\mathbb{U} \in C^1 (\R^+, {\mathcal{H}}) \cap C^0(\R^+, D(\mathcal{A})).
\end{equation*}
\end{Theorem}
\subsubsection{Strong Stability}\label{NG-S-2.1.2}
Our main result in this part is the following theorem. 
\begin{Theorem} \label{NG-T-2.4} 
Under hypothesis {\rm(H)}, the $C_0-$semigroup of contractions $e^{t\mathcal{A}}$ is strongly stable on the Hilbert space ${\mathcal{H}}$ in the sense that
\begin{equation*}
 \lim\limits_{t \to + \infty} || e^{t\mathcal{A}} \mathbb{U}_0||_{{\mathcal{H}}} = 0, \qquad \qquad \forall \ \mathbb{U}_0\in {\mathcal{H}}.
\end{equation*}
\end{Theorem}
\noindent For the proof of Theorem \ref{NG-T-2.4}, according to Theorem \ref{NG-T-A.2}, we need to prove that the operator  $\mathcal{A}$ has no pure imaginary eigenvalues and $\sigma(\mathcal{A}) \cap i\R$ contains only a countable number of continuous spectrum of $\mathcal{A}$. The argument for Theorem \ref{NG-T-2.4} relies on the subsequent lemmas. 
\begin{lemma}\label{NG-L-2.5} 
Under hypothesis {\rm(H)}, for $\lambda\in\mathbb{R},$ we have $ i \lambda I - \mathcal{A}$ is injective, \it{i.e.}
\begin{equation*}
{\rm{ker}} ( i \lambda I - \mathcal{A}) = \{0\}, \qquad \forall \lambda \in \R. 
\end{equation*}
\end{lemma}
\begin{proof}
From Proposition \ref{NG-R-2.1}, we have $0 \in \rho(\mathcal{A}).$ We still need to show the result for $\lambda \in \R^*.$ Suppose that there exists a real number $\lambda \neq 0$ and $ \mathbb{U} = (u,v,w,y,z, \phi,\eta(L,\cdot)) \in D(\mathcal{A})$ such that
\begin{equation}\label{NG-E(2.47)}
\mathcal{A} \mathbb{U} = i \lambda \mathbb{U}.
\end{equation}
First, similar to Equation \eqref{NG-E(2.26)}, we have
\begin{equation*}
0= \text{Re}\left<\mathcal{A} \mathbb{U},\mathbb{U}\right>_{\mathcal{H}}  \leq -\delta_1\int_\alpha^\beta  |z_{x}|^2 dx-\left(\frac{1}{2}-\frac{\kappa_3|\delta_2|}{2p}\right)|\eta(L,1)|^2
-\left(\kappa_3\delta_3-\frac{1}{2}-\frac{\kappa_3|\delta_2|\, p}{2}\right)|\eta(L,0)|^2\leq0.
\end{equation*}
Thus, 
\begin{equation}\label{NG-E(2.48)}
z_x = 0 \quad \textrm{in} \ (\alpha,\beta)\ \ \ \text{and} \ \ \ \eta(L,1)=\eta(L,0)=0.
\end{equation}
Next, writing \eqref{NG-E(2.47)} in a detailed form gives
\begin{eqnarray}
y = i \lambda u ,&x\in  (0,\alpha), \label{NG-E(2.49)}\\
z = i \lambda v, & x\in (\alpha,\beta) , \label{NG-E(2.50)}\\
w = i \lambda \phi,& x\in   (\beta,L) , \label{NG-E(2.51)}\\
\kappa_1 u_{xx} = i  \lambda y, & x\in (0,\alpha), \label{NG-E(2.52)}\\
\left(\kappa_2 v_x + \delta_1 z_x\right)_{x} = i  \lambda z, & x\in (\alpha,\beta), \label{NG-E(2.53)}\\
\kappa_3 w_{xx} = i \lambda \phi, & x\in (\beta,L), \label{NG-E(2.54)}\\
\eta_\rho(L,\rho)= -i\lambda \tau \eta(L,\rho),&\rho\in (0,1).\label{NG-E(2.55)}
\end{eqnarray}
From \eqref{NG-E(2.55)} and \eqref{NG-E(2.48)}, we get
\begin{equation}\label{NG-E(2.56)}
\eta(L,\cdot)=\eta(L,0)e^{-i\lambda\tau \cdot}=0 \quad \textrm{in} \ (0,1).
\end{equation}
Combining \eqref{NG-E(2.48)} with \eqref{NG-E(2.50)}, we get that
\begin{equation}\label{NG-E(2.57)}
v_x=z_x = 0 \quad \textrm{in} \ (\alpha,\beta).
\end{equation}
Thus,
\begin{equation*}
v_{xx} =z_{xx}= 0 \quad \textrm{in} \ (\alpha,\beta).
\end{equation*}
Inserting the above result in \eqref{NG-E(2.53)}, then taking into consideration \eqref{NG-E(2.50)},   we obtain
\begin{equation}\label{NG-E(2.58)}
v =z= 0  \quad \textrm{in} \ (\alpha,\beta).
\end{equation}
From the definition of $D(\mathcal{A})$ and using \eqref{NG-E(2.56)}-\eqref{NG-E(2.58)}, we get
\begin{equation*}
\left\{
\begin{array}{lll}
u(\alpha) = v (\alpha)=0,\quad    w(\beta)=v (\beta) =0,\quad y(\alpha) = z (\alpha)=0,\quad  \phi (\beta)= z (\beta) =0,\nline
\kappa_1 u_x(\alpha)=\kappa_2 v_x(\alpha)+\delta_1 z_x(\alpha)=0,\quad \kappa_3 w_x(\beta)=k_2 v_x(\beta)+\delta_1 z_x(\beta)=0,\nline
w(L)=i\lambda\phi(L)=i\lambda \eta(L,0)=0, \quad w_x(L) = - \delta_3 \eta(L,0) - \delta_2 \eta(L,1) = 0.
\end{array}
\right.
\end{equation*}
Combining \eqref{NG-E(2.49)} with \eqref{NG-E(2.52)} and \eqref{NG-E(2.51)} with \eqref{NG-E(2.54)} and using the above equation  as boundary conditions,  we get
\begin{equation*}
\left\{
\begin{array}{lll}
\displaystyle{u_{xx} + \frac{\lambda^2 }{\kappa_1} u = 0, \quad\ x\in   (0,\alpha),}\nline
\displaystyle{ u(0) = u (\alpha) = u_x (\alpha) = 0,} 
\end{array}\right. \qquad
\left\{
\begin{array}{lll}
\displaystyle{w_{xx} + \frac{ \lambda^2 }{\kappa_3} w = 0, \quad x\in  (\beta,L),}\nline
\displaystyle{ w(\beta) =w(L)= w_x (\beta) = w_x (L) =0}.
\end{array}\right. 
\end{equation*}
Thus, 
\begin{equation}\label{NG-E(2.59)}
u(x) = 0 \quad \forall\,x\in   (0,\alpha)
 \ \ \ \text{and}\ \ \
w(x) = 0 \quad \forall\,x\in    (\beta,L).
\end{equation}
Combining \eqref{NG-E(2.59)} with \eqref{NG-E(2.49)} and \eqref{NG-E(2.51)}, we obtain
\begin{equation*}
y(x) = 0 \quad \forall\,x\in   (0,\alpha)
 \ \ \ \text{and}\ \ \
\phi(x) = 0 \quad \forall\,x\in    (\beta,L).
\end{equation*} 
Finally, from the above result,  \eqref{NG-E(2.56)}, \eqref{NG-E(2.58)} and \eqref{NG-E(2.59)}, we get that  $\mathbb{U} = 0.$ The proof is thus complete.
\end{proof}
\begin{lemma} \label{NG-L-2.6} 
Under hypothesis {\rm(H)}, for $\lambda\in\mathbb{R},$ we have $ i \lambda I - \mathcal{A}$ is surjective, \it{i.e.} 
\begin{equation*}
R(i \lambda I - \mathcal{A}) = {\mathcal{H}},\qquad \forall\,\lambda \in \R.
\end{equation*}
\end{lemma}
\begin{proof}
Since $0 \in \rho(\mathcal{A}),$ we still need to show the result for $\lambda \in \R^*.$ For any $F = (f_1,f_2,f_3,f_4,f_5,f_6,f_7(L,\cdot))\in{\mathcal{H}}$ and $\lambda \in \R^*,$ we prove the existence of $\mathbb{U} = (u,v,w,y,z,\phi,\eta(L,\cdot)) \in D(\mathcal{A})$ solution for the following equation
\begin{equation*}
(i \lambda I - \mathcal{A}) \mathbb{U} = F.
\end{equation*}
\noindent Equivalently, we consider the following problem 
\begin{eqnarray}
 y =i \lambda u- f_1 & \textrm{in} \quad H^1(0,\alpha) , \label{NG-E(2.60)}\\
 z =i \lambda v- f_2 & \textrm{in} \quad H^1(\alpha,\beta), \label{NG-E(2.61)}\\
 \phi =i \lambda w- f_3&\textrm{in} \quad H^1(\beta,L) , \label{NG-E(2.62)}\\
i \lambda y - \kappa_1 u_{xx} =  f_4 & \textrm{in} \quad L^2(0,\alpha) , \label{NG-E(2.63)}\\
i \lambda z -(\kappa_2 v_x + \delta_1 z_x)_{x}=  f_5 & \textrm{in} \quad L^2(\alpha,\beta), \label{NG-E(2.64)}\\
i \lambda \phi - \kappa_3 w_{xx} = f_6 & \textrm{in} \quad L^2(\beta,L) , \label{NG-E(2.65)}\\
 \eta_\rho(L,\cdot)+i\tau \lambda\eta(L,\cdot)= \tau f_7(L,\cdot) & \textrm{in} \quad L^2(0,1),\label{NG-E(2.66)}
\end{eqnarray}
with  the following boundary conditions  
\begin{eqnarray}
u(0) = 0, \quad u( \alpha ) = v (\alpha), \quad v (\beta) = w(\beta),\label{NG-E(2.67)}\\
\kappa_2 v_x (\alpha) +\delta_1  z_x (\alpha)=\kappa_1  u _x (\alpha), \quad  \kappa_2 v_x (\beta) +\delta_1 z_x (\beta)=\kappa_3 w _x (\beta), \label{NG-E(2.68)}\\
w_x(L)=-\delta_3 \eta(L,0)-\delta_2 \eta(L,1),\label{NG-E(2.69)}\\
\eta(L,0)=\phi(L).\label{NG-E(2.70)}
\end{eqnarray}
It follows from \eqref{NG-E(2.66)}, \eqref{NG-E(2.70)} and \eqref{NG-E(2.62)} that
\begin{equation}\label{NG-E(2.71)}
\eta(L,\rho)=\left(i\lambda w(L)-f_3(L)\right) e^{-i\tau\lambda\, \rho}+\tau \int_0^\rho e^{i\tau\lambda\left(\xi-\rho\right)} f_7(L,\xi)d\xi. 
\end{equation}
Inserting \eqref{NG-E(2.60)}-\eqref{NG-E(2.62)} and \eqref{NG-E(2.71)} in \eqref{NG-E(2.63)}-\eqref{NG-E(2.70)} and deriving \eqref{NG-E(2.61)} with respect to $x$, we get
\begin{eqnarray}
\displaystyle{- \lambda^2 u  -\kappa_1 u_{xx} =  i \lambda f_1+ f_4},\label{NG-E(2.72)}\nline
\displaystyle{- \lambda^2 v -\left(\kappa_2 v_x+\delta_1 z_x\right)_{x}  = i \lambda f_2+f_5},\label{NG-E(2.73)}\nline
\displaystyle{- \lambda^2 w  - \kappa_3 w_{xx} = i \lambda f_3 +f_6  },\label{NG-E(2.74)}\nline
 \displaystyle{z_x =i \lambda v_x- (f_2)_x },\label{NG-E(2.75)}\nline
\displaystyle{u(0) = 0, \quad u( \alpha ) = v(\alpha),\quad  \kappa_1  u _x (\alpha)=\kappa_2 v_x (\alpha) +\delta_1  z_x (\alpha)},\label{NG-E(2.76)}\nline
\displaystyle{ w(\beta)=v(\beta),\quad  \kappa_3 w _x (\beta)=\kappa_2 v_x (\beta) +\delta_1 z_x (\beta)},\label{NG-E(2.77)}\nline
\displaystyle{w_x(L)=-i\lambda \left(\delta_3+\delta_2 e^{-i\tau\lambda}\right) w(L)+\left(\delta_3+\delta_2 e^{-i\tau\lambda}\right) f_3(L)-\tau \delta_2\int_0^1 e^{i\tau\lambda\left(\xi-1\right)} f_7(L,\xi)d\xi
}.\label{NG-E(2.78)}
\end{eqnarray}
Let $\left(\varphi, \psi, \theta\right)\in \mathbb{H}^1_L$.  Multiplying Equations \eqref{NG-E(2.72)}, \eqref{NG-E(2.73)}, \eqref{NG-E(2.74)} by $\overline{\varphi}$, $\overline{\psi}$, $\overline{\theta}$, integrating  over $(0,\alpha),$ $(\alpha,\beta)$  and $(\beta,L)$ respectively,  taking the sum, then using by parts integration, we get
\begin{equation}\label{NG-E(2.79)}
\begin{array}{ll}
\displaystyle{
\kappa_1\int_0^{\alpha} u_x\overline{\varphi}_xdx+\int_{\alpha}^{\beta} (\kappa_2 v_x+\delta_1 z_x)\overline{\psi}_xdx+\kappa_3\int_{\beta}^L
w_x\overline{\theta}_xdx+\kappa_1u_x(0)\overline{\varphi}(0)}\nline

\displaystyle{-\kappa_1 u_x(\alpha)\overline{\varphi}(\alpha)
+(\kappa_2v_x(\alpha)+\delta_1z_x(\alpha))\overline{\psi}(\alpha)
-(\kappa_2v_x(\beta)+\delta_1z_x(\beta))\overline{\psi}(\beta)+\kappa_3 w_x(\beta)\overline{\theta}(\beta)
}\nline
\displaystyle{
-\lambda^2\int_0^\alpha u\overline{\varphi}dx-\lambda^2\int_\alpha^\beta  v\overline{\psi}dx-\lambda^2\int_\beta^L  w\overline{\theta}dx-\kappa_3w_x(L)\overline{\theta}(L)}\nline

\displaystyle{
=\int_0^{\alpha} \left(i \lambda f_1+ f_4\right)\overline{\varphi}dx+\int_{\alpha}^{\beta} \left(i \lambda f_2+ f_5\right)\overline{\psi}dx+\int_{\beta}^L
\left(i \lambda f_3+ f_6\right)\overline{\theta}dx.}
\end{array}
\end{equation}
From the fact that $\left(\varphi, \psi, \theta\right)\in \mathbb{H}^1_L,$ we have
\begin{equation*}
\varphi(0)=0,\quad \varphi(\alpha)=\psi(\alpha),\quad \theta(\beta)=\psi(\beta).
\end{equation*}
Inserting the above equation in \eqref{NG-E(2.79)}, then using \eqref{NG-E(2.75)}-\eqref{NG-E(2.78)}, we get
\begin{equation}\label{NG-E(2.80)}
a\left(\left(u, v, w\right),\left(\varphi, \psi, \theta\right)\right)=\mathrm{F}\left(\varphi, \psi, \theta\right),\quad \forall \, \left(\varphi, \psi, \theta\right)\in \mathbb{H}^1_L,
\end{equation}
where 
\begin{equation*}
\begin{array}{ll}
\displaystyle{\mathrm{F}\left(\varphi, \psi, \theta\right)
=\int_0^{\alpha} \left(i \lambda f_1+ f_4\right)\overline{\varphi}dx+\int_{\alpha}^{\beta} \left(i \lambda f_2+ f_5\right)\overline{\psi}dx+ \delta_1 \int_{\alpha}^{\beta}  (f_2)_x\overline{\psi}_x  dx}
\nline
\displaystyle{
+\int_{\beta}^L
\left(i \lambda f_3+ f_6\right)\overline{\theta}dx+
\kappa_3\left(\left(\delta_3+\delta_2 e^{-i\tau\lambda}\right) f_3(L)-\tau \delta_2\int_0^1 e^{i\tau\lambda\left(\xi-1\right)} f_7(L,\xi)d\xi\right)\overline{\theta}(L)
}

\end{array}
\end{equation*}
and
\begin{equation*}
a\left(\left(u, v, w\right),\left(\varphi, \psi, \theta\right)\right)=a_{1}\left(\left(u, v, w\right),\left(\varphi, \psi, \theta\right)\right)+a_{2}\left(\left(u, v, w\right),\left(\varphi, \psi, \theta\right)\right),
\end{equation*}
such that 
\begin{equation*}
\left\{
\begin{array}{ll}
\displaystyle{
a_{1}\left(\left(u, v, w\right),\left(\varphi, \psi, \theta\right)\right)=\kappa_1\int_0^{\alpha} u_x\overline{\varphi}_xdx+\left(\kappa_2+ i\delta_1 \lambda\right)\int_{\alpha}^{\beta}  v_x \overline{\psi}_x   dx+\kappa_3\int_{\beta}^L
w_x\overline{\theta}_xdx,}\nline

\displaystyle{
a_{2}\left(\left(u, v, w\right),\left(\varphi, \psi, \theta\right)\right)=-\lambda^2\int_0^\alpha u\overline{\varphi}dx-\lambda^2\int_\alpha^\beta  v\overline{\psi}dx-\lambda^2\int_\beta^L  w\overline{\theta}dx+i\kappa_3\lambda \left(\delta_3+\delta_2 e^{-i\tau\lambda}\right) w(L)\overline{\theta}(L)}.
\end{array}
\right.
\end{equation*}
Let $\left(\mathbb{H}^1_L\right)'$ be the dual space of $\mathbb{H}^1_L$. We define the operators $\mathrm{A}$, $\mathrm{A}_1$ and $\mathrm{A}_2$ by
\begin{equation*}
\left\{
\begin{array}{lll}
\mathrm{A}\,:\, \mathbb{H}^1_L\to \left(\mathbb{H}^1_L\right)'
\nline
(u,v,w)\to \mathrm{A}(u,v,w)
\end{array}\right.
\quad\left\{
\begin{array}{lll}
\mathrm{A}_1\,:\, \mathbb{H}^1_L\to \left(\mathbb{H}^1_L\right)'
\nline
(u,v,w)\to \mathrm{A}_1(u,v,w)
\end{array}\right.
\quad\left\{
\begin{array}{lll}
\mathrm{A}_2\,:\, \mathbb{H}^1_L\to \left(\mathbb{H}^1_L\right)'
\nline
(u,v,w)\to \mathrm{A}_2(u,v,w)
\end{array}\right.
\end{equation*}
such that 
\begin{equation}\label{NG-E(2.81)}
\left\{
\begin{array}{lll}
\displaystyle{
\left(\mathrm{A}(u,v,w)\right)\left(\varphi, \psi, \theta\right)=a\left(\left(u, v, w\right),\left(\varphi, \psi, \theta\right)\right),}&\displaystyle{\forall \, \left(\varphi, \psi, \theta\right)\in \mathbb{H}^1_L,}\nline

\displaystyle{
\left(\mathrm{A}_1(u,v,w)\right)\left(\varphi, \psi, \theta\right)=a_1\left(\left(u, v, w\right),\left(\varphi, \psi, \theta\right)\right),}&\displaystyle{ \forall \, \left(\varphi, \psi, \theta\right)\in \mathbb{H}^1_L,}\nline

\displaystyle{
\left(\mathrm{A}_2(u,v,w)\right)\left(\varphi, \psi, \theta\right)=a_2\left(\left(u, v, w\right),\left(\varphi, \psi, \theta\right)\right),}&\displaystyle{ \forall \, \left(\varphi, \psi, \theta\right)\in \mathbb{H}^1_L.}

\end{array}\right.
\end{equation}
Our aim is to prove that  the operator $\mathrm{A}$  is an isomorphism. For this aim, we proceed the proof in three steps.\\[0.1in]
\textbf{Step 1.} In this step we proof that the operator $\mathrm{A}_1$ is an isomorphism.  For this aim, according to \eqref{NG-E(2.81)}, we have
\begin{equation*}
a_1\left(\left(u, v, w\right),\left(\varphi, \psi, \theta\right)\right)=\kappa_1\int_0^{\alpha} u_x\overline{\varphi}_xdx+\left(\kappa_2+ i\delta_1 \lambda\right)\int_{\alpha}^{\beta}  v_x \overline{\psi}_x   dx+\kappa_3\int_{\beta}^L
w_x\overline{\theta}_xdx
\end{equation*}
We can easily verify that $a_1$  is a bilinear continuous coercive form on $\mathbb{H}^1_L \times \mathbb{H}^1_L$. Then, by Lax-Milgram lemma, the operator $\mathrm{A}_1$ is an isomorphism.\\[0.1in]
\textbf{Step 2.} In this step we proof that the operator $ \mathrm{A}_2$ is compact. First,  for $\frac{1}{2}<r<1$, we introduce the Hilbert space $\mathbb{H}^r_L$ by
\begin{equation*}
  \mathbb{H}^r_L = \{(\varphi,\psi,\theta) \in H^r (0,\alpha) \times H^r (\alpha, \beta) \times H^r (\beta, L)\ | \ \varphi (0) = 0, \ \varphi(\alpha) = \psi (\alpha), \ \psi(\beta) = \theta(\beta) \}.
\end{equation*}
Thus by trace theorem, there exists $C>0$,  such that
\begin{equation}\label{NG-E(2.82)}
|\theta(L)|\leq C\left\|(\varphi,\psi,\theta)\right\|_{ \mathbb{H}^r_L}.
\end{equation}
Now, according to \eqref{NG-E(2.81)}, we have
\begin{equation*}
a_2\left(\left(u, v, w\right),\left(\varphi, \psi, \theta\right)\right)=-\lambda^2\int_0^\alpha u\overline{\varphi}dx-\lambda^2\int_\alpha^\beta  v\overline{\psi}dx-\lambda^2\int_\beta^L  w\overline{\theta}dx+i\kappa_3\lambda \left(\delta_3+\delta_2 e^{-i\tau\lambda}\right) w(L)\overline{\theta}(L).
\end{equation*}
Then, by using \eqref{NG-E(2.82)}, we get
\begin{equation*}
\left|a_2\left(\left(u, v, w\right),\left(\varphi, \psi, \theta\right)\right)\right|\leq 
C_1 \left\|(u,v,w)\right\|_{\mathbb{H}^1_L} \left\|(\varphi,\psi,\theta)\right\|_{\mathbb{L}^2}+C_1 \left\|(u,v,w)\right\|_{\mathbb{H}^1_L}\left\|(\varphi,\psi,\theta)\right\|_{ \mathbb{H}^r_L},
\end{equation*}
where $C_1>0$. Therefore,  for all $r\in (\frac{1}{2},1)$ there exists $C_2>0$, such that
\begin{equation*}
\left|a_2\left(\left(u, v, w\right),\left(\varphi, \psi, \theta\right)\right)\right|\leq 
C_2 \left\|(u,v,w)\right\|_{\mathbb{H}^1_L}\left\|(\varphi,\psi,\theta)\right\|_{ \mathbb{H}^r_L},
\end{equation*}
which implies that 
\begin{equation*}
 \mathrm{A}_2\in \mathcal{L}\left({\mathbb{H}^1_L},\left( \mathbb{H}^r_L\right)'\right).
\end{equation*}
Finally, using the   compactness embedding from $\left(\mathbb{H}^r_L\right)'$ into 
$\left(\mathbb{H}^1_L\right)'$ we deduce that $ \mathrm{A}_2$ is compact.\\[0.1in]
From steps 1 and 2, we get that the operator $ \mathrm{A}=\mathrm{A}_1+\mathrm{A}_2$ is a Fredholm operator of index zero  0. Consequently, by Fredholm alternative, proving  the operator $\mathrm{A}$  is an isomorphism reduces to proving  ${\rm{ker}} ( \mathrm{A}) = \{0\}$.\\[0.1in]
\textbf{Step 3.} In this step we proof that the ${\rm{ker}} ( \mathrm{A}) = \{0\}$. For this aim, let $(\tilde{u},\tilde{v},\tilde{w})\in {\rm{ker}} ( \mathrm{A})$, {\it i.e.}
\begin{equation*}
a\left(\left(\tilde{u}, \tilde{v}, \tilde{w}\right),\left(\varphi, \psi, \theta\right)\right)=0,\quad \forall \, \left(\varphi, \psi, \theta\right)\in \mathbb{H}^1_L.
\end{equation*}
Equivalently,
\begin{equation*}
\begin{array}{ll}
\displaystyle{
\int_0^{\alpha} \left(\kappa_1\tilde{u}_x\overline{\varphi}_x-\lambda^2\tilde{u}\overline{\varphi}\right) dx+\int_{\alpha}^{\beta}  \left(\left(\kappa_2+ i\delta_1 \lambda\right)\tilde{v}_x \overline{\psi}_x-\lambda^2\tilde{v}\overline{\psi}\right)   dx

+\int_{\beta}^L\left(\kappa_3\tilde{w}_x\overline{\theta}_x-\lambda^2  \tilde{w}\overline{\theta}\right)dx}\nline

\displaystyle{
+i\kappa_3\lambda \left(\delta_3+\delta_2 e^{-i\tau\lambda}\right) \tilde{w}(L)\overline{\theta}(L)=0,\quad \forall \, \left(\varphi, \psi, \theta\right)\in \mathbb{H}^1_L.}

\end{array}
\end{equation*}
Then, we find that
\begin{equation*}
\left\{
\begin{array}{ll}
\displaystyle{- \lambda^2 \tilde{u}  -\kappa_1 \tilde{u}_{xx} =  0},\nline
\displaystyle{- \lambda^2\tilde{v} -\left(\kappa_2 +i\delta_1 \lambda\right)\tilde{v}_{xx}  =0},\nline
\displaystyle{- \lambda^2 \tilde{w}  - \kappa_3 \tilde{w}_{xx} = 0   },\nline
\displaystyle{\tilde{u}(0) = 0, \quad \tilde{u}( \alpha ) =\tilde{ v}(\alpha),\quad \kappa_1 \tilde{u} _x (\alpha)=(\kappa_2+i\delta_1\lambda)\tilde{v}_x(\alpha)},\nline
\displaystyle{ \tilde{w}(\beta)=\tilde{v}(\beta),\quad  \kappa_3 \tilde{w} _x (\beta)=(\kappa_2+i\delta_1\lambda)\tilde{v}_x(\beta)},\nline
\displaystyle{\tilde{w}_x(L)=-i\lambda \left(\delta_3+\delta_2 e^{-i\tau\lambda}\right) \tilde{w}(L)
}.

\end{array}\right.
\end{equation*}
Therefore, the vector $\tilde{V}$ define by
\begin{equation*}
\tilde{V}=\left(\tilde{u},\tilde{v},\tilde{w},i\lambda \tilde{u},i\lambda\tilde{v},i\lambda \tilde{w}, i\lambda \tilde{w}(L) e^{-i\tau\lambda\, \cdot}\right)
\end{equation*}
belongs to $D(\mathcal{A})$ and we have
\begin{equation*}
i\lambda\tilde{V}-\mathcal{A}\tilde{V}=0.
\end{equation*}
Thus,  $\tilde{V}\in {\rm{ker}} ( i \lambda I - \mathcal{A})$, therefore by Lemma \ref{NG-L-2.5} , we get $\tilde{V}=0$, this implies that $\tilde{u}=0,\ \tilde{v}=0$ and $\tilde{w}=0$,  so ${\rm{ker}} ( \mathrm{A}) = \{0\}$.\\[0.1in]
Therefore, from step 3 and  Fredholm alternative, we get that the operator $\mathrm{A}$  is an isomorphism. It easy to see that the operator $\mathrm{F}$ is continuous form on $\mathbb{H}^1_L$. Consequently,  Equation \eqref{NG-E(2.80)} admits a unique solution $(u,v,w)\in \mathbb{H}^1_L$.  Thus, using \eqref{NG-E(2.60)}-\eqref{NG-E(2.62)}, \eqref{NG-E(2.71)}  and a classical regularity arguments, we conclude that $(i \lambda I - \mathcal{A}) \mathbb{U} = F$ admits a unique solution $\mathbb{U}\in D\left(\mathcal{A}\right)$. The proof is thus complete.
\end{proof}
\noindent \textbf{Proof of Theorem \ref{NG-T-2.4}.}  Form Lemma \ref{NG-L-2.5}, we have that the operator $\mathcal{A}$ has no pure imaginary eigenvalues and by Lemma \ref{NG-L-2.6}, $R(i \lambda I - \mathcal{A}) = {\mathcal{H}}$ for all $\lambda \in \R.$ Therefore, the closed graph theorem implies that $\sigma(\mathcal{A}) \cap i \R = \emptyset.$ Thus, we get the conclusion by applying Theorem \ref{NG-T-A.2} of Arendt and Batty. The proof is thus complete. \xqed{$\square$}
\subsubsection{Polynomial Stability}\label{NG-S-2.1.3}
In this part, we will prove the polynomial stability of System \eqref{NG-E(2.8)}-\eqref{NG-E(2.21)}. Our main result in this part is the following theorem.
\begin{Theorem}\label{NG-T-2.7}  
Under hypothesis {\rm(H)}, for all initial data $\mathbb{U}_0 \in D(\mathcal{A}),$ there exists a constant $C>0$ independent of $\mathbb{U}_0$ such that the energy of System \eqref{NG-E(2.8)}-\eqref{NG-E(2.21)} satisfies the following estimation
 \begin{equation}\label{NG-E(2.83)}
 E(t) \leq  \frac{C}{t^4} \left\|\mathbb{U}_0\right\|^2_{D(\mathcal{A})}, \quad \forall t>0.
 \end{equation}
 \end{Theorem}
\noindent From Lemma \ref{NG-L-2.5}   and Lemma \ref{NG-L-2.6}, we have seen that $i \R \subset \rho(\mathcal{A}),$  then for the proof of Theorem \ref{NG-T-2.7}, according to Theorem \ref{NG-T-A.5} (part (ii)), we need to prove that
\begin{equation}\label{NG-E(2.84)}
\sup_{\lambda\in\mathbb{R}}\left\|\left(i\lambda I-\mathcal{A}\right)^{-1}\right\|_{\mathcal{L}\left(\mathcal{H}\right)}=O\left(|\lambda|^{\frac{1}{2}}\right).
\end{equation}
 We will argue by contradiction. Indeed, suppose  there exists 
 \begin{equation*}
 \left\{(\lambda_n,\mathbb{U}_n:=\left(u_n,v_n,w_n ,y_n,z_n,\phi_n,\eta_n(L,\cdot)\right))\right\}_{n\geq 1}\subset \mathbb{R}_{+}^*\times D\left(\mathcal{A}\right),
 \end{equation*}
 such that
\begin{equation}\label{NG-E(2.85)}
\lambda_n \to  + \infty, \quad \left\|\mathbb{U}_n\right\|_{{\mathcal{H}}} = 1
\end{equation}
and there exists sequence $\mathbb{F}_n:=(f_{1,n},f_{2,n},f_{3,n},f_{4,n}, f_{5,n},f_{6,n},f_{7,n}(L,\cdot)) \in \mathcal{H}$, such that
\begin{equation}\label{NG-E(2.86)}
\lambda_n^{\ell} (i \lambda_n I - \mathcal{A})\mathbb{U}_n =\mathbb{F}_n \to  0 \ \textrm{in} \ {\mathcal{H}}.
\end{equation}
In case that $\ell=\frac{1}{2}$, we will check condition \eqref{NG-E(2.84)}  by finding a contradiction with $\left\|\mathbb{U}_n\right\|_{{\mathcal{H}}} = 1$ such as $\left\| \mathbb{U}_n\right\|_{\mathcal{H}} =o(1).$
 From  now on, for simplicity, we drop the index $n$. By detailing Equation \eqref{NG-E(2.86)}, we get the following system
 \begin{eqnarray}
i \lambda u -y = \lambda^{-\ell}\, f_1 &\textrm{in} \quad H^1(0,\alpha), \label{NG-E(2.87)}\\
i \lambda v -z  =  \lambda^{-\ell}\, f_2 &\textrm{in} \quad H^1(\alpha,\beta), \label{NG-E(2.88)}\\
i \lambda w - \phi =  \lambda^{-\ell}\, f_3 & \textrm{in}\quad  H^1(\beta,L), \label{NG-E(2.89)}\\
i \lambda  y - \kappa_1 u_{xx} =    \lambda^{-\ell}\, f_4 & \textrm{in} \quad L^2(0,\alpha), \label{NG-E(2.90)}\\
i \lambda  z -\left(\kappa_2 v_x + \delta_1 z_x\right)_{x}=   \lambda^{-\ell}\, f_5& \textrm{in} \quad L^2(\alpha,\beta), \label{NG-E(2.91)}\\
i \lambda  \phi - \kappa_3 w_{xx}=    \lambda^{-\ell}\, f_6&\textrm{in} \quad L^2(\beta,L), \label{NG-E(2.92)}\\
 \eta_\rho(L,\cdot)+i\tau \lambda\eta(L,\cdot)= \tau \lambda^{-\ell} f_7(L,\cdot) &\textrm{in} \quad L^2(0,1).\label{NG-E(2.93)}
\end{eqnarray}
Remark that, since  $\mathbb{U}=(u,v,w,y,z, \phi,\eta(L,\cdot)) \in D(\mathcal{A})$, we have the following boundary conditions  
\begin{equation}\label{NG-E(2.94)}
\left\{
\begin{array}{ll}

\displaystyle{\left|u_x(\alpha)\right| = \kappa_1 ^{-1}\left|\kappa_2 v_x(\alpha) + \delta_1 z_x(\alpha)\right|},& \displaystyle{ \left|y (\alpha)\right|= \left|z (\alpha)\right|,}
\nline
\displaystyle{\left|w_x(\beta)\right| = \kappa_3 ^{-1}\left|\kappa_2 v_x(\beta) + \delta_1 z_x(\beta)\right|},& \displaystyle{ \left|z(\beta)\right|= \left|\phi (\beta)\right|}
\end{array}\right.
\end{equation}
and
\begin{equation}\label{NG-E(2.95)}
w_x (L) = - \delta_3 \eta(L,0) - \delta_2 \eta (L,1), \quad \phi(L) = \eta(L,0).
\end{equation}
The proof of Theorem \ref{NG-T-2.7}   is divided into several lemmas.
\begin{lemma}\label{NG-L-2.8}  
Under hypothesis {\rm(H)}, for all $\ell\geq0$, the solution $(u,v,w, y,z , \phi,\eta(L,\cdot))\in D(\mathcal{A})$ of Equations \eqref{NG-E(2.87)}-\eqref{NG-E(2.93)} satisfies the following asymptotic behavior  estimations
\begin{eqnarray}
&\int_\alpha^\beta  |z_{x}|^2 =o\left(\lambda ^{-\ell}\right),\label{NG-E(2.96)}\\
&|\phi(L)|^2=|\eta(L,0)|^2=o\left(\lambda ^{-\ell}\right),\quad |\eta(L,1)|^2=o\left(\lambda ^{-\ell}\right)\label{NG-E(2.97)},\\
&\int_{\alpha}^{\beta} |v_x|^2 dx = o\left(\lambda ^{-\ell-2}\right), \label{NG-E(2.98)}\\
&|w_x (L)|^2=o\left(\lambda ^{-\ell}\right).\label{NG-E(2.99)}
\end{eqnarray}
\end{lemma}
\begin{proof}
Taking the inner product of \eqref{NG-E(2.86)}  with $\mathbb{U}$ in $\mathcal{H}$, then using the fact that $\mathbb{U}$ is uniformly bounded in $\mathcal{H}$, we get
\begin{equation*}
-\text{Re}\left<\mathcal{A}\mathbb{U} , \mathbb{U}\right>_{{\mathcal{H}}}=\text{Re}\left<(i \lambda I - \mathcal{A})\mathbb{U} , \mathbb{U}\right>_{{\mathcal{H}}}  =o\left(\lambda ^{-\ell}\right),
\end{equation*}
Now, under hypothesis {\rm(H)}, similar to Equation \eqref{NG-E(2.26)}, we get
\begin{equation}\label{NG-E(2.100)}
0\leq \delta_1\int_\alpha^\beta  |z_{x}|^2 dx+C_1|\eta(L,1)|^2
+C_2|\eta(L,0)|^2\leq-\text{Re}\left<\mathcal{A} \mathbb{U},\mathbb{U}\right>_{\mathcal{H}}   =o\left(\lambda ^{-\ell}\right),
\end{equation}
where
\begin{equation*}
C_1=\frac{1}{2}-\frac{\kappa_3|\delta_2|}{2p}>0\ \ \ \text{and}\ \ \ C_2=\kappa_3\delta_3-\frac{1}{2} - \frac{\kappa_3|\delta_2|\, p}{2}>0.
\end{equation*}
Therefore, from \eqref{NG-E(2.100)}, we get \eqref{NG-E(2.96)} and \eqref{NG-E(2.97)}. Next, from \eqref{NG-E(2.88)}, \eqref{NG-E(2.96)} and the fact that $(f_2)_x\to 0$ in $L^2(\alpha,\beta)$, we get \eqref{NG-E(2.98)}. Finally, from \eqref{NG-E(2.95)} and \eqref{NG-E(2.97)}, we obtain \eqref{NG-E(2.99)}.  Thus, the proof of the lemma is complete.
\end{proof}
\begin{lemma}\label{NG-L-2.9}  
Under hypothesis {\rm(H)}, for all $\ell\geq0$, the solution $(u,v,w, y,z , \phi,\eta(L,\cdot))\in D(\mathcal{A})$ of Equations \eqref{NG-E(2.87)}-\eqref{NG-E(2.93)} satisfies the following asymptotic behavior  estimation
\begin{equation}\label{NG-E(2.101)}
\int_0^1  |\eta(L,\rho)|^2d\rho =o\left(\lambda ^{-\ell}\right).
\end{equation}
\end{lemma}
\begin{proof} It follows from \eqref{NG-E(2.93)} that
\begin{equation*}
\eta(L,\rho)=\eta(L,0) e^{-i\tau\lambda\, \rho}+\tau \lambda^{-\ell}\int_0^\rho e^{i\tau\lambda\left(\xi-\rho\right)} f_7(L,\xi)d\xi \quad \forall\ \rho\in(0,1).
\end{equation*}
By using Cauchy Schwarz inequality, we get
\begin{equation*}
|\eta(L,\rho)|^2\leq 2|\eta(L,0)|^2 +2\tau^2\lambda^{-2\ell}  \left(\int_0^1 |f_7(L,\xi)| d\xi\right)^2\leq 2|\eta(L,0)|^2 +2\tau^2\lambda^{-2\ell} \int_0^1 |f_7(L,\xi)|^2 d\xi\quad \forall\ \rho\in(0,1). 
\end{equation*}
Integrating  over $(0,1)$ with respect to $\rho$, then using \eqref{NG-E(2.97)} and the fact that $f_7(L,\cdot)\to0$ in $L^2(0,1)$, we get 
\begin{equation*}
\int_0^1|\eta(L,\rho)|^2d\rho \leq 2|\eta(L,0)|^2 +2\tau^2\, \lambda^{-2\ell} \int_0^1 |f_7(L,\xi)|^2 d\xi=o\left(\lambda ^{-\ell}\right),
\end{equation*}
hence, we get \eqref{NG-E(2.101)}. Thus, the proof of the lemma is complete.
\end{proof}
 \begin{lemma}\label{NG-L-2.10}  
Under hypothesis {\rm(H)}, for all $\ell\geq0$, the solution $(u,v,w, y,z , \phi,\eta(L,\cdot))\in D(\mathcal{A})$ of Equations \eqref{NG-E(2.87)}-\eqref{NG-E(2.93)} satisfies the following asymptotic behavior  estimations
\begin{eqnarray}
\int_{\beta }^{ L} |\phi|^2 dx = o\left(\lambda ^{-\ell}\right), \quad \int_{\beta }^{ L } |w_x|^2 dx = o\left(\lambda ^{-\ell}\right),\label{NG-E(2.102)}\\
|w_x (\beta)|^2=o\left(\lambda ^{-\ell}\right),\quad |\phi(\beta)|^2=o\left(\lambda ^{-\ell}\right) ,\label{NG-E(2.103)}\\
\left|\kappa_2 v_x(\beta) + \delta_1 z_x(\beta)\right|^2=o\left(\lambda ^{-\ell}\right),\quad  \left|z(\beta)\right|^2=o\left(\lambda ^{-\ell}\right).\label{NG-E(2.104)}
\end{eqnarray}
 \end{lemma}
\begin{proof} Multiplying Equation \eqref{NG-E(2.92)} by $x \overline{w}_x$ and integrating over $(\beta,L),$ we get
\begin{equation}\label{NG-E(2.105)}
i \lambda  \int_{\beta }^{ L}  x \phi \overline{w}_x dx - \kappa_3 \int_{\beta }^{ L}  x w_{xx} \overline{w}_x dx =  \lambda^{-\ell} \int_{\beta }^{ L}  x\, f_6\, \overline{w}_x dx.
\end{equation}
From \eqref{NG-E(2.89)}, we deduce that
\begin{equation*}
i \lambda \overline{w}_x = -\overline{\phi}_x - \lambda^{-\ell} (\overline{f_3})_x .
\end{equation*}
Inserting the above result in \eqref{NG-E(2.105)}, then using the fact that $\phi,\, w_x$ are uniformly bounded in $L^2(\beta,L)$ and $(f_3)_x,\, f_6$ converge to zero   in $L^2(\beta,L)$ gives
\begin{equation*}
-  \int_{\beta }^{ L}  x\, \phi \overline{\phi}_xdx  - \kappa_3 \int_{\beta }^{ L}  x w_{xx} \overline{w}_x dx = o\left(\lambda ^{-\ell}\right).
\end{equation*}
Taking the real part in the above equation, then using by parts integration, we get 
\begin{equation*}
\frac{1}{2} \int_{\beta }^{ L} |\phi|^2 dx + \frac{\kappa_3}{2} \int_{\beta }^{ L} |w_x|^2 dx +\frac{\beta}{2} \left ( \kappa_3 |w_x (\beta)|^2 + |\phi(\beta)|^2 \right ) 
= \frac{L}{2} \left ( \kappa_3 |w_x (L)|^2 + |\phi(L)|^2 \right )+o\left(\lambda ^{-\ell}\right)  .
\end{equation*}
Inserting \eqref{NG-E(2.97)} and \eqref{NG-E(2.99)}  in the above equation, we get 
\begin{equation*}
\frac{1}{2} \int_{\beta }^{ L} |\phi|^2 dx + \frac{\kappa_3}{2} \int_{\beta }^{ L} |w_x|^2 dx +\frac{\beta}{2} \left ( \kappa_3 |w_x (\beta)|^2 + |\phi(\beta)|^2 \right ) 
= o\left(\lambda ^{-\ell}\right),
\end{equation*}
hence, we get \eqref{NG-E(2.102)} and \eqref{NG-E(2.103)}. Finally, from \eqref{NG-E(2.94)} and \eqref{NG-E(2.103)}, we obtain \eqref{NG-E(2.104)}. The proof is thus complete.
 \end{proof}
\begin{lemma}\label{NG-L-2.11}  
Under hypothesis {\rm(H)}, for all $\ell\geq0$, the solution $(u,v,w, y,z , \phi,\eta(L,\cdot))\in D(\mathcal{A})$ of Equations \eqref{NG-E(2.87)}-\eqref{NG-E(2.93)} satisfies the following asymptotic behavior  estimations
\begin{eqnarray}
\int_{\alpha}^{ \beta} |z|^2 dx = o\left(\lambda^{-\min\left(2\ell+\frac{1}{2},\ell+1\right)}\right),\label{NG-E(2.106)}\nline
\left| z(\alpha)\right|^2=o\left(\lambda^{-\min\left(2\ell,\ell+\frac{1}{2}\right)}\right), \quad \left| z(\beta)\right|^2=o\left(\lambda^{-\min\left(2\ell,\ell+\frac{1}{2}\right)}\right),\label{NG-E(2.107)}\nline
\left|\kappa_2 v_x(\alpha) + \delta_1 z_x(\alpha)\right|^2=o\left(\lambda^{-\min\left(2\ell-1,\ell-\frac{1}{2}\right)}\right).\label{NG-E(2.108)}
\end{eqnarray}
\end{lemma}
\begin{proof} Let $g \in C^1\left([\alpha,\beta]\right)$ such that 
\begin{equation*}
g(\beta)=-g(\alpha) = 1,\quad \max \limits_ {x \in [\alpha, \beta]} |g(x)|= c_g\quad \text{and}\quad \max \limits_ {x \in [\alpha, \beta]} |g'(x)| = c_{g'},
\end{equation*}
where $c_g$ and $c_{g'}$ are strictly positive constant numbers independent from $\lambda$.
The proof is divided into three  steps.\\[0.1in]
\textbf{Step 1.} In this step, we prove the following asymptotic behavior estimate
\begin{equation}\label{NG-E(2.109)}
 |z(\beta)|^2+|z(\alpha)|^2 \leq\left(\frac{\lambda^{\frac{1}{2}}}{2}+2 c_{g'}\right)
 \int_\alpha^\beta  |z|^2 \, dx+o\left(\lambda^{-\min\left(2\ell,\ell+\frac{1}{2}\right)}\right).
\end{equation}
First, from \eqref{NG-E(2.88)}, we have
\begin{equation}\label{NG-E(2.110)}
 z_x  = i \lambda v_x- \lambda^{-\ell}\, (f_2)_x \quad \textrm{in} \quad L^2(\alpha,\beta).
\end{equation}
Multiplying \eqref{NG-E(2.110)} by $2\, g \overline{z}$ and integrating over $(\alpha,\beta),$ then taking the real part, we get
\begin{equation*}
\int_\alpha^\beta g(x)\, (|z|^2)_x \, dx = \text{Re}\left\{2i \lambda \int_\alpha^\beta g(x)\,  v_x\overline{z}dx\right\}- \text{Re}\left\{2\lambda^{-\ell}\,\int_\alpha^\beta  g(x)\, (f_2)_x\overline{z} dx\right\},
\end{equation*}
using by parts integration in the left hand side of above equation, we get
\begin{equation*}
\left[ g(x)\, |z|^2\right]^{\beta}_{\alpha} =\int_\alpha^\beta g'(x)\, |z|^2 \, dx+ \text{Re}\left\{2i \lambda \int_\alpha^\beta g(x)\,  v_x\overline{z}dx\right\}- \text{Re}\left\{2\lambda^{-\ell}\,\int_\alpha^\beta  g(x)\, (f_2)_x\overline{z} dx\right\},
\end{equation*}
consequently,
\begin{equation}\label{NG-E(2.111)}
 |z(\beta)|^2+|z(\alpha)|^2 \leq c_{g'}\int_\alpha^\beta  |z|^2 \, dx+ 2 \lambda \, c_g \int_\alpha^\beta   |v_x|\left|{z}\right|dx+2\lambda^{-\ell} \, c_g\,\int_\alpha^\beta   \left|(f_2)_x\right|\left|{z}\right|\, dx.
\end{equation}
On the other hand, we have 
\begin{equation*}
2 \lambda \, c_g |v_x| |z|\leq \frac{\lambda^{\frac{1}{2}}|z|^2}{2}+2\lambda^{\frac{3}{2}} \, c_g^2\,|v_x|^2\ \ \ \text{and} \ \ \ 2 \lambda^{-\ell} c_g |(f_2)_x| |z|\leq  c_{g'} \, |z|^2+\frac{c_g ^ 2 \, \lambda^{-2\ell} }{c_{g'}}|(f_2)_x|^2.
\end{equation*}
Inserting the above equation in \eqref{NG-E(2.111)}, then using \eqref{NG-E(2.98)} and the fact that $(f_2)_x\to 0$ in $L^2(\alpha,\beta)$, we get
\begin{equation*}
 |z(\beta)|^2+|z(\alpha)|^2 \leq\left(\frac{\lambda^{\frac{1}{2}}}{2}+2 \, c_{g'} \right)
 \int_\alpha^\beta  |z|^2 \, dx+o\left(\lambda^{-\min\left(2\ell,\ell+\frac{1}{2}\right)}\right),
\end{equation*}
hence, we get \eqref{NG-E(2.109)}.\\[0.1in]
\textbf{Step 2.} In this step, we prove the following asymptotic behavior estimate
\begin{equation}\label{NG-E(2.112)}
 \left|\kappa_2 v_x(\alpha) + \delta_1  z_x(\alpha)\right|^2 +\left|\kappa_2 v_x(\beta) + \delta_1  z_x(\beta)\right|^2 \leq \frac{ \lambda^{\frac{3}{2}}}{2} \int_\alpha^\beta  |z|^2 dx +o\left(\lambda^{-\ell+\frac{1}{2}}\right).
\end{equation}
First, multiplying \eqref{NG-E(2.91)} by $-2\, g \left(\kappa_2\, \overline{v}_x+\delta_1  \overline{z}_x\right)$ and integrating over $(\alpha,\beta),$ then taking the real part, we get
\begin{equation*}
\begin{array}{ll}

\displaystyle{
\int_\alpha^\beta g(x)\left(\left|\kappa_2 v_x + \delta_1  z_x\right|^2\right)_x\, dx  =2\text{Re}\left\{i \lambda  \int_\alpha^\beta g(x)  z\left(\kappa_2 \overline{v}_x + \delta_1  \overline{z}_x\right) \, dx\right\} }\nline \hspace{3cm} \displaystyle{- 2   \lambda^{-\ell}\,\text{Re}\left\{ \int_\alpha^\beta g(x) f_5\left(\kappa_2 \overline{v}_x + \delta_1  \overline{z}_x\right) dx\right\},}
\end{array}
\end{equation*}
using by parts integration in the left hand side of above equation, we get
\begin{equation*}
\begin{array}{lll}
\displaystyle{
\left[g(x) \left|\kappa_2 v_x + \delta_1  z_x\right|^2\right]_{\alpha}^{\beta}  =\int_\alpha^\beta g'(x)\left|\kappa_2 v_x + \delta_1  z_x\right|^2\, dx+2\text{Re}\left\{i \lambda \int_\alpha^\beta g(x)  \, z\left(\kappa_2 \overline{v}_x + \delta_1  \overline{z}_x\right) \, dx\right\}}\nline\hspace{3.5cm}
\displaystyle{ - 2   \lambda^{-\ell}\,\text{Re}\left\{ \int_\alpha^\beta g(x)  f_5\left(\kappa_2 \overline{v}_x + \delta_1  \overline{z}_x\right) dx\right\},}
\end{array}
\end{equation*}
consequently,
\begin{equation*}
\begin{array}{lll}

\displaystyle{
 \left|\kappa_2 v_x(\beta) + \delta_1  z_x(\beta)\right|^2+ \left|\kappa_2 v_x(\alpha) + \delta_1  z_x(\alpha)\right|^2  \leq c_{g'} \int_\alpha^\beta \left|\kappa_2 v_x + \delta_1  z_x\right|^2\, dx+2 \lambda\, c_g \int_\alpha^\beta  |z|\left|\kappa_2 {v}_x + \delta_1  {z}_x\right| \, dx}\nline\hspace{3.5cm}
\displaystyle{ +  2  \lambda^{-\ell}\, c_g \, \int_\alpha^\beta  |f_5|\left|\kappa_2 {v}_x + \delta_1  {z}_x\right| dx.}
\end{array}
\end{equation*}
Now, using Cauchy Schwarz inequality, Equations \eqref{NG-E(2.96)}, \eqref{NG-E(2.98)} and the fact that $f_5\to 0$ in $L^2(\alpha,\beta)$ in the right hand side of above equation, we get 
\begin{equation}\label{NG-E(2.113)}
 \left|\kappa_2 v_x(\beta) + \delta_1  z_x(\beta)\right|^2 +\left|\kappa_2 v_x(\alpha) + \delta_1  z_x(\alpha)\right|^2 \leq 2\lambda  c_g \int_\alpha^\beta  |z|\left|\kappa_2 {v}_x + \delta_1  {z}_x\right| \, dx +o\left(\lambda^{-\ell}\right) .
\end{equation}
On the other hand, we have
\begin{equation*}
2 \lambda c_g \, |z|\left|\kappa_2 {v}_x + \delta_1  {z}_x\right|\leq \frac{\lambda^{\frac{3}{2}}}{2}|z|^2+2\lambda^{\frac{1}{2} } c_g^2  \,  \left|\kappa_2 {v}_x + \delta_1  {z}_x\right|^2.
\end{equation*}
Inserting the above equation in \eqref{NG-E(2.113)}, then using Equations \eqref{NG-E(2.96)} and \eqref{NG-E(2.98)}, we get
\begin{equation*}
 \left|\kappa_2 v_x(\alpha) + \delta_1  z_x(\alpha)\right|^2 +\left|\kappa_2 v_x(\beta) + \delta_1  z_x(\beta)\right|^2 \leq \frac{ \lambda^{\frac{3}{2}}}{2}  \int_\alpha^\beta  |z|^2 dx +o\left(\lambda^{-\ell+\frac{1}{2}}\right) ,
\end{equation*}
hence, we get \eqref{NG-E(2.112)}.\\[0.1in]
\textbf{Step 3.}  In this step, we prove the asymptotic behavior  estimations of  \eqref{NG-E(2.106)}-\eqref{NG-E(2.108)}. First, multiplying  \eqref{NG-E(2.91)} by $-i \lambda^{-1} \overline{z}$ and integrating over $(\alpha,\beta),$ then taking the real part, we get
\begin{equation*}
\int_\alpha^\beta | z|^2\, dx  =- \text{Re}\left\{i \lambda^{-1}\int_\alpha^\beta\left(\kappa_2 v_x + \delta_1  z_x\right)_{x}\,\overline{z} \, dx\right\}-\text{Re}\left\{ i\,    \lambda^{-\ell-1}\int_\alpha^\beta f_5\overline{z} \, dx\right\},
\end{equation*}
consequently,
\begin{equation}\label{NG-E(2.114)}
\int_\alpha^\beta | z|^2\, dx  \leq  \lambda^{-1} \left|\int_\alpha^\beta\left(\kappa_2 v_x + \delta_1  z_x\right)_{x}\,\overline{z} \, dx\right|+    \lambda^{-\ell-1}\int_\alpha^\beta \left|f_5\right|\left|z \right| dx.
\end{equation}
From the fact that $z$  is uniformly bounded in   $L^2(\alpha,\beta)$ and $f_5\to 0$ in $L^2(\alpha,\beta)$, we get
\begin{equation}\label{NG-E(2.115)}
\lambda^{-\ell-1}\int_\alpha^\beta \left|f_5\right|\left|z \right| dx=o\left(\lambda^{-\ell-1}\right).
\end{equation}
On the other hand, using by parts integration and \eqref{NG-E(2.96)}, \eqref{NG-E(2.98)}, we get
\begin{equation}\label{NG-E(2.116)}
\begin{array}{ll}

\displaystyle{
\left|\int_\alpha^\beta\left(\kappa_2 v + \delta_1  z\right)_{xx}\,\overline{z} \, dx\right|=\left|\left[\left(\kappa_2 v_x + \delta_1  z_x\right) \overline{z}\right]^\beta_\alpha-\int_\alpha^\beta\left(\kappa_2 v_x + \delta_1  z_x\right) \overline{z}_x \, dx\right|}\nline

\displaystyle{
\leq \left|\kappa_2 v_x(\beta) + \delta_1  z_x(\beta)\right| \left|{z}(\beta)\right|+\left|\kappa_2 v_x(\alpha) + \delta_1  z_x(\alpha)\right| \left|{z}(\alpha)\right|+\int_\alpha^\beta\left|\kappa_2 v_x + \delta_1  z_x\right| \left|{z}_x\right| \, dx
}
\nline
\displaystyle{
\leq \left|\kappa_2 v_x(\beta) + \delta_1  z_x(\beta)\right| \left|{z}(\beta)\right|+\left|\kappa_2 v_x(\alpha) + \delta_1  z(\alpha)\right| \left|{z}(\alpha)\right|+o\left(\lambda^{-\ell}\right).
}
\end{array}
\end{equation}
Inserting \eqref{NG-E(2.115)} and \eqref{NG-E(2.116)} in \eqref{NG-E(2.114)}, we get
\begin{equation}\label{NG-E(2.117)}
\int_\alpha^\beta | z|^2\, dx  \leq  \lambda^{-1} \left|\kappa_2 v_x(\beta) + \delta_1  z_x(\beta)\right| \left|{z}(\beta)\right|+ \lambda^{-1} \left|\kappa_2 v_x(\alpha) + \delta_1  z_x(\alpha)\right| \left|{z}(\alpha)\right|+o\left(\lambda^{-\ell-1}\right).
\end{equation}
Now, for $\zeta=\beta$ or $\zeta=\alpha$,  we have
\begin{equation*}
\lambda^{-1} \left|\kappa_2 v_x(\zeta) + \delta_1  z_x(\zeta)\right| \left|{z}(\zeta)\right|\leq 
 \frac{\lambda^{-\frac{1}{2}}}{2}|z(\zeta)|^2+\frac{\lambda^{-\frac{3}{2}}}{2}\left|\kappa_2 v_x(\zeta) + \delta_1  z_x(\zeta) \right|^2.
\end{equation*}
Inserting the above equation in \eqref{NG-E(2.117)}, we get
\begin{equation*}
\int_\alpha^\beta | z|^2\, dx  \leq  \frac{\lambda^{-\frac{1}{2}}}{2}\left(|z(\alpha)|^2+|z(\beta)|^2\right)+\frac{\lambda^{-\frac{3}{2}}}{2}\left(\left|\kappa_2 v_x(\alpha) + \delta_1  z_x(\alpha) \right|^2+\left|\kappa_2 v_x(\beta) + \delta_1  z_x(\beta) \right|^2\right)+o\left(\lambda^{-\ell-1}\right).
\end{equation*}
Next, inserting Equations \eqref{NG-E(2.109)} and \eqref{NG-E(2.112)} in the above inequality, we obtain 
\begin{equation*}
\int_\alpha^\beta | z|^2\, dx \leq  \left(\frac{1}{2}+\frac{c_{g'}}{ \lambda^{\frac{1}{2}}  }\right)
 \int_\alpha^\beta  |z|^2 \, dx+o\left(\lambda^{-\min\left(2\ell+\frac{1}{2},\ell+1\right)}\right),
\end{equation*}
consequently,
\begin{equation*}
  \left(\frac{1}{2}-\frac{c_{g'}}{ \lambda^{\frac{1}{2}} } \right)
 \int_\alpha^\beta  |z|^2 \, dx\leq o\left(\lambda^{-\min\left(2\ell+\frac{1}{2},\ell+1\right)}\right).
\end{equation*}
Since $\lambda\to+\infty$, by choosing  $ \lambda >4   \, c_{g'}^2 $, we get 
\begin{equation*}
 0 < \left(\frac{1}{2}-\frac{c_{g'}}{ \lambda^{\frac{1}{2}} } \right)  \int_\alpha^\beta  |z|^2 \, dx\leq o\left(\lambda^{-\min\left(2\ell+\frac{1}{2},\ell+1\right)}\right),
\end{equation*}
hence, we get \eqref{NG-E(2.106)}. Finally,  inserting \eqref{NG-E(2.106)} in \eqref{NG-E(2.109)}  and \eqref{NG-E(2.112)} and using the first asymptotic estimates of \eqref{NG-E(2.104)}, we get 
\eqref{NG-E(2.107)} and \eqref{NG-E(2.108)}. Thus, the proof of the
lemma is complete.
\end{proof}
\begin{Remark}\label{NG-R-2.12}  
An example about $g$, we can take $g(x)=\cos\left(\frac{ (\beta-x)\pi}{\beta-\alpha}\right)$ to get 
\begin{equation*}
g(\beta)=-g(\alpha)=1,\quad g\in   C^1([\alpha,\beta]),\quad \max \limits_ {x \in [\alpha, \beta]} |g(x)|=  1,\quad \max \limits_ {x \in [\alpha, \beta]} |g'(x)| = \frac{\pi}{\beta-\alpha}.
\end{equation*}
Also, we can take  
\begin{equation*}
g(x)=\left(\frac{\beta-x}{\beta-\alpha}\right)^2-3\left(\frac{\beta-x}{\beta-\alpha}\right)+1.
\end{equation*}
\xqed{$\square$}
\end{Remark}
 \begin{lemma}\label{NG-L-2.13} 
Under hypothesis {\rm(H)},  for all $\ell\geq0$, the solution $(u,v,w, y,z , \phi,\eta(L,\cdot))\in D(\mathcal{A})$ of Equations \eqref{NG-E(2.87)}-\eqref{NG-E(2.93)} satisfies the following asymptotic behavior  estimations
\begin{equation} \label{NG-E(2.118)}
\int_{0 }^{ \alpha } |y|^2 dx = o\left(\lambda^{-\min\left(2\ell-1,\ell-\frac{1}{2}\right)}\right)\ \ \ \text{and} \ \ \ \int_{0 }^{ \alpha } |u_x|^2 dx =  o\left(\lambda^{-\min\left(2\ell-1,\ell-\frac{1}{2}\right)}\right).
\end{equation}
 \end{lemma}
 \begin{proof}
Multiply Equation \eqref{NG-E(2.90)} by $x \overline{u}_x$ and integrating over $(0,\alpha),$ we get
\begin{equation}\label{NG-E(2.119)}
i \lambda  \int_{0}^{\alpha} x y \overline{u}_x dx - \kappa_1 \int_{0}^{\alpha} x u_{xx} \overline{u}_x dx =  \lambda^{-\ell} \int_{0}^{\alpha} x\, f_4\, \overline{u}_x dx.
\end{equation}
From \eqref{NG-E(2.87)}, we deduce that
\begin{equation*}
i \lambda \overline{u}_x = -\overline{y}_x - \lambda^{-\ell} (\overline{f_1})_x .
\end{equation*}
Inserting the above result in \eqref{NG-E(2.119)}, then using the fact that $u_x,\, y$ are uniformly bounded in $L^2(0,\alpha)$ and $(f_1)_x,\, f_4$ converge  to zero  in $L^2(0, \alpha)$ gives
\begin{equation*}
-  \int_{0}^{\alpha} x y \overline{y}_xdx  - \kappa_1 \int_{0}^{\alpha} x u_{xx} \overline{u}_x dx = o\left(\lambda ^{-\ell}\right).
\end{equation*}
Taking the real part in the above equation, then using by parts integration, we get 
\begin{equation}\label{NG-E(2.120)}
\frac{1}{2} \int_{0}^{\alpha} |y|^2 dx + \frac{\kappa_1}{2} \int_{0}^{\alpha} |u_x|^2 dx - \frac{\alpha}{2} \left ( \kappa_1 |u_x (\alpha)|^2 + |y(\alpha)|^2 \right )   = o\left(\lambda ^{-\ell}\right).
\end{equation}
Inserting the boundary conditions  \eqref{NG-E(2.94)} at $x=\alpha$  in \eqref{NG-E(2.120)} gives
\begin{equation*}
\frac{1}{2} \int_{0}^{\alpha} |y|^2 dx + \frac{\kappa_1}{2} \int_{0}^{\alpha} |u_x|^2 dx = \frac{\alpha}{2} \left(\kappa_1^{-1}  |\kappa_2 v_x (\alpha) + \delta_1  z_x(\alpha)|^2 + |z(\alpha)|^2\right)   + o\left(\lambda ^{-\ell}\right).
\end{equation*}
Inserting  \eqref{NG-E(2.107)}-\eqref{NG-E(2.108)} in the above equation, we obtain the first and the second  asymptotic estimates of \eqref{NG-E(2.118)}.  The proof is thus complete.
 \end{proof}
\noindent \textbf{Proof of Theorem \ref{NG-T-2.7}.} From Lemma \ref{NG-L-2.8}, Lemma \ref{NG-L-2.9},  Lemma \ref{NG-L-2.10}, Lemma \ref{NG-L-2.11}   and Lemma \ref{NG-L-2.13}, we get
\begin{equation*}
\begin{array}{lll}

\left\|U\right\|_{\mathcal{H}}^2=\kappa_1\int_{0 }^{ \alpha } |u_x|^2 dx+\kappa_2\int_{\alpha}^{\beta} |v_x|^2 dx +\kappa_3\int_{\beta }^{ L } |w_x|^2 dx +\int_{0 }^{ \alpha } |y|^2 dx \nline\hspace{1cm}

+\int_{\alpha}^{ \beta} |z|^2 dx+\int_{\beta }^{ L} |\phi|^2 dx+\tau\int_0^1  |\eta(L,\rho)|^2d\rho=o\left(\lambda^{-\min\left(2\ell-1,\ell-\frac{1}{2}\right)}\right).

\end{array}
\end{equation*}
To obtain $\| \mathbb{ U}  \|_{{\mathcal{H}}} = o(1)$, we need $\min\left(2\ell-1,\ell-\frac{1}{2}\right)\geq 0$, so we choose $\ell=\frac{1}{2}$ as the optimal value. Hence, we obtain that $\| \mathbb{ U}  \|_{{\mathcal{H}}} = o(1)$ which contradicts \eqref{NG-E(2.85)}. Therefore, the energy of System \eqref{NG-E(2.8)}-\eqref{NG-E(2.21)} satisfies estimation \eqref{NG-E(2.83)} for all initial data $\mathbb{U}_0 \in D(\mathcal{A}).$ \xqed{$\square$}
\subsection{Wave equation with local  Kelvin-Voigt damping  near the boundary and boundary delay feedback}\label{NG-S-2.2} 
\noindent In this subsection, we study the stability of System \eqref{NG-E(2.1)}, but in the case that the Kelvin-Voigt damping is near the boundary, {\it i.e.} $\alpha=0$ and $0<\beta<L$ (see Figure \ref{Fig2}).
For this aim, we denote the longitudinal displacement by $U$ and this displacement is divided into two parts
\begin{equation*}
U(x,t)= \left\{
        \begin{array}{ll}
            v(x,t), & (x,t) \in \ (\alpha,\beta)\times (0, + \infty) , \nline
            w(x,t), &  (x,t) \in \ (\beta,L)\times (0, + \infty).
        \end{array}
    \right.
\end{equation*}
In this case, System \eqref{NG-E(2.1)} is equivalent to the following system
\begin{equation}\label{NG-E(2.121)}
\left\{
\begin{array}{lll}
v_{tt} -\left(\kappa_2  v_{x} +\delta_1  v_{xt}\right)_x = 0, &(x,t) \in (0, \beta ) \times (0, + \infty),\nline
 w_{tt} -\kappa_3  w_{xx} = 0, &(x,t) \in  ( \beta, L )  \times (0, + \infty),\nline
 \tau \eta_t(L,\rho,t)+\eta_\rho(L,\rho,t)=0,& (\rho,t)\in (0,1)\times (0, + \infty),\nline
v(0,t) = 0, & t\in (0,+\infty),\nline
w_x(L,t) = - \delta_3 w_{t}(L,t) - \delta_2 \eta(L,1,t),& t\in (0, + \infty),\nline
v (\beta , t) = w (\beta , t), &t\in  (0, + \infty),\nline
\kappa_2 v_x (\beta , t) + \delta_1 v_{xt} (\beta,t ) =\kappa_3 w_x (\beta ,t), &t\in  (0, + \infty),\nline
\left(v(x, 0 ),v_t(x,0)  \right) =\left(v_0(x), v_1 (x)\right), \quad &x\in  (\alpha,\beta),\nline
\left(w(x, 0 ),w_t(x,0)  \right) =\left(w_0(x), w_1 (x)\right), \quad &x\in   (\beta,L),\nline
\eta(L,\rho,0)=f_0(L,-\rho\, \tau),& \rho\in (0,1),

\end{array}\right.
\end{equation}
where the initial data $(v_0,v_1,w_0,w_1,f_0)$ belongs to a suitable space.  Similar to Section \ref{NG-S-2.1}, we define
\begin{equation*}
\begin{array}{ll}
\displaystyle{ \mathbb{X}^m =   H^m (0, \beta) \times H^m (\beta, L), \quad m = 1 , 2, }\nline 
\displaystyle{\mathbb{X}^0 =  L^2 (0, \beta) \times L^2(\beta,L)},\quad
\displaystyle{  \mathbb{X}^1_L = \{(v,w) \in \mathbb{X}^1\ |  \  v (0)=0, \ v(\beta) = w(\beta) \},}
\end{array}
\end{equation*}
where the Hilbert space $\mathbb{X}^0$ is equipped  with the norm:
\begin{equation*}
\left\|(v,w)\right\|^2_{\mathbb{X}^0}=  \int_0^\beta |v|^2dx+\int_\beta^L |w|^2dx.
\end{equation*}
Moreover, it is easy to check that the space $\mathbb{X}^1_L$ is Hilbert space over $\mathbb{C}$ equipped  with the norm:
\begin{equation*}
\left\|(v,w)\right\|^2_{\mathbb{X}^1_L}=\kappa_2 \int_0^\beta |v_x|^2dx+\kappa_3 \int_\beta^L |w_x|^2dx.
\end{equation*}
In addition, by Poincaré inequality we can easily  verify that there exists $C>0$, such that
\begin{equation*}
\left\|(v,w)\right\|_{\mathbb{X}^0}\leq C \left\|(v,w)\right\|_{\mathbb{X}^1_L}.
\end{equation*}
\noindent We now  define the Hilbert energy space  by
\begin{equation*}
{\mathcal{H}_1} = \mathbb{X}^1_L  \times \mathbb{X}^0 \times L^2(0,1)
\end{equation*}
\noindent  equipped with the following inner product
\begin{equation*}
\langle \mathtt{U},\tilde{\mathtt{U}}\rangle _{{\mathcal{H}_1 }} = \kappa_2 \int_{0}^{\beta} v_x \overline{\tilde{v}}_x dx + \kappa_3 \int_{\beta}^{L} w_x \overline{\tilde{w}}_x dx+  \int_{0}^{\beta} z \overline{\tilde{z}} dx +  \int_{\beta}^{L} \phi \overline{\tilde{\phi}} dx +\tau \int_{0}^{1} \eta(L,\rho)\, \overline{\tilde{\eta}}(L,\rho) d \rho,
\end{equation*}
where $\mathtt{U}=( v,w ,z, \phi, \eta(L,\cdot) )\in{\mathcal{H}_1}$ and  $\tilde{\mathtt{U}}= ( \tilde{ v}, \tilde{w},  \tilde{z}, \tilde{\phi}, \tilde{\eta}(L,\cdot)) \in{\mathcal{H}_1}$. We use  $\|\mathtt{U}\|_{\mathcal{H}_1}$ to denote the corresponding norm. We  define the linear unbounded operator $\mathcal{A}_1: D(\mathcal{A}_1) \subset {\mathcal{H}_1}\longrightarrow {\mathcal{H}_1} $ by:
\begin{equation*}
\begin{array}{lll}
\displaystyle{D(\mathcal{A}_1) = \bigg\{\mathtt{U}=(v,w,z,\phi, \eta(L,\cdot))\in {\mathbb{X}_L^1 \times \mathbb{X}_L^1 \times H^1 (0,1) }\ |  \ ( \kappa_2 v  +\delta_1 z , w) \in \mathbb{X}^2,} 
\\ \noalign{\medskip}\hspace{1.9cm}
\displaystyle{ \kappa_2 v_x (\beta) +\delta_1  z_x (\beta)= \kappa_3 w _x (\beta),\ w_x (L) = - \delta_3 \eta(L,0) - \delta_2 \eta (L,1), \ \phi(L) = \eta(L,0)\bigg\}}
\end{array}
\end{equation*}
and for all $\mathtt{U}=(v,w,z, \phi, \eta(L,\cdot))\in D(\mathcal{A}_1)$
\begin{equation*}
\mathcal{A}_1\mathtt{U} = \left(  z ,\phi, \left(\kappa_2  v_{x} +\delta_1 z_{x}\right)_x , \kappa_3 w_{xx} , - \tau^{-1} \eta_{\rho} (L,\cdot)\right).
\end{equation*}
If $\mathtt{U}=(v,w,v_t,w_t, \eta(L,\cdot))$ is a regular solution of System \eqref{NG-E(2.121)}, then we transform this system into the following initial value problem
\begin{equation} \label{NG-E(2.122)}
\begin{cases}
\mathtt{U}_t &= \mathcal{A}_1 \mathtt{U},\\
\mathtt{U}(0)&= \mathtt{U}_0,
\end{cases}
\end{equation}
where $\mathtt{U}_0 = (v_0 , w_0 , v_1, w_1, f_0(L,-\cdot \tau))\in{\mathcal{H}_1}.$ Note that $D(\mathcal{A}_1)$ is dense in $\mathcal{H}_1$ and that for all $\mathtt{U}\in D(\mathcal{A}_1)$, we have
\begin{equation}\label{NG-E(2.123)}
\text{Re}\left<\mathcal{A}_1 \mathtt{U},\mathtt{U}\right>_{\mathcal{H}_1}  \leq -\delta_1\int_0^\beta  |z_{x}|^2 dx-\left(\frac{1}{2}-\frac{\kappa_3|\delta_2|}{2p}\right)|\eta(L,1)|^2
-\left(\kappa_3\delta_3-\frac{1}{2}-\frac{\kappa_3|\delta_2|\, p}{2}\right)|\eta(L,0)|^2,
\end{equation}
where $p$ is defined in \eqref{NG-E(2.7)}. Consequently, under hypothesis {\rm(H)}, the system becomes dissipative. We can easily adapt the proof in Subsection \ref{NG-S-2.1.1} to prove the well-posedness of System \eqref{NG-E(2.122)}.
\begin{Theorem}\label{NG-T-2.14}  
Under hypothesis {\rm(H)}, for all initial data $\mathtt{U}_0 \in \mathcal{H}_1,$ the System \eqref{NG-E(2.121)}  is exponentially stable.
 \end{Theorem}
\noindent	According to Theorem \ref{NG-T-A.5} (part (i)), we have to check if the following conditions hold,
\begin{equation}\label{NG-E(2.124)}
	\ i\mathbb{R}\subseteq\rho\left(\mathcal{A}_1\right) 
\end{equation}
and
\begin{equation}\label{NG-E(2.125)}
	\sup_{\lambda\in\mathbb{R}}\left\|\left(i\lambda I-\mathcal{A}_1\right)^{-1}\right\|_{\mathcal{L}\left(\mathcal{H}_1\right)}=O\left(1\right).
\end{equation}
\begin{proof} First,  we can easily adapt the proof in Subsection  \ref{NG-S-2.1.2} to prove the strong stability $\left(\text{condition \eqref{NG-E(2.124)}}\right)$ of System \eqref{NG-E(2.121)}. Next, we will prove condition  \eqref{NG-E(2.125)} by a contradiction argument. Indeed, suppose  there exists 
 \begin{equation*}
 \left\{(\lambda_n,\mathtt{U}_n:=\left(v_n,w_n ,z_n,\phi_n,\eta_n(L,\cdot)\right))\right\}_{n\geq 1}\subset \mathbb{R}_{+}^*\times D\left(\mathcal{A}_1\right),
 \end{equation*}
 such that
\begin{equation}\label{NG-E(2.126)}
\lambda_n \to  + \infty, \quad \left\|\mathtt{U}_n\right\|_{{\mathcal{H}_1}} = 1
\end{equation}
and there exists sequence $\mathtt{G}_n:=(g_{1,n},g_{2,n},g_{3,n},g_{4,n},g_{5,n}(L,\cdot)) \in \mathcal{H}_1$, such that
\begin{equation}\label{NG-E(2.127)}
 (i \lambda_n I - \mathcal{A}_1)\mathtt{U}_n =\mathtt{G}_n \to  0 \ \textrm{in} \ {\mathcal{H}_1}.
\end{equation}
We will check condition \eqref{NG-E(2.125)}  by finding a contradiction with $\left\|\mathtt{U}_n\right\|_{{\mathcal{H}}_1} = 1$ such as $\left\| \mathtt{U}_n\right\|_{\mathcal{H}_1} =o(1).$
 From  now on, for simplicity, we drop the index $n$. By detailing Equation \eqref{NG-E(2.127)}, we get the following system
 \begin{eqnarray}
i \lambda v -z  =   g_1 &\textrm{in} \quad H^1(0,\beta), \label{NG-E(2.128)}\\
i \lambda w - \phi =   g_2 & \textrm{in}\quad  H^1(\beta,L), \label{NG-E(2.129)}\\
i \lambda  z -\left(\kappa_2 v_x + \delta_1 z_x\right)_{x}=    g_3& \textrm{in} \quad L^2(0,\beta), \label{NG-E(2.130)}\\
i \lambda  \phi - \kappa_3 w_{xx}=     g_4&\textrm{in} \quad L^2(\beta,L), \label{NG-E(2.131)}\\
 \eta_\rho(L,\cdot)+i\tau \lambda\eta(L,\cdot)= \tau  g_5(L,\cdot) &\textrm{in} \quad L^2(0,1).\label{NG-E(2.132)}
\end{eqnarray}
Remark that, since  $\mathtt{U}=(v,w,z, \phi,\eta(L,\cdot)) \in D(\mathcal{A}_1)$, we have the following boundary conditions  
\begin{eqnarray}
\left|w_x(\beta)\right| = \kappa_3 ^{-1}\left|\kappa_2 v_x(\beta) + \delta_1 z_x(\beta)\right|,&  \left|z(\beta)\right|= \left|\phi (\beta)\right|,\label{NG-E(2.133)}\\
w_x (L) = - \delta_3 \eta(L,0) - \delta_2 \eta (L,1), & \phi(L) = \eta(L,0).\label{NG-E(2.134)}
\end{eqnarray}
Taking the inner product of \eqref{NG-E(2.127)} with $\mathtt{U}$ in $\mathcal{H}_1$, then using  \eqref{NG-E(2.123)}, hypothesis  {\rm (H)} and the fact that $\mathtt{U}$ is uniformly bounded in $\mathcal{H}_1$, we obtain
\begin{equation}\label{NG-E(2.135)}
\int_0^\beta  |z_{x}|^2 =o\left(1\right),\quad |\phi(L)|^2=|\eta(L,0)|^2=o\left(1\right),\quad \quad |\eta(L,1)|^2=o\left(1\right).
\end{equation}
From \eqref{NG-E(2.128)}, then using the first asymptotic estimate  of \eqref{NG-E(2.135)} and the fact that $(g_1)_x\to 0$ in $L^2(0,\beta)$, we get
\begin{equation}\label{NG-E(2.136)}
\int_{0}^{\beta} |v_x|^2 dx = o\left(\lambda ^{-2}\right).
\end{equation}
From the first asymptotic estimate  of \eqref{NG-E(2.134)}, then using  the second  and the third asymptotic estimates of \eqref{NG-E(2.135)}, we obtain 
\begin{equation}\label{NG-E(2.137)}
|w_x (L)|^2=o\left(1\right).
\end{equation}
Similar to  Lemma \ref{NG-L-2.9}, with $\ell=0$, from  \eqref{NG-E(2.132)}, then using the second  and the third asymptotic estimates of \eqref{NG-E(2.135)}, we obtain 
\begin{equation}\label{NG-E(2.138)}
\int_0^1  |\eta(L,\rho)|^2d\rho =o\left(1\right).
\end{equation}
Similar to Lemma \ref{NG-L-2.10}, with $\ell=0$, multiplying Equation \eqref{NG-E(2.131)} by $x \overline{w}_x$ and integrating over $(\beta,L),$ after that using the  fact that $i \lambda \overline{w}_x = -\overline{\phi}_x - (\overline{g_2})_x$,  then using the fact that $\phi,\, w_x$ are uniformly bounded in $L^2(\beta,L)$ and $(g_2)_x,\, g_4$ converge  to zero  in $L^2(\beta,L)$ gives
\begin{equation*}
-  \int_{\beta }^{ L}  x\, \phi \overline{\phi}_xdx  - \kappa_3 \int_{\beta }^{ L}  x w_{xx} \overline{w}_x dx = o\left(1\right).
\end{equation*}
Taking the real part in the above equation, then using by parts integration, Equation \eqref{NG-E(2.137)} and the second  asymptotic estimate of \eqref{NG-E(2.135)}, we obtain 
\begin{equation*}
\frac{1}{2} \int_{\beta }^{ L} |\phi|^2 dx + \frac{\kappa_3}{2} \int_{\beta }^{ L} |w_x|^2 dx +\frac{\beta}{2} \left ( \kappa_3 |w_x (\beta)|^2 + |\phi(\beta)|^2 \right ) 
= o\left(1\right),
\end{equation*}
hence, we get
\begin{equation}\label{NG-E(2.139)}
 \int_{\beta }^{ L} |\phi|^2 dx=o(1),\quad  \int_{\beta }^{ L} |w_x|^2 dx=o(1),\quad   |w_x (\beta)|^2=o(1),\quad  |\phi(\beta)|^2= o\left(1\right).
\end{equation}
Inserting the third and the fourth  asymptotic estimates of \eqref{NG-E(2.139)} in \eqref{NG-E(2.133)}, we get
\begin{equation}\label{NG-E(2.140)}
\left|\kappa_2 v_x(\beta) + \delta_1 z_x(\beta)\right|=o(1),\quad  \left|z(\beta)\right|= o(1).
\end{equation}
Similar to step 3 of Lemma \ref{NG-L-2.11}, with $\alpha=0$ and $\ell=0$, multiplying  \eqref{NG-E(2.130)} by $-i \lambda^{-1} \overline{z}$ and integrating over $(0,\beta),$ taking the real part, then  using the fact that $z$  is uniformly bounded in   $L^2(0,\beta)$ and $g_3\to 0$ in $L^2(0,\beta)$, we get
\begin{equation}\label{NG-E(2.141)}
\int_0^\beta | z|^2\, dx  \leq  \lambda^{-1} \left|\int_0^\beta\left(\kappa_2 v_x + \delta_1  z_x\right)_{x}\,\overline{z} \, dx\right|+    o\left(\lambda^{-1}\right).
\end{equation}
On the other hand, using by parts integration, the fact that $z(0)=0$, and Equations   \eqref{NG-E(2.135)}-\eqref{NG-E(2.136)}, \eqref{NG-E(2.140)}, we get
\begin{equation}\label{NG-E(2.142)}
\begin{array}{lll}

\displaystyle{
\left|\int_0^\beta\left(\kappa_2 v_x + \delta_1  z_x\right)_{x}\,\overline{z} \, dx\right|}&=\displaystyle{\left|\left[\left(\kappa_2 v_x + \delta_1  z_x\right) \overline{z}\right]^\beta_0-\int_0^\beta\left(\kappa_2 v_x + \delta_1  z_x\right) \overline{z}_x \, dx\right|}\nline

&\displaystyle{
\leq \left|\kappa_2 v_x(\beta) + \delta_1  z_x(\beta)\right| \left|{z}(\beta)\right|+\int_0^\beta\left|\kappa_2 v_x + \delta_1  z_x\right| \left|{z}_x\right| \, dx=o(1)}.

\end{array}
\end{equation}
Inserting \eqref{NG-E(2.142)} in \eqref{NG-E(2.141)}, we get
\begin{equation}\label{NG-E(2.143)}
\int_0^\beta | z|^2\, dx =  o\left(\lambda^{-1}\right).
\end{equation}
Finally, from \eqref{NG-E(2.136)}, \eqref{NG-E(2.138)}, \eqref{NG-E(2.139)} and \eqref{NG-E(2.143)}, we get
\begin{equation*}
\left\| \mathtt{U}\right\|_{\mathcal{H}_1} =o(1),
\end{equation*}
 which contradicts \eqref{NG-E(2.126)}. Therefore, \eqref{NG-E(2.125)} holds and the result follows from Theorem  \ref{NG-T-A.5} (part (i)).
\end{proof}
\section{Wave equation with local internal Kelvin-Voigt damping and local internal delay feedback}\label{NG-S-3}
\noindent In this section, we study the stability of System \eqref{NG-E(1.2)}. We assume that there exists $\alpha$ and $\beta$ such that    $0<\alpha<\beta<L$, in this case, the Kelvin-Voigt damping and the time delay feedback  are  locally internal (see Figure \ref{Fig3}).  For this aim, we denote the longitudinal displacement by $U$ and this displacement is divided into three parts
\begin{equation*}
U(x,t)= \left\{
        \begin{array}{ll}
           u(x,t), &  (x,t) \in \ (0,\alpha)\times (0, + \infty), \nline
            v(x,t), & (x,t) \in \ (\alpha,\beta)\times (0, + \infty) , \nline
            w(x,t), &  (x,t) \in \ (\beta,L)\times (0, + \infty).
        \end{array}
    \right.
\end{equation*}
 Furthermore,  let us introduce the auxiliary unknown
\begin{equation*}
\eta(x,\rho,t)=v_t(x,t-\rho\, \tau),\quad x\in(\alpha,\beta),\ \rho\in(0,1),\ t>0.
\end{equation*}
In this case, System \eqref{NG-E(1.2)} is equivalent to the following system
\begin{eqnarray}
 u_{tt} - \kappa_1 u_{xx} = 0, &(x,t) \in  (0,\alpha) \times (0, + \infty), \label{NG-E(3.1)}\nline
 v_{tt} -\left(\kappa_2  v_{x} +  \delta_1 v_{xt} (x,t) + \delta_2 \eta_{x}(x,1,t)\right)_x = 0, &(x,t) \in (\alpha, \beta ) \times (0, + \infty),\label{NG-E(3.2)}\nline
 w_{tt} - \kappa_3 w_{xx} = 0, &(x,t) \in  ( \beta, L )  \times (0, + \infty), \label{NG-E(3.3)}\nline
 \tau \eta_t(x,\rho,t)+\eta_\rho(x,\rho,t)=0,& (x,\rho,t)\in (\alpha,\beta)\times (0,1)\times (0, + \infty),\label{NG-E(3.4)}
\end{eqnarray}
with  the Dirichlet  boundary conditions
\begin{equation}\label{NG-E(3.5)}
u(0,t)= w(L,t) = 0, \quad  t\in (0,+\infty),
\end{equation}
with the following transmission conditions 
\begin{equation}\label{NG-E(3.6)}
\left\{
\begin{array}{ll}

\displaystyle{

u(\alpha,t ) = v(\alpha,t), \quad v (\beta , t) = w (\beta , t)}, &t\in  (0, + \infty), \nline

\displaystyle{\kappa_1 u_x (\alpha ,t)=\kappa_2 v_x (\alpha , t) + \delta_1 v_{xt} (\alpha,t ) + \delta_2 \eta_{x}(\alpha ,1,t) }, &t\in  (0, + \infty),\nline

\displaystyle{ \kappa_3 w_x (\beta ,t)=\kappa_2 v_x (\beta , t) + \delta_1 v_{xt} (\beta,t ) + \delta_2 \eta_{x}(\beta ,1,t)}, &t\in  (0, + \infty),
\end{array}
\right.
\end{equation}
and with the following initial conditions
\begin{equation} \label{NG-E(3.7)}
\left\{
\begin{array}{ll}

\displaystyle{ \left(u(x, 0 ),u_t(x,0)  \right) =\left(u_0(x), u_1 (x)\right)}, \quad & x\in  (0,\alpha),

\nline

\displaystyle{ \left(v(x, 0 ),v_t(x,0)  \right) =\left(v_0(x), v_1 (x)\right)}, \quad &x\in  (\alpha,\beta),\nline

\displaystyle{ \left(w(x, 0 ),w_t(x,0)  \right) =\left(w_0(x), w_1 (x)\right)}, \quad &x\in   (\beta,L),\nline
\displaystyle{ \eta(x,\rho,0) = f_0(x,-\rho\tau),}  & (x,\rho) \in  (\alpha,\beta) \times (0, 1),

\end{array}
\right.
\end{equation}
where the initial data $(u_0,u_1,v_0,v_1,w_0,w_1,f_0)$ belongs to a suitable space. To a strong solution of  System  \eqref{NG-E(3.1)}-\eqref{NG-E(3.7)}, we associate the energy defined by
\begin{equation*}
\begin{array}{ll}

\displaystyle{E(t)=\frac{1}{2}\int_0^\alpha  \left(|u_t(x,t)|^2+\kappa_1|u_x(x,t)|^2\right) dx+\frac{1}{2}\int_\alpha^\beta  \left(|v_t(x,t)|^2+\kappa_2|v_x(x,t)|^2\right) dx}\nline
\hspace{1cm}
\displaystyle{ +\frac{1}{2}\int_\beta^L \left(|w_t(x,t)|^2+\kappa_3|w_x(x,t)|^2\right) dx+\frac{\tau|\delta_2|}{2}\int_0^1\int_\alpha^\beta|\eta_x(x,\rho,t)|^2 d\rho\, dx.}

\end{array}
\end{equation*}
 Multiplying \eqref{NG-E(3.1)}, \eqref{NG-E(3.2)}, \eqref{NG-E(3.3)} and \eqref{NG-E(3.4)}$_x$   by $u_t$, $y_t$, $w_t$ and $|\delta_2|\eta_x$, integrating over $(0,\alpha),\ (\alpha,\beta)$, $(\beta,L)$ and $(\alpha,\beta)\times (0,1)$ respectively, taking the sum, then using by parts integration and  the boundary conditions in \eqref{NG-E(3.5)}-\eqref{NG-E(3.6)}, we get
\begin{equation*}
E'(t)=\left(-\delta_1+\frac{|\delta_2|}{2}\right)\int_\alpha^\beta |v_{xt}(x,t)|^2dx-\frac{|\delta_2|}{2}\int_\alpha^\beta |\eta_x(x,1,t)|^2dx-\delta_2\int_\alpha^\beta v_{xt}(x,t)\, \eta_x(x,1,t)dx.
\end{equation*}
Using Young's inequality for the third term in the right, we get 
\begin{equation*}
E'(t)\leq\left(-\delta_1+|\delta_2|\right)\int_\alpha^\beta |v_{xt}(x,t)|^2dx.
\end{equation*}
In the sequel, the assumption on $\delta_1$ and $\delta_2$  will ensure that 
\begin{equation}\tag{H1}
\delta_1>0,\quad \delta_2\in\mathbb{R}^*,\quad |\delta_2|<\delta_1.
\end{equation}
In this case,  the energies of the strong solutions satisfy $E'(t)\leq0.$ Hence, the System \eqref{NG-E(3.1)}-\eqref{NG-E(3.7)}  is dissipative in the sense that its energy is non increasing with respect to the time $t.$
\subsection{Well-posedness of the problem}\label{NG-S-3.1}
We start this part by formulating System \eqref{NG-E(3.1)}-\eqref{NG-E(3.7)} as an abstract Cauchy problem. For this aim, let us define  
\begin{equation*}
\begin{array}{ll}
\displaystyle{\mathbb{L}_*^2 =  L^2 (0,\alpha) \times L^2_* (\alpha, \beta) \times L^2(\beta,L)},\nline
\displaystyle{  \mathbb{H}^1_* = \{(u,v,w) \in  H^1 (0,\alpha) \times H^1_* (\alpha, \beta) \times H^1 (\beta, L)\ | \ u (0) = 0, \ u(\alpha) = v (\alpha), \ v(\beta) = w(\beta),\ w(L)=0 \}},\nline
\displaystyle{ \mathbb{H}^2 =  H^2 (0,\alpha) \times H^2 (\alpha, \beta) \times H^2 (\beta, L). }
\end{array}
\end{equation*}
Here we consider
\begin{equation*}
L^2_*(\alpha,\beta)=\left\{z\in L^2(\alpha,\beta)\  |\ \int_\alpha^\beta z\, dx=0\right\}\ \ \ \text{and} \ \ \ H^1_*(\alpha,\beta)=H^1(\alpha,\beta)\cap L^2_*(\alpha,\beta).
\end{equation*}
The spaces $\mathbb{L}_*^2$ and $ \mathbb{H}^1_*$  are obviously a Hilbert spaces equipped respectively with the norms
\begin{equation*}
\left\|(u,v,w)\right\|^2_{\mathbb{L}_*^2}=\int_0^\alpha|u|^2dx+  \int_\alpha^\beta |v|^2dx+\int_\beta^L |w|^2dx
\end{equation*}
and
\begin{equation*}
\left\|(u,v,w)\right\|^2_{\mathbb{H}^1_*}=\kappa_1\int_0^\alpha|u_x|^2dx+\kappa_2 \int_\alpha^\beta |v_x|^2dx+\kappa_3 \int_\beta^L |w_x|^2dx.
\end{equation*}
In addition by Poincaré inequality we can easily  verify that there exists $C>0$, such that
\begin{equation*}
\left\|(u,v,w)\right\|_{\mathbb{L}_*^2}\leq C \left\|(u,v,w)\right\|_{\mathbb{H}^1_*}.
\end{equation*}
Let us define the energy Hilbert space $\mathcal{H}_2$ by
\begin{equation*}
\mathcal{H}_2=\mathbb{H}^1_*\times \mathbb{L}_*^2\times L^2\left((0,1),H^1_*(\alpha,\beta)\right)
\end{equation*}
\noindent equipped with the following inner product
\begin{equation*}
\begin{array}{ll}
\displaystyle{
\langle {U},\tilde{{U}}\rangle _{{\mathcal{H}_2}} = \kappa_1 \int_{0}^{\alpha} u_x \overline{\tilde{u}}_x dx + \kappa_2 \int_{\alpha}^{\beta} v_x \overline{\tilde{v}}_x dx + \kappa_3 \int_{\beta}^{L} w_x \overline{\tilde{w}}_x dx
 }\nline \hspace{1.5cm}
\displaystyle{+
   \int_{0}^{\alpha} y \overline{\tilde{y}} dx +  \int_{\alpha}^{\beta} z \overline{\tilde{z}} dx +  \int_{\beta}^{L} \phi \overline{\tilde{\phi}} dx +\tau|\delta_2| \int_{0}^{1}\int_\alpha^\beta  \eta_x(x,\rho)\, \overline{\tilde{\eta}}_x(x,\rho) dx\, d \rho,}
   \end{array}
\end{equation*}
where $U=( u,v,w,y ,z, \phi, \eta(\cdot,\cdot) )\in{\mathcal{H}_2}$ and  $\tilde{{U}}= ( \tilde{u}, \tilde{ v}, \tilde{w}, \tilde{y}, \tilde{z}, \tilde{\phi}, \tilde{\eta}(\cdot,\cdot)) \in{\mathcal{H}_2}$.  We use  $\|{U}\|_{\mathcal{H}_2}$ to denote the corresponding norm. We  define the linear unbounded operator $\mathcal{A}_2: D(\mathcal{A}_2) \subset {\mathcal{H}_2}\longrightarrow {\mathcal{H}_2} $ by:
\begin{equation*}
\begin{array}{lll}
\displaystyle{D(\mathcal{A}_2) = \bigg\{(u,v,w,y,z,\phi, \eta(\cdot,\cdot))\in\mathcal{H}_2\   |\ 
\left(y,z,\phi\right)\mathbb{H}_*^1  } 
\\ \noalign{\medskip}\hspace{1.9cm}
\displaystyle{\ (u , \kappa_2 v  +\delta_1 z+\delta_2\eta(\cdot,1)  , w) \in \mathbb{H}^2,\  \kappa_2 v_x (\alpha) +\delta_1 z_x (\alpha)+\delta_2 \eta_x(\alpha,1)=\kappa_1 u _x (\alpha),}
 \\ \noalign{\medskip}\hspace{1.9cm}
\displaystyle{ \kappa_2 v_x (\beta) +\delta_1  z_x (\beta)+\delta_2 \eta_x(\beta,1)= \kappa_3 w _x (\beta),\ \eta,\,\eta_\rho\in L^2\left((0,1),H^1_*(\alpha,\beta)\right), \ z(\cdot) = \eta(\cdot,0)\bigg\}}
\end{array}
\end{equation*}
and for all $\mathbb{U}=(u,v,w, y,z, \phi, \eta(\cdot,\cdot))\in D(\mathcal{A}_2)$
\begin{equation*}
\mathcal{A}_2\mathbb{U} = \left( y , z ,\phi, \kappa_1 u_{xx}, \left(\kappa_2  v_{x} +\delta_1 z_{x}+\delta_2\eta_{x}(\cdot,1)\right)_x , \kappa_3 w_{xx} , - \tau^{-1} \eta_{\rho} (\cdot,\cdot)\right).
\end{equation*}
If $U=(u,v,w,u_t,v_t,w_t, \eta(\cdot,\cdot))$ is a regular solution of System \eqref{NG-E(3.1)}-\eqref{NG-E(3.7)}, then we transform this system into the following initial value problem
\begin{equation} \label{NG-E(3.8)}
\begin{cases}
U_t &= \mathcal{A}_2 U,\\
U(0)&= U_0,
\end{cases}
\end{equation}
where ${U}_0 = (u_0, v_0 , w_0 , u_1, v_1, w_1, f_0(\cdot,-\cdot \tau))\in{\mathcal{H}_2}.$ We now use  semigroup approach to establish well-posedness result for the  System \eqref{NG-E(3.1)}-\eqref{NG-E(3.7)}. We prove the following proposition.
\begin{Proposition}\label{NG-P-3.1} 
Under hypothesis {\rm(H1)}, the unbounded linear operator $\mathcal{A}_2$ is m-dissipative in the energy space ${\mathcal{H}_2}$.
\end{Proposition}
\begin{proof}
For all $U = (u,v,w,y,z, \phi,\eta(\cdot,\cdot)) \in D(\mathcal{A}_2),$ we have
\begin{equation*}
\begin{array}{lll}
\displaystyle{\text{Re}\left<\mathcal{A}_2 {U},{U}\right>_{\mathcal{H}_2} =
\kappa_1\text{Re}\int_{0}^{\alpha} \left(y_x \overline{{u}}_x+u_{xx} \overline{{y}}\right) dx +  \text{Re}\int_{\alpha}^{\beta} \left(\kappa_2 z_x \overline{{v}}_x+ (\kappa_2 v_x+\delta_1z_x+\delta_2 \eta_x(\cdot,1) )_{x} \overline{{z}}\right) dx 
 }\nline

   \hspace{2.3cm}
 \displaystyle{  + \kappa_3 \text{Re}\int_{\beta}^{L} \left(\phi_x \overline{{w}}_x+ w_{xx}\overline{{\phi}}\right) dx -|\delta_2| \text{Re}\int_\alpha^\beta\int_{0}^{1} \eta_{x\rho}(x,\rho)\, \overline{{\eta}}_x(x,\rho)dx\, d \rho.

 }
\end{array}
\end{equation*}
 Using by parts integration in the above equation, we get 
\begin{equation} \label{NG-E(3.9)}
\begin{array}{lll}
\displaystyle{\text{Re}\left<\mathcal{A}_2 {U},{U}_2\right>_{\mathcal{H}_2} =
-\delta_1\int_\alpha^\beta|z_x|^2dx-\delta_2 \text{Re}\int_\alpha^\beta\eta_x(\cdot,1) \overline{{z}}_xdx
+\frac{|\delta_2|}{2}\int_\alpha^\beta|\eta_x(x,0)|^2dx
}\nline

   \hspace{2.4cm}
 \displaystyle{ -\frac{|\delta_2|}{2}\int_\alpha^\beta|\eta_x(x,1)|^2dx -\kappa_1\text{Re}\left(u_x(0)\overline{y}(0)\right)+\kappa_3\text{Re}\left(w_x(L)\overline{\phi}(L)\right)}\nline

   \hspace{2.4cm}
 \displaystyle{
 
  +\text{Re}\left(\kappa_1 u_x(\alpha)\overline{y}(\alpha)-\kappa_2 v_x(\alpha)\overline{z}(\alpha)-\delta_1z_x(\alpha)\overline{z}(\alpha)-\delta_2 \eta_x(\alpha,1)\overline{z}(\alpha)\right) }\nline 
   \hspace{2.4cm}

 \displaystyle{
+ \text{Re}\left(\kappa_2v_x(\beta)\overline{z}(\beta)+\delta_1 z_x(\beta)\overline{z}(\beta)+\delta_2 \eta_x(\beta,1)\overline{z}(\beta)-\kappa_3 w_x(\beta)\overline{\phi}(\beta)\right).

 }
\end{array}
\end{equation}
Since ${U}\in D(\mathcal{A}_2)$, we have
\begin{equation*}
\left\{
\begin{array}{lll}
\displaystyle{y(0)=\phi(0)=0,\ y(\alpha)=z(\alpha),\ z(\beta)=\phi(\beta),\ z(x) = \eta(x,0),}\nline
  \displaystyle{ \kappa_1 u_x(\alpha)-\kappa_2 v_x(\alpha)-\delta_1z_x(\alpha)-\delta_2\eta_x(\alpha,1)=0,\ \kappa_2v_x(\beta)+\delta_1 z_x(\beta)+\delta_2\eta_x(\beta,1)-\kappa_3w_x(\beta)=0.}

\end{array}
\right.
\end{equation*}
Substituting  the above boundary conditions in \eqref{NG-E(3.9)}, then using  Young's inequality, we get
\begin{equation}\label{NG-E(3.10)}
\begin{array}{lll}
\displaystyle{\text{Re}\left<\mathcal{A}_2 {U},{U}\right>_{\mathcal{H}_2}} &=
\displaystyle{\left(-\delta_1+\frac{|\delta_2|}{2}\right)\int_\alpha^\beta|z_x|^2dx-\frac{|\delta_2|}{2}\int_\alpha^\beta|\eta_x(x,1)|^2dx-\delta_2 \text{Re}\int_\alpha^\beta\eta_x(\cdot,1) \overline{{z}}_xdx}\nline

&\leq \displaystyle{\left(-\delta_1+|\delta_2|\right)\int_\alpha^\beta|z_x|^2dx},
\end{array}
\end{equation}
hence under hypothesis {\rm(H1)},  we get 
\begin{equation*}
\text{Re}\left<\mathcal{A}_2 {U},{U}\right>_{\mathcal{H}_2}  \leq 0,
\end{equation*}
which implies that $\mathcal{A}_2$ is dissipative. To prove that $\mathcal{A}_2$ is m-dissipative, it is enough to prove that $0 \in \rho (\mathcal{A}_2)$ since $\mathcal{A}_2$ is a closed operator and $\overline{D(\mathcal{A}_2)} ={\mathcal{H}_2}$. Let $F= (f_1,f_2,f_3,f_4,f_5,f_6,f_7(\cdot,\cdot)) \in {\mathcal{H}_2}.$ We should prove that there exists a unique solution $U= (u,v,w, y,z , \phi,\eta(\cdot,\cdot)) \in D(\mathcal{A}_2)$ of the equation
\begin{equation*}
-\mathcal{A}_2U=F.
\end{equation*}
\noindent Equivalently, we consider the following system
\begin{eqnarray}
-y &=& f_1, \label{NG-E(3.11)}\\
-z &=& f_2, \label{NG-E(3.12)}\\
-\phi &=& f_3,\label{NG-E(3.13)}\\
- \kappa_1 u_{xx} &=& f_4, \label{NG-E(3.14)}\\
- \left(\kappa_2 v_x +\delta_1  z_x+\delta_2\eta_x(\cdot,1)\right)_{x} &=&  f_5, \label{NG-E(3.15)}\\
-\kappa_3 w_{xx} &=&  f_6, \label{NG-E(3.16)}\\
\eta_\rho (x,\rho)&=&\tau f_7(x,\rho).\label{NG-E(3.17)}
\end{eqnarray}
In addition, we consider the following boundary conditions  
\begin{eqnarray}
u(0) = 0, \quad u( \alpha ) = v (\alpha), \quad v (\beta) = w(\beta),\quad w(L)=0,\label{NG-E(3.18)}\\
\kappa_2 v_x (\alpha) +\delta_1  z_x (\alpha)+\delta_2 \eta_x(\alpha,1)=\kappa_1  u _x (\alpha), \quad  \kappa_2 v_x (\beta) +\delta_1 z_x (\beta)+\delta_2 \eta_x(\beta,1)=\kappa_3 w _x (\beta), \label{NG-E(3.19)}\\
\eta(\cdot,0)=z(\cdot).\label{NG-E(3.20)}
\end{eqnarray}
From \eqref{NG-E(3.11)}-\eqref{NG-E(3.13)} and the fact that $F\in \mathcal{H}$, we obtain $(y,z,\phi)\in \mathbb{H}^1_*$. Next, from \eqref{NG-E(3.12)}, \eqref{NG-E(3.20)} and the fact that $f_2\in H^1_*(\alpha,\beta)$, we get
\begin{equation*}
\eta(\cdot,0)=z(\cdot)=-f_2(\cdot)\in H^1_*(\alpha,\beta).
\end{equation*}
From the above equation and Equation \eqref{NG-E(3.17)}, we can determine
\begin{equation}\label{NG-E(3.21)}
\eta(x,\rho)=\tau \int_0^{\rho}f_7(x,\xi)\, d\xi  -f_2(x).
\end{equation}
Since $f_2\in H^1_*(\alpha,\beta)$ and $f_7\in L^2\left((0,1),H^1_*(\alpha,\beta)\right)$, then it is clear that $\eta,\, \eta_\rho\in L^2((0,1),H^1_*(0,1))$. Now,  let $\left(\varphi, \psi, \theta\right)\in \mathbb{H}^1_*$.  Multiplying Equations \eqref{NG-E(3.14)}, \eqref{NG-E(3.15)}, \eqref{NG-E(3.16)} by $\overline{\varphi}$, $\overline{\psi}$, $\overline{\theta}$, integrating over $(0,\alpha),\ (\alpha,\beta)$ and $(\beta,L)$ respectively, taking the sum, then using by parts integration, we get
\begin{equation}\label{NG-E(3.22)}
\begin{array}{ll}
\displaystyle{
\kappa_1\int_0^{\alpha} u_x\overline{\varphi}_xdx+\int_{\alpha}^{\beta} (\kappa_2 v_x+\delta_1 z_x+\delta_2\eta_x(\cdot,1))\overline{\psi}_xdx+\kappa_3\int_{\beta}^L
w_x\overline{\theta}_xdx+\kappa_1u_x(0)\overline{\varphi}(0)-\kappa_3w_x(L)\overline{\theta}(L)}\nline

\displaystyle{-\kappa_1 u_x(\alpha)\overline{\varphi}(\alpha)+
(\kappa_2v_x(\alpha)+\delta_1z_x(\alpha)+\delta_2\eta_x(\alpha,1))\overline{\psi}(\alpha)
-(\kappa_2v_x(\beta)+\delta_1z_x(\beta)+\delta_2\eta_x(\beta,1))\overline{\psi}(\beta)
}\nline
\displaystyle{
+\kappa_3 w_x(\beta)\overline{\theta}(\beta)=\int_0^{\alpha} f_4\overline{\varphi}dx+\int_{\alpha}^{\beta} f_5\overline{\psi}dx+\int_{\beta}^L
f_6\overline{\theta}dx.}
\end{array}
\end{equation}
From the fact that $\left(\varphi, \psi, \theta\right)\in \mathbb{H}^1_*,$ we have
\begin{equation*}
\varphi(0)=0,\quad \varphi(\alpha)=\psi(\alpha),\quad \theta(\beta)=\psi(\beta),\quad \theta(L)=0.
\end{equation*}
Inserting the above equation  in \eqref{NG-E(3.22)}, then  using \eqref{NG-E(3.12)}, \eqref{NG-E(3.19)} and  \eqref{NG-E(3.21)}, we get
\begin{equation}\label{NG-E(3.23)}
\begin{array}{ll}
\displaystyle{
\kappa_1\int_0^{\alpha} u_x\overline{\varphi}_xdx+\kappa_2\int_{\alpha}^{\beta}  v_x\overline{\psi}_xdx+\kappa_3\int_{\beta}^L
w_x\overline{\theta}_xdx}\nline

\displaystyle{
=\int_0^{\alpha} f_4\overline{\varphi}dx+\int_{\alpha}^{\beta} f_5\overline{\psi}dx+\int_{\beta}^L
f_6\overline{\theta}dx+\int_{\alpha}^{\beta} \left((\delta_1+\delta_2) (f_2)_x-\delta_2\tau \int_0^{1}(f_7(\cdot,\xi))_x\, d\xi   \right)\overline{\psi}_xdx.}
\end{array}
\end{equation}
We can easily verify that the left hand side of  \eqref{NG-E(3.23)} is a bilinear continuous coercive form on $\mathbb{H}^1_* \times \mathbb{H}^1_*$,  and the right hand side of   \eqref{NG-E(3.23)} is a linear continuous form on $\mathbb{H}^1_*$. Then, using Lax-Milgram theorem, we deduce that there exists $(u,v,w) \in \mathbb{H}^1_*  $ unique solution of the variational Problem \eqref{NG-E(3.23)}. Using standard arguments, we can show that $(u,\kappa_2v+\delta_1z+\delta_2\eta(\cdot,1) , w) \in \mathbb{H}^2$. Thus, from \eqref{NG-E(3.11)}-\eqref{NG-E(3.13)}, \eqref{NG-E(3.21)} and applying the classical elliptic regularity we deduce that $U=(u,v,w,y,z,\phi,\eta(\cdot,\cdot))\in D({\mathcal{A}_2})$.  The proof is thus complete.
\end{proof}
\noindent Thanks to Lumer-Philips theorem (see \cite{Pazy01}), we deduce that $\mathcal{A}_2$ generates a $C_0-$semigroup of contractions $e^{t\mathcal{A}_2}$ in ${\mathcal{H}_2}$ and therefore Problem \eqref{NG-E(3.1)}-\eqref{NG-E(3.7)} is well-posed. 
\subsection{Polynomial Stability}\label{NG-S-3.2}
 The main result in this subsection is the following theorem.
\begin{Theorem}\label{NG-T-3.2}  
Under hypothesis {\rm(H1)}, for all initial data ${U}_0 \in D(\mathcal{A}_2),$ there exists a constant $C>0$ independent of ${U}_0$ such that the energy of System \eqref{NG-E(3.1)}-\eqref{NG-E(3.7)} satisfies the following estimation
 \begin{equation*}
 E(t) \leq  \frac{C}{t^4} \left\|{U}_0\right\|^2_{D(\mathcal{A}_2)}, \quad \forall t>0.
 \end{equation*}
 \end{Theorem}
\noindent	According to Theorem \ref{NG-T-A.5} (part (ii)), we have to check if the following conditions hold,
\begin{equation}\label{NG-E(3.24)}
	\ i\mathbb{R}\subseteq\rho\left(\mathcal{A}_2\right) 
\end{equation}
and
\begin{equation}\label{NG-E(3.25)}
\sup_{\lambda\in\mathbb{R}}\left\|\left(i\lambda I-\mathcal{A}_2\right)^{-1}\right\|_{\mathcal{L}\left(\mathcal{H}_2\right)}=O\left(|\lambda|^{\frac{1}{2}}\right).
\end{equation}
The next proposition is a technical result to be used in the proof of Theorem \ref{NG-T-3.2}   given below.
\begin{Proposition}\label{NG-P-3.3}
Under hypothesis {\rm(H1)}, let  $(\lambda,{U}:=\left(u,v,w ,y,z,\phi,\eta(\cdot,\cdot)\right))\in  \mathbb{R}^*\times D\left(\mathcal{A}_2\right),$
such that
\begin{equation}\label{NG-E(3.26)}
 (i \lambda I - \mathcal{A}_2){U} ={F}:=(f_1,f_2,f_3,f_4,f_5,f_6,g(\cdot,\cdot))\in \mathcal{H}_2,   
\end{equation}
{\it i.e. }
 \begin{eqnarray}
i \lambda u -y =  f_1&\textrm{in} \quad H^1(0,\alpha), \label{NG-E(3.27)}\\
i \lambda v -z  =  f_2 &\textrm{in} \quad H^1_*(\alpha,\beta), \label{NG-E(3.28)}\\
i \lambda w - \phi =   f_3 & \textrm{in}\quad  H^1(\beta,L), \label{NG-E(3.29)}\\
i \lambda  y - \kappa_1 u_{xx} =     f_4& \textrm{in} \quad L^2(0,\alpha), \label{NG-E(3.30)}\\
i \lambda  z -\left(\kappa_2 v_x + \delta_1 z_x+\delta_2\eta_x(\cdot,1)\right)_{x}=    f_5& \textrm{in} \quad L^2_*(\alpha,\beta), \label{NG-E(3.31)}\\
i \lambda  \phi - \kappa_3 w_{xx}=     f_6&\textrm{in} \quad L^2(\beta,L), \label{NG-E(3.32)}\\
 \eta_\rho(\cdot,\cdot)+i\tau \lambda\eta(\cdot,\cdot)= \tau  g(\cdot,\cdot)&\textrm{in} \quad L^2\left((0,1),H^1_*(\alpha,\beta)\right).\label{NG-E(3.33)}
\end{eqnarray}
Then, we have the following inequality 
\begin{equation}\label{NG-E(3.34)}
\left\|{U}\right\|_{{\mathcal{H}_2}}^2 \leq K_1 \lambda^{-4}\left(|\lambda|+1\right)^6\left(\left\|F\right\|_{\mathcal{H}_2}\left\|U\right\|_{\mathcal{H}_2}+\left\|F\right\|_{\mathcal{H}_2}^2\right).
\end{equation}
In addition, if $|\lambda|\geq M>0,$ then we have
{\begin{equation}\label{NG-E(3.35)}
\left\|{U}\right\|_{{\mathcal{H}_2}}^2\leq K_2\left(\sqrt{M}+\frac{1}{\sqrt{M}}\right)^2|\lambda|^{\frac{1}{2}}\left(1+ |\lambda|^{-\frac{1}{2} }\right)^8\left(\left\|F\right\|_{\mathcal{H}_2}\left\|U\right\|_{\mathcal{H}_2}+\left\|F\right\|_{\mathcal{H}_2}^2\right).
\end{equation}}
Here and below we denote  by $K_j$   a positive constant number independent of $\lambda$. 
\end{Proposition}
Before stating  the  proof of Proposition \ref{NG-P-3.3}, { let $h \in C^1\left([\alpha,\beta]\right)$ such that 
\begin{equation*}
h(\alpha)=-h(\beta) = 1,\quad \max \limits_ {x \in [\alpha, \beta]} |h(x)|= C_{h}\quad \text{and}\quad \max \limits_ {x \in [\alpha, \beta]} |h'(x)| = C_{h'},
\end{equation*}
where $C_{h}$ and $C_{h'}$ are strictly positive constant numbers independent of $\lambda$. An example about $h$, we can take $h(x)=-\frac{2\left(x-\alpha\right)}{\beta-\alpha}+1$ to get 
\begin{equation*}
h(\alpha)=-h(\beta)=1,\quad h\in   C^1([\alpha,\beta]),\quad C_{h}=  1,\quad C_{h'}= \frac{2}{\beta-\alpha}.
\end{equation*}}  For the proof of Proposition \ref{NG-P-3.3}, we need the following lemmas.
\begin{lemma}\label{NG-L-3.4}
 Under hypothesis {\rm(H1)},  the solution $(u,v,w, y,z , \phi,\eta(\cdot,\cdot))\in D(\mathcal{A}_2)$ of Equation \eqref{NG-E(3.26)} satisfies the following estimations
\begin{eqnarray}
\int_\alpha^\beta|z_x|^2dx\leq  K_3\left\|F\right\|_{\mathcal{H}_2}\left\|U\right\|_{\mathcal{H}_2},\label{NG-E(3.36)}\nline
\int_\alpha^\beta|v_x|^2dx\leq 
 K_4 \lambda^{-2}\left(\left\|F\right\|_{\mathcal{H}_2}\left\|U\right\|_{\mathcal{H}_2}+\left\|F\right\|_{\mathcal{H}_2}^2\right) ,\label{NG-E(3.37)}\nline
\int_\alpha^\beta\int_0^1|\eta_x(x,\rho)|^2dx\, d\rho\leq  K_5\left(\left\|F\right\|_{\mathcal{H}_2}\left\|U\right\|_{\mathcal{H}_2}+\left\|F\right\|_{\mathcal{H}_2}^2\right),\label{NG-E(3.38)}\\
 \int_\alpha^\beta\left|\kappa_2 {v}_x + \delta_1 {z}_x+\delta_2{\eta}_x(\cdot,1)\right|^2 dx\leq K_6\left(1+\lambda^{-2}\right)\left(\left\|F\right\|_{\mathcal{H}_2}\left\|U\right\|_{\mathcal{H}_2}+\left\|F\right\|_{\mathcal{H}_2}^2\right),\label{NG-E(3.39)}
\end{eqnarray}
where
\begin{equation*}
\left\{
\begin{array}{ll}
K_3=\left(\delta_1-|\delta_2|\right)^{-1},\quad K_4=2\max\left(K_3,\kappa_2^{-1}\right),\nline
K_5=2\max\left(K_3,\tau |\delta_2|^{-1}\right),\quad K_6=3\max\left(\kappa_2^2K_4,\delta_1^2K_3+\delta_2^2K_5\right).
\end{array}
\right.
\end{equation*}
\end{lemma}
\begin{proof} First, taking the inner product of \eqref{NG-E(3.26)} with ${U}$ in $\mathcal{H}_2$, then using  hypothesis  {\rm (H1)}, arguing in the same way as \eqref{NG-E(3.10)}, we obtain
\begin{equation*}
\int_\alpha^\beta|z_x|^2dx\leq -\frac{1}{\delta_1-|\delta_2|}\text{Re}\left<\mathcal{A}_2{U} , {U}\right>_{{\mathcal{H}_2}}=\frac{1}{\delta_1-|\delta_2|}\text{Re}\left<F, {U}\right>_{{\mathcal{H}_2}} \leq \frac{1}{\delta_1-|\delta_2|}\left\|F\right\|_{\mathcal{H}_2}\left\|U\right\|_{\mathcal{H}_2},
\end{equation*}
hence we get \eqref{NG-E(3.36)}. Next, from \eqref{NG-E(3.28)}, \eqref{NG-E(3.36)} and the fact that $\kappa_2 \int_\alpha^\beta|(f_2)_x|^2dx\leq \left\|F\right\|_{\mathcal{H}_2}^2$, we obtain
\begin{equation*}
\begin{array}{ll}

\int_\alpha^\beta|v_x|^2dx&\leq 2 \lambda^{-2}\int_\alpha^\beta|z_x|^2dx+2\lambda^{-2}\int_\alpha^\beta|(f_2)_x|^2dx\nline &\leq 
2 \lambda^{-2}\left( K_3\left\|F\right\|_{\mathcal{H}_2}\left\|U\right\|_{\mathcal{H}_2}+\kappa_2^{-1}\left\|F\right\|_{\mathcal{H}_2}^2\right)\leq 2 \lambda^{-2}\max\left(K_3,\kappa_2^{-1}\right)\left(\left\|F\right\|_{\mathcal{H}_2}\left\|U\right\|_{\mathcal{H}_2}+\left\|F\right\|_{\mathcal{H}_2}^2\right) .

\end{array}
\end{equation*}
therefore  we get \eqref{NG-E(3.37)}. Now, from \eqref{NG-E(3.33)} and using the fact that $U\in D\left(\mathcal{A}_2\right)\, (\text{{\it i.e. }} \eta(\cdot,0)=z(\cdot))$, we obtain 
\begin{equation}\label{NG-E(3.40)}
\eta(x,\rho)=z(x) e^{-i\tau\lambda\, \rho}+\tau \int_0^\rho e^{i\tau\lambda\left(\xi-\rho\right)} g(x,\xi)d\xi \quad (x,\rho)\in (\alpha,\beta)\times (0,1),
\end{equation}
consequently, we obtain
\begin{equation*}
\int_\alpha^\beta\int_0^1|\eta_x(x,\rho)|^2dx\, d\rho\leq2\int_\alpha^\beta|z_x|^2dx+2\tau^2 \int_\alpha^\beta\int_0^1  \left|g_x(x,\xi)\right|^2d\xi dx.
\end{equation*}
Inserting \eqref{NG-E(3.36)} in the above equation, then using the fact that $\tau |\delta_2|\int_\alpha^\beta\int_0^1  \left|g_x(x,\xi)\right|^2d\xi dx \leq \left\|F\right\|_{\mathcal{H}_2}^2$, we obtain
\begin{equation*}
\begin{array}{ll}
\int_\alpha^\beta\int_0^1|\eta_x(x,\rho)|^2dx\, d\rho&\leq 2K_3\left\|F\right\|_{\mathcal{H}_2}\left\|U\right\|_{\mathcal{H}_2}+2\tau|\delta_2|^{-1} \left\|F\right\|_{\mathcal{H}_2}^2
\nline&
\leq 2\max\left(K_3,\tau|\delta_2|^{-1}\right)\left(\left\|F\right\|_{\mathcal{H}_2}\left\|U\right\|_{\mathcal{H}_2}+\left\|F\right\|_{\mathcal{H}_2}^2\right),
\end{array}
\end{equation*}
hence we get \eqref{NG-E(3.38)}. On the other hand, from \eqref{NG-E(3.40)}, we get
\begin{equation*}
\eta_x(x,1)=z_x(x) e^{-i\tau\lambda}+\tau \int_0^1 e^{i\tau\lambda\left(\xi-1\right)} g_x(x,\xi)d\xi \quad x\in (\alpha,\beta).
\end{equation*}
From the above equation and \eqref{NG-E(3.36)}, we obtain
\begin{equation*}
\begin{array}{ll}

\int_\alpha^\beta|\eta_x(x,1)|^2dx\, &\leq2\int_\alpha^\beta|z_x|^2dx+2\tau^2 \int_\alpha^\beta\int_0^1  \left|g_x(x,\xi)\right|^2d\xi dx\nline
&\leq 2K_3\left\|F\right\|_{\mathcal{H}_2}\left\|U\right\|_{\mathcal{H}_2}+2\tau|\delta_2|^{-1} \left\|F\right\|_{\mathcal{H}_2}^2

\nline & \leq 2\max\left(K_3,\tau|\delta_2|^{-1}\right)\left(\left\|F\right\|_{\mathcal{H}_2}\left\|U\right\|_{\mathcal{H}_2}+\left\|F\right\|_{\mathcal{H}_2}^2\right).
\end{array}
\end{equation*}
Finally, from \eqref{NG-E(3.36)}, \eqref{NG-E(3.37)} and the above inequality, we get
\begin{equation*}
\begin{array}{lll}

\int_\alpha^\beta\left|\kappa_2 {v}_x + \delta_1 {z}_x+\delta_2{\eta}_x(\cdot,1)\right|^2 dx&\leq 3\kappa_2^2\int_\alpha^\beta|v_x|^2dx+3\delta_1^2\int_\alpha^\beta|z_x|^2dx+3\delta_2^2\int_\alpha^\beta|\eta_x(\cdot,1)|^2dx\nline
&\leq 3\left(\kappa_2^2K_4 \lambda^{-2}+\delta_1^2K_3+\delta_2^2K_5\right)\left(\left\|F\right\|_{\mathcal{H}_2}\left\|U\right\|_{\mathcal{H}_2}+\left\|F\right\|_{\mathcal{H}_2}^2\right)\nline
&\leq 3\max\left(\kappa_2^2K_4,\delta_1^2K_3+\delta_2^2K_5\right)\left(1+\lambda^{-2}\right)\left(\left\|F\right\|_{\mathcal{H}_2}\left\|U\right\|_{\mathcal{H}_2}+\left\|F\right\|_{\mathcal{H}_2}^2\right),

\end{array}
\end{equation*}
hence we get \eqref{NG-E(3.39)}.
\end{proof}
\begin{lemma}\label{NG-L-3.5}
 Under hypothesis {\rm(H1)},  for all $s_1,\, s_2\in\mathbb{R}$ and $r_1,\, r_2\in \mathbb{R}^*_+$, the solution $(u,v,w, y,z , \phi,\eta(\cdot,\cdot))\in D(\mathcal{A}_2)$ of Equation \eqref{NG-E(3.26)} satisfies the following estimations
\begin{equation}\label{NG-E(3.41)}
|z(\beta)|^2+|z(\alpha)|^2  \leq\left(C_{h'}+\frac{|\lambda|^{\frac{1}{2}-s_1}}{r_1}\right)\int_\alpha^\beta |z|^2dx +K_7 r_1 C_{h}^2|\lambda|^{-\frac{1}{2}+s_1} \left(\left\|F\right\|_{\mathcal{H}_2}\left\|U\right\|_{\mathcal{H}_2}+\left\|F\right\|_{\mathcal{H}_2}^2\right)
\end{equation}
and
\begin{equation}\label{NG-E(3.42)}
\begin{array}{lll}

\left|\kappa_2 {v}_x(\beta) + \delta_1 {z}_x(\beta)+\delta_2{\eta}_x(\beta,1)\right|^2+\left|\kappa_2 {v}_x(\alpha) + \delta_1{z}_x(\alpha)+\delta_2{\eta}_x(\alpha,1)\right|^2
\nline
\leq \frac{|\lambda|^{\frac{3}{2}-s_2}}{r_2}\int_\alpha^\beta|z|^2dx
+K_8 \left(C_{h'}+C_{h}+r_2C_{h}^2|\lambda|^{\frac{1}{2}+s_2 }\right)\left(1+\lambda^{-2}\right)\left(\left\|F\right\|_{\mathcal{H}_2}\left\|U\right\|_{\mathcal{H}_2}+\left\|F\right\|_{\mathcal{H}_2}^2\right),

\end{array}
\end{equation}
where
\begin{equation*}
K_7=2\left(K_4+\kappa_2^{-1}\right),\quad  K_8=K_6+1.
\end{equation*}
\end{lemma}
\begin{proof} First,  from  Equation \eqref{NG-E(3.28)}, we have
\begin{equation*}
 -z_x  =  (f_2)_x-i \lambda v_x.
\end{equation*}
Multiplying the above equation by $2h \overline{z}$, integrating over $(\alpha,\beta)$ and taking the real parts, then using  by parts integration and the fact that $h(\alpha)=-h(\beta) = 1$, we get
\begin{equation}\label{NG-E(3.43)}
\begin{array}{lll}

\displaystyle{|z(\beta)|^2+|z(\alpha)|^2} & =\displaystyle{-\int_\alpha^\beta h'\,  |z|^2dx -2\text{Re}\left\{\int_\alpha^\beta h \left(i\lambda v_x-(f_2)_x\right) \overline{z}dx\right\}}
\nline
&\leq
\displaystyle{

C_{h' }\int_\alpha^\beta |z|^2dx +2C_{h }|\lambda|\int_\alpha^\beta |v_x||{z}|dx+2C_{h }\int_\alpha^\beta|(f_2)_x||{z}|dx
}.

\end{array}
\end{equation}
On the other hand, for all $s_1\in\mathbb{R}$ and $r_1\in\mathbb{R}^*_+$, we have 
\begin{equation*}
2 C_{h }|\lambda|    |v_x| |z|\leq \frac{|\lambda|^{\frac{1}{2}-s_1} |z|^2}{2r_1}+2r_1 C_{h }^2|\lambda|^{\frac{3}{2}+s_1} \, |v_x|^2\ \ \ \text{and} \ \ \ 2C_{h } |(f_2)_x| |z|\leq  \frac{|\lambda|^{\frac{1}{2}-s_1}|z|^2}{2r_1} +2r_1C_{h }^2 |\lambda|^{-\frac{1}{2}+s_1}|(f_2)_x|^2.
\end{equation*}
Inserting the above equation in \eqref{NG-E(3.43)}, then using \eqref{NG-E(3.37)} and the fact that $\int_\alpha^\beta|(f_2)_x|^2dx\leq \kappa_2^{-1} \left\|F\right\|_{\mathcal{H}_2}^2,$ we get
\begin{equation*}
\begin{array}{ll}

|z(\beta)|^2+|z(\alpha)|^2
\nline \leq \left(C_{h' }+\frac{|\lambda|^{\frac{1}{2}-s_1}}{r_1}\right)\int_\alpha^\beta |z|^2dx +2r_1C_{h }^2|\lambda|^{s_1}\left(|\lambda|^{\frac{3}{2}} \int_\alpha^\beta |v_x|^2dx+ |\lambda|^{-\frac{1}{2}}\int_\alpha^\beta|(f_2)_x|^2dx\right)

\nline  \leq\left(C_{h' }+\frac{|\lambda|^{\frac{1}{2}-s_1}}{r_1}\right)\int_\alpha^\beta |z|^2dx +2r_1C_{h }^2|\lambda|^{s_1-\frac{1}{2}}\left( K_4 \left(\left\|F\right\|_{\mathcal{H}_2}\left\|U\right\|_{\mathcal{H}_2}+\left\|F\right\|_{\mathcal{H}_2}^2\right)+ \kappa_2^{-1}\left\|F\right\|_{\mathcal{H}_2}^2\right)
\nline 
\leq\left(C_{h' }+\frac{|\lambda|^{\frac{1}{2}-s_1}}{r_1}\right)\int_\alpha^\beta |z|^2dx +2r_1C_{h }^2|\lambda|^{s_1-\frac{1}{2}}\left(K_4+\kappa_2^{-1}\right)\left(\left\|F\right\|_{\mathcal{H}_2}\left\|U\right\|_{\mathcal{H}_2}+\left\|F\right\|_{\mathcal{H}_2}^2\right),

\end{array}
\end{equation*}
hence we get \eqref{NG-E(3.41)}. Next, multiplying  Equation \eqref{NG-E(3.31)} by $2h \left(\kappa_2 \overline{v}_x + \delta_1 \overline{z}_x+\delta_2\overline{\eta}_x(\cdot,1)\right)$, integrating over $(\alpha,\beta)$ and taking the real parts, then using  by parts integration, we get
\begin{equation*}
\begin{array}{lll}

\left|\kappa_2 {v}_x(\beta) + \delta_1 {z}_x(\beta)+\delta_2{\eta}_x(\beta,1)\right|^2+\left|\kappa_2 {v}_x(\alpha) + \delta_1 {z}_x(\alpha)+\delta_2{\eta}_x(\alpha,1)\right|^2\nline
=-\int_\alpha^\beta h' \left|\kappa_2 {v}_x + \delta_1 {z}_x+\delta_2{\eta}_x(\cdot,1)\right|^2dx

\displaystyle{+2\text{Re}\int_\alpha^\beta h\left(f_5-i\lambda\, z\right)\left(\kappa_2 \overline{v}_x + \delta_1 \overline{z}_x+\delta_2\overline{\eta}_x(\cdot,1)\right)dx},

\end{array}
\end{equation*}
consequently, we have
\begin{equation}\label{NG-E(3.44)}
\begin{array}{lll}

\displaystyle{

\left|\kappa_2 {v}_x(\beta) + \delta_1{z}_x(\beta)+\delta_2{\eta}_x(\beta,1)\right|^2+\left|\kappa_2 {v}_x(\alpha) + \delta_1{z}_x(\alpha)+\delta_2{\eta}_x(\alpha,1)\right|^2

}\nline

\displaystyle{
\leq C_{h'}\int_\alpha^\beta \left|\kappa_2 {v}_x + \delta_1 {z}_x+\delta_2{\eta}_x(\cdot,1)\right|^2dx

+2C_{h}\int_\alpha^\beta|f_5|\left|\kappa_2 {v}_x + \delta_1 {z}_x+\delta_2{\eta}_x(\cdot,1)\right|dx

}

\nline

\ \ \ \displaystyle{
+2C_{h}|\lambda|\int_\alpha^\beta|z|\left|\kappa_2 {v}_x + \delta_1 {z}_x+\delta_2{\eta}_x(\cdot,1)\right|dx}.

\end{array}
\end{equation}
On the other hand, for all $s_2\in\mathbb{R}$ and $r_2\in\mathbb{R}^*_+$, we have
\begin{equation*}
\left\{
\begin{array}{ll}

2  C_{h} |f_5| |\left|\kappa_2 {v}_x + \delta_1 {z}_x+\delta_2{\eta}_x(\cdot,1)\right|\leq C_{h} |f_5|^2+ C_{h}   \left|\kappa_2 {v}_x + \delta_1 {z}_x+\delta_2{\eta}_x(\cdot,1)\right|^2,\nline
2C_{h} |\lambda|   |z||\left|\kappa_2 {v}_x + \delta_1 {z}_x+\delta_2{\eta}_x(\cdot,1)\right|\leq \frac{|\lambda|^{\frac{3}{2}-s_2}}{r_2}|z|^2+r_2C_{h}^2|\lambda|^{\frac{1}{2}+s_2 } |\left|\kappa_2 {v}_x + \delta_1 {z}_x+\delta_2{\eta}_x(\cdot,1)\right|^2.

\end{array}
\right.
\end{equation*}
Inserting the above equation in \eqref{NG-E(3.44)}, we get 
\begin{equation*}
\begin{array}{lll}

\displaystyle{

\left|\kappa_2 {v}_x(\beta) + \delta_1 {z}_x(\beta)+\delta_2{\eta}_x(\beta,1)\right|^2+\left|\kappa_2 {v}_x(\alpha) + \delta_1 {z}_x(\alpha)+\delta_2{\eta}_x(\alpha,1)\right|^2\leq \frac{|\lambda|^{\frac{3}{2}-s_2}}{r_2}\int_\alpha^\beta|z|^2dx

}\nline

\displaystyle{+\left(C_{h'}+C_{h}+r_2C_{h}^2|\lambda|^{\frac{1}{2}+s_2 }\right)\left[\int_\alpha^\beta|\left|\kappa_2 {v}_x + \delta_1 {z}_x+\delta_2{\eta}_x(\cdot,1)\right|^2dx+\int_\alpha^\beta|f_5|^2dx\right].

}

\end{array}
\end{equation*}
Inserting \eqref{NG-E(3.39)} in the above equation, then using  the fact that 
\begin{equation*}
\int_\alpha^\beta|f_5|^2dx\leq \left\|F\right\|^2_{\mathcal{H}_2}\leq \left(1+\lambda^{-2}\right)\left(\left\|F\right\|_{\mathcal{H}_2}\left\|U\right\|_{\mathcal{H}_2}+\left\|F\right\|_{\mathcal{H}_2}^2\right),
\end{equation*}
 we get 
 \begin{equation*}
\begin{array}{lll}

\left|\kappa_2 {v}_x(\beta) + \delta_1 {z}_x(\beta)+\delta_2{\eta}_x(\beta,1)\right|^2+\left|\kappa_2 {v}_x(\alpha) + \delta_1{z}_x(\alpha)+\delta_2{\eta}_x(\alpha,1)\right|^2\nline\leq \frac{|\lambda|^{\frac{3}{2}-s_2}}{r_2}\int_\alpha^\beta|z|^2dx

+ (K_6+1) \left(C_{h'}+C_{h}+r_2C_{h}^2|\lambda|^{\frac{1}{2}+s_2 }\right)\left(1+\lambda^{-2}\right)\left(\left\|F\right\|_{\mathcal{H}_2}\left\|U\right\|_{\mathcal{H}_2}+\left\|F\right\|_{\mathcal{H}_2}^2\right),

\end{array}
\end{equation*}
hence we get \eqref{NG-E(3.42)}.
\end{proof}
\begin{lemma}\label{NG-L-3.6}
 Under hypothesis {\rm(H1)},  for all $s_1,\, s_2\in\mathbb{R}$,  the solution $(u,v,w, y,z , \phi,\eta(\cdot,\cdot))\in D(\mathcal{A}_2)$ of Equation \eqref{NG-E(3.26)} satisfies the following estimation
 {\begin{equation}\label{NG-E(3.45)}
\begin{array}{ll}
\int_{0}^{\alpha}  |y|^2  dx+ \kappa_1 \int_{0}^{\alpha}  |u_{x}|^2  dx +\int_{\beta}^{L}  |\phi|^2  dx+ \kappa_3 \int_{\beta}^{L}  |w_{x}|^2  dx

\nline\leq K_9\left(1+|\lambda|^{\frac{1}{2}-s_1}+|\lambda|^{\frac{3}{2}-s_2}\right)\int_\alpha^\beta |z|^2dx

\nline+K_{10}\left(1+|\lambda|^{-\frac{1}{2}+s_1}+|\lambda|^{\frac{1}{2}+s_2 }\right)\left(1+\lambda^{-2}\right)\left(\left\|F\right\|_{\mathcal{H}_2}\left\|U\right\|_{\mathcal{H}_2}+\left\|F\right\|_{\mathcal{H}_2}^2\right),
\end{array}
\end{equation}
where
\begin{equation*}
K_9=\max\left[\max(\alpha,L-\beta),\max\left(\kappa^{-1}_1\alpha,\kappa^{-1}_3\,(L-\beta)\right)\right]\max\left(C_{h'},1\right)
\end{equation*}
and 
\begin{equation*}
K_{10}=2\max\left[K_7 C_{h}^2\max(\alpha,L-\beta) ,\ K_8\max\left(\kappa^{-1}_1\alpha, \kappa^{-1}_3\,(L-\beta)\right)\max\left(C_{h'}+C_{h},C_{h}^2\right),\ 4\left(\alpha\kappa_1^{-\frac{1}{2}}+(L-\beta)\kappa_3^{-\frac{1}{2}}\right)\right]. 
\end{equation*}}
\end{lemma}
\begin{proof} First, multiplying Equation \eqref{NG-E(3.30)} by $2x \overline{u}_x$, integrating over $(0,\alpha)$ and taking the real parts, then using  by parts integration, we get
\begin{equation}\label{NG-E(3.46)}
2\text{Re}\left\{i \lambda \int_{0}^{\alpha} x y \overline{u}_x dx\right\}+ \kappa_1 \int_{0}^{\alpha}  |u_{x}|^2  dx   =  \kappa_1\alpha |u_x(\alpha)|^2+2\text{Re}\left\{\int_{0}^{\alpha} x\, f_4\, \overline{u}_x dx\right\}.
\end{equation}
From \eqref{NG-E(3.27)}, we deduce that
\begin{equation*}
i \lambda \overline{u}_x = -\overline{y}_x - (\overline{f_1})_x .
\end{equation*}
Inserting the above result in \eqref{NG-E(3.46)}, then using  by parts integration, we get
\begin{equation*}
\int_{0}^{\alpha}  |y|^2  dx+ \kappa_1 \int_{0}^{\alpha}  |u_{x}|^2  dx   =  \kappa_1\alpha |u_x(\alpha)|^2+\alpha |y(\alpha)|^2+2\text{Re}\left\{\int_{0}^{\alpha} x\, f_4\, \overline{u}_x dx\right\}+2\text{Re}\left\{ \int_{0}^{\alpha} x\, y\,  (\overline{f_1})_x dx\right\},
\end{equation*}
consequently, we get
\begin{equation}\label{NG-E(3.47)}
\int_{0}^{\alpha}  |y|^2  dx+ \kappa_1 \int_{0}^{\alpha}  |u_{x}|^2  dx   \leq   \kappa_1\alpha |u_x(\alpha)|^2+\alpha |y(\alpha)|^2+2\alpha\left(\int_{0}^{\alpha}  |{u}_x| |f_4| dx+ \int_{0}^{\alpha} | y|  |({f_1})_x| dx\right).
\end{equation}
Using Cauchy Schwarz inequality, we get
\begin{equation}\label{NG-E(3.48)}
\int_{0}^{\alpha}  |{u}_x| |f_4| dx+ \int_{0}^{\alpha} | y|  |({f_1})_x| dx\leq 2\kappa_1^{-\frac{1}{2}}\left\|F\right\|_{\mathcal{H}_2}\left\|U\right\|_{\mathcal{H}_2}.
\end{equation}
On the other hand, since $U\in D\left(\mathcal{A}_2\right)$, we have
\begin{equation}\label{NG-E(3.49)}
\left\{\begin{array}{ll}

\displaystyle{

\kappa_1 |u _x (\alpha)|= |\kappa_2 v_x (\alpha) +\delta_1 z_x (\alpha)+\delta_2 \eta_x(\alpha,1)|,\ |y(\alpha)|=|z(\alpha)|,}\nline

\displaystyle{

\kappa_3 |w _x (\beta)|= |\kappa_2 v_x (\beta) +\delta_1 z_x (\beta)+\delta_2 \eta_x(\beta,1)|,\ |\phi(\beta)|=|z(\beta)| }.

 \end{array}\right.
\end{equation}
Substituting \eqref{NG-E(3.48)} and the boundary conditions \eqref{NG-E(3.49)} at $x=\alpha$  in \eqref{NG-E(3.47)},  we obtain  
\begin{equation}\label{NG-E(3.50)}
\int_{0}^{\alpha}  |y|^2  dx+ \kappa_1 \int_{0}^{\alpha}  |u_{x}|^2  dx   \leq 
\alpha \kappa_1^{-1}|\kappa_2 v_x (\alpha) +\delta_1 z_x (\alpha)+\delta_2 \eta_x(\alpha,1)|^2+\alpha  |z(\alpha)|^2+4\alpha \kappa_1^{-\frac{1}{2}}\left\|F\right\|_{\mathcal{H}_2}\left\|U\right\|_{\mathcal{H}_2}.
\end{equation}
Next,  by the same way, we multiply equation \eqref{NG-E(3.32)} by $2(x - L) \overline{w}_x$ and integrate over $(\beta, L),$ then we use \eqref{NG-E(3.29)}. Arguing in the same way as \eqref{NG-E(3.47)}, we get
\begin{equation*}
\int_{\beta}^{L}  |\phi|^2  dx+ \kappa_3 \int_{\beta}^{L}  |w_{x}|^2  dx \leq 
\kappa_3(L-\beta) |w_x(\beta)|^2+(L-\beta)  |\phi(\beta)|^2+4(L-\beta)\kappa_3^{-\frac{1}{2}}\left\|F\right\|_{\mathcal{H}_2}\left\|U\right\|_{\mathcal{H}_2}.
\end{equation*}
Substituting  the boundary conditions \eqref{NG-E(3.49)} at $x=\beta$  in the above equation,  we obtain
\begin{equation*}
\int_{\beta}^{L}  |\phi|^2  dx+ \kappa_3 \int_{\beta}^{L}  |w_{x}|^2  dx   \leq 
(L-\beta)\left[\kappa_3^{-1} |\kappa_2 v_x (\beta) +\delta_1 z_x (\beta)+\delta_2 \eta_x(\beta,1)|^2+|z(\beta)|^2+4\kappa_3^{-\frac{1}{2}}\left\|F\right\|_{\mathcal{H}_2}\left\|U\right\|_{\mathcal{H}_2}\right].
\end{equation*}
Now, adding the above equation and \eqref{NG-E(3.50)}, we get
\begin{equation*}
\begin{array}{ll}
\int_{0}^{\alpha}  |y|^2  dx+ \kappa_1 \int_{0}^{\alpha}  |u_{x}|^2  dx +\int_{\beta}^{L}  |\phi|^2  dx+ \kappa_3 \int_{\beta}^{L}  |w_{x}|^2  dx\leq \max(\alpha,L-\beta)\left(|z(\alpha)|^2+|z(\beta)|^2\right)
\nline+\max\left(\kappa^{-1}_1\alpha,\kappa^{-1}_3\,(L-\beta)\right)\left(|\kappa_2 v_x (\alpha) +\delta_1 z_x (\alpha)+\delta_2 \eta_x(\alpha,1)|^2+|\kappa_2 v_x (\beta) +\delta_1 z_x (\beta)+\delta_2 \eta_x(\beta,1)|^2\right)
\nline
+4\left(\alpha\kappa_1^{-\frac{1}{2}}+(L-\beta)\kappa_3^{-\frac{1}{2}}\right)\left\|F\right\|_{\mathcal{H}_2}\left\|U\right\|_{\mathcal{H}_2}.
\end{array}
\end{equation*}
Inserting \eqref{NG-E(3.41)} and \eqref{NG-E(3.42)} with $r_1=r_2=1$ in the above estimation, we get
\begin{equation*}
\begin{array}{ll}
\int_{0}^{\alpha}  |y|^2  dx+ \kappa_1 \int_{0}^{\alpha}  |u_{x}|^2  dx +\int_{\beta}^{L}  |\phi|^2  dx+ \kappa_3 \int_{\beta}^{L}  |w_{x}|^2  dx

\nline\leq \max\left[\max(\alpha,L-\beta),\max\left(\kappa^{-1}_1\alpha,\kappa^{-1}_3\,(L-\beta)\right)\right]\left(C_{h'}+|\lambda|^{\frac{1}{2}-s_1}+|\lambda|^{\frac{3}{2}-s_2}\right)\int_\alpha^\beta |z|^2dx

\nline+ \max(\alpha,L-\beta) K_7 C_{h}^2|\lambda|^{-\frac{1}{2}+s_1} \left(\left\|F\right\|_{\mathcal{H}_2}\left\|U\right\|_{\mathcal{H}_2}+\left\|F\right\|_{\mathcal{H}_2}^2\right)

\nline+\max\left(\kappa^{-1}_1\alpha,\kappa^{-1}_3\,(L-\beta)\right)K_8 \left(C_{h'}+C_{h}+C_{h}^2|\lambda|^{\frac{1}{2}+s_2 }\right)\left(1+\lambda^{-2}\right)\left(\left\|F\right\|_{\mathcal{H}_2}\left\|U\right\|_{\mathcal{H}_2}+\left\|F\right\|_{\mathcal{H}_2}^2\right)
\nline
+4\left(\alpha\kappa_1^{-\frac{1}{2}}+(L-\beta)\kappa_3^{-\frac{1}{2}}\right)\left\|F\right\|_{\mathcal{H}_2}\left\|U\right\|_{\mathcal{H}_2}.
\end{array}
\end{equation*}
In the above equation, using the fact that
\begin{equation*}
\left\{\begin{array}{ll}
\left\|F\right\|_{\mathcal{H}_2}\left\|U\right\|_{\mathcal{H}_2}\leq \left(1+\lambda^{-2}\right)\left(\left\|F\right\|_{\mathcal{H}_2}\left\|U\right\|_{\mathcal{H}_2}+\left\|F\right\|_{\mathcal{H}_2}^2\right),\nline

\left\|F\right\|_{\mathcal{H}_2}\left\|U\right\|_{\mathcal{H}_2}+\left\|F\right\|_{\mathcal{H}_2}^2\leq \left(1+\lambda^{-2}\right)\left(\left\|F\right\|_{\mathcal{H}_2}\left\|U\right\|_{\mathcal{H}_2}+\left\|F\right\|_{\mathcal{H}_2}^2\right),

\end{array}\right.
\end{equation*}
 we get 
 {\begin{equation*}
\begin{array}{ll}
\int_{0}^{\alpha}  |y|^2  dx+ \kappa_1 \int_{0}^{\alpha}  |u_{x}|^2  dx +\int_{\beta}^{L}  |\phi|^2  dx+ \kappa_3 \int_{\beta}^{L}  |w_{x}|^2  dx

\nline\leq \max\left[\max(\alpha,L-\beta),\max\left(\kappa^{-1}_1\alpha,\kappa^{-1}_3\,(L-\beta)\right)\right]\max\left(C_{h'},1\right)\left(1+|\lambda|^{\frac{1}{2}-s_1}+|\lambda|^{\frac{3}{2}-s_2}\right)\int_\alpha^\beta |z|^2dx

\nline+K_7 C_{h}^2 \max(\alpha,L-\beta) |\lambda|^{-\frac{1}{2}+s_1} \left(1+\lambda^{-2}\right)\left(\left\|F\right\|_{\mathcal{H}_2}\left\|U\right\|_{\mathcal{H}_2}+\left\|F\right\|_{\mathcal{H}_2}^2\right)

\nline+K_8\max\left(\kappa^{-1}_1\alpha,\kappa^{-1}_3\,(L-\beta)\right)\max\left(C_{h'}+C_{h},C_{h}^2\right) \left(1+|\lambda|^{\frac{1}{2}+s_2 }\right)\left(1+\lambda^{-2}\right)\left(\left\|F\right\|_{\mathcal{H}_2}\left\|U\right\|_{\mathcal{H}_2}+\left\|F\right\|_{\mathcal{H}_2}^2\right)
\nline
+4\left(\alpha\kappa_1^{-\frac{1}{2}}+(L-\beta)\kappa_3^{-\frac{1}{2}}\right)\left(1+\lambda^{-2}\right)\left(\left\|F\right\|_{\mathcal{H}_2}\left\|U\right\|_{\mathcal{H}_2}+\left\|F\right\|_{\mathcal{H}_2}^2\right),
\end{array}
\end{equation*}
 hence,  we get \eqref{NG-E(3.45)}.
 }
\end{proof}
\begin{lemma}\label{NG-L-3.7} 
Under hypothesis {\rm(H1)}, for all $s_1,\, s_2,\ s_3\in\mathbb{R}$ and $r_1,\, r_2,\, r_3\in \mathbb{R}^*_+$,  the solution $(u,v,w, y,z , \phi,\eta(\cdot,\cdot))\in D(\mathcal{A}_2)$ of Equation \eqref{NG-E(3.26)} satisfies the following estimations
{\begin{equation}\label{NG-E(3.51)}
\begin{array}{lll}
\left\|{U}\right\|_{{\mathcal{H}_2}}^2 \leq

K_{11}\left(1+|\lambda|^{\frac{1}{2}-s_1}+|\lambda|^{\frac{3}{2}-s_2}\right)\int_\alpha^\beta |z|^2dx 

\nline

\hspace{1.2cm}
+K_{12} \left(1+|\lambda|^{\frac{1}{2}+s_2 }+|\lambda|^{-\frac{1}{2}+s_1}\right)\left(1+\lambda^{-2}\right)\left(\left\|F\right\|_{\mathcal{H}_2}\left\|U\right\|_{\mathcal{H}_2}+\left\|F\right\|_{\mathcal{H}_2}^2\right)

\end{array}
\end{equation}
}
and
\begin{equation}\label{NG-E(3.52)}
R_{1,\lambda}\int_\alpha^\beta|z|^2dx\leq K_{13}\, R_{2,\lambda}\, \left(1+\lambda^{-2}\right)\left(\left\|F\right\|_{\mathcal{H}_2}\left\|U\right\|_{\mathcal{H}_2}+\left\|F\right\|_{\mathcal{H}_2}^2\right),
\end{equation}
such that
\begin{equation*}
\left\{
\begin{array}{ll}
R_{1,\lambda}=1- \frac{1}{2}\left(\frac{|\lambda|^{-s_3-s_2}}{r_2r_3}+\frac{r_3|\lambda|^{-s_1+s_3}}{r_1}+r_3C_{h'}|\lambda|^{s_3-\frac{1}{2}} \right),\nline
R_{2,\lambda}=C_{h}^2 |\lambda|^{-1 }(r_1 r_3 |\lambda|^{s_1+s_3}+r_2 r_3^{-1}|\lambda|^{s_2-s_3 })+ r_3^{-1}|\lambda|^{-s_3-\frac{3}{2}} (C_{h'}+C_{h})+|\lambda|^{-1},
\end{array}
\right.
\end{equation*}
where
\begin{equation*}
K_{11}=K_{9}+1,\quad K_{12}=K_{10}+\max\left(\kappa_2K_4,\tau|\delta_2|K_5\right),\quad K_{13}=\frac{\max\left(K_3+K_6+2,\max(K_7,K_8) \right)}{2}.
\end{equation*}
\end{lemma}
\begin{proof} First, from \eqref{NG-E(3.37)}, \eqref{NG-E(3.38)} and \eqref{NG-E(3.45)}, we get
\begin{equation*}
\begin{array}{lll}
\left\|{U}\right\|_{{\mathcal{H}_2}}^2\leq \left(K_9+1+K_9\left(|\lambda|^{\frac{1}{2}-s_1}+|\lambda|^{\frac{3}{2}-s_2}\right)\right)\int_\alpha^\beta |z|^2dx

\nline + K_{10}\left(1+|\lambda|^{\frac{1}{2}+s_2 }+|\lambda|^{-\frac{1}{2}+s_1}\right)\left(1+\lambda^{-2}\right)\left(\left\|F\right\|_{\mathcal{H}_2}\left\|U\right\|_{\mathcal{H}_2}+\left\|F\right\|_{\mathcal{H}_2}^2\right)\nline
+\max\left(\kappa_2K_4,\tau|\delta_2|K_5\right)\left(1+\lambda^{-2}\right) \left(\left\|F\right\|_{\mathcal{H}_2}\left\|U\right\|_{\mathcal{H}_2}+\left\|F\right\|_{\mathcal{H}_2}^2\right),

\end{array}
\end{equation*}
hence we get \eqref{NG-E(3.51)}. Next, multiplying  \eqref{NG-E(3.31)} by $-i \lambda^{-1} \overline{z}$ and integrating over $(\alpha,\beta),$ then taking the real part, then using  by parts integration, we get
\begin{equation*}
\begin{array}{lll}

\displaystyle{
\int_\alpha^\beta | z|^2\, dx  =-\text{Re}\left\{ i\,    \lambda^{-1}\int_\alpha^\beta f_5\overline{z} \, dx\right\}+ \text{Re}\left\{i \lambda^{-1}\int_\alpha^\beta\left(\kappa_2 v + \delta_1  z+\delta_2\eta(\cdot,1)\right)_{x}\,\overline{z}_x \, dx\right\}
}\nline

\displaystyle{
- \text{Re}\left\{i \lambda^{-1}\left(\kappa_2 v_x(\beta) + \delta_1  z_x(\beta)+\delta_2\eta_x(\beta,1)\right)\overline{z}(\beta)  \right\}
+\text{Re}\left\{i \lambda^{-1}\left(\kappa_2 v_x(\alpha) + \delta_1  z_x(\alpha)+\delta_2\eta_x(\alpha,1)\right)\overline{z}(\alpha)  \right\}

},
\end{array}
\end{equation*}
consequently,
\begin{equation}\label{NG-E(3.53)}
\begin{array}{lll}

\displaystyle{
\int_\alpha^\beta | z|^2\, dx \leq    |\lambda|^{-1}\int_\alpha^\beta |f_5||{z}| \, dx+ |\lambda|^{-1}\int_\alpha^\beta\left|\kappa_2 v_x + \delta_1  z_x+\delta_2\eta_x(\cdot,1)\right||{z}_x| \, dx
}\nline

\displaystyle{
+ |\lambda|^{-1}\left|\kappa_2 v_x(\beta) + \delta_1  z_x(\beta)+\delta_2\eta_x(\beta,1)\right||{z}(\beta) | + |\lambda|^{-1}\left|\kappa_2 v_x(\alpha) + \delta_1  z_x(\alpha)+\delta_2\eta_x(\alpha,1)\right||{z}(\alpha)|

}.
\end{array}
\end{equation}
Using Cauchy Schwarz inequality, we have
\begin{equation}\label{NG-E(3.54)}
|\lambda|^{-1}\int_\alpha^\beta |f_5||{z}| \, dx\leq |\lambda|^{-1}\left\|F\right\|_{\mathcal{H}_2}\left\|U\right\|_{\mathcal{H}_2}\leq |\lambda|^{-1}  \left(1+\lambda^{-2}\right)\left(\left\|F\right\|_{\mathcal{H}_2}\left\|U\right\|_{\mathcal{H}_2}+\left\|F\right\|_{\mathcal{H}_2}^2\right).
\end{equation}
From  \eqref{NG-E(3.36)} and \eqref{NG-E(3.39)}, we get 
\begin{equation*}
\begin{array}{lll}

\displaystyle{
 \int_\alpha^\beta\left|\kappa_2 v_x + \delta_1  z_x+\delta_2\eta_x(\cdot,1)\right||{z}_x| \, dx}\nline \leq\displaystyle{
\frac{1}{2}\int_\alpha^\beta|z_x|^2dx+\frac{1}{2}\int_\alpha^\beta\left|\kappa_2 v_x + \delta_1  z_x+\delta_2\eta_x(\cdot,1)\right|^2dx
}\nline

\leq \frac{K_3}{2}\left\|F\right\|_{\mathcal{H}_2}\left\|U\right\|_{\mathcal{H}_2}+\frac{K_6}{2}\left(1+\lambda^{-2}\right)\left(\left\|F\right\|_{\mathcal{H}_2}\left\|U\right\|_{\mathcal{H}_2}+\left\|F\right\|_{\mathcal{H}_2}^2\right)\nline
\leq

\frac{K_3+K_6}{2}\left(1+\lambda^{-2}\right)\left(\left\|F\right\|_{\mathcal{H}_2}\left\|U\right\|_{\mathcal{H}_2}+\left\|F\right\|_{\mathcal{H}_2}^2\right).

\end{array}
\end{equation*}
Inserting \eqref{NG-E(3.54)} and the above estimation  in \eqref{NG-E(3.53)}, we get
\begin{equation}\label{NG-E(3.55)}
\begin{array}{lll}

\displaystyle{
\int_\alpha^\beta | z|^2\, dx \leq   \frac{K_3+K_6+2}{2}  |\lambda|^{-1}\left(1+\lambda^{-2}\right)\left(\left\|F\right\|_{\mathcal{H}_2}\left\|U\right\|_{\mathcal{H}_2}+\left\|F\right\|_{\mathcal{H}_2}^2\right)}\nline

\displaystyle{
+ |\lambda|^{-1}\left|\kappa_2 v_x(\beta) + \delta_1  z_x(\beta)+\delta_2\eta_x(\beta,1)\right||{z}(\beta) | + |\lambda|^{-1}\left|\kappa_2 v_x(\alpha) + \delta_1  z_x(\alpha)+\delta_2\eta_x(\alpha,1)\right||{z}(\alpha)|

}.

\end{array}
\end{equation}
Now, for all $s_3\in\mathbb{R}$, $r_3\in\mathbb{R}^*_+$  and for $\zeta=\alpha$ or $\zeta=\beta$, we get
\begin{equation*}
|\lambda|^{-1} \left|\kappa_2 v_x(\zeta) + \delta_1  z_x(\zeta)+\delta_2\eta_x(\zeta,1)\right| \left|{z}(\zeta)\right|\leq 
 \frac{r_3|\lambda|^{s_3-\frac{1}{2}}}{2}|z(\zeta)|^2+\frac{|\lambda|^{-s_3-\frac{3}{2}}}{2r_3}\left|\kappa_2 v_x(\zeta) + \delta_1  z_x(\zeta)+\delta_2\eta_x(\zeta,1) \right|^2.
\end{equation*}
From the above inequality, we get
\begin{equation*}
\begin{array}{lll}

\displaystyle{
|\lambda|^{-1}\left|\kappa_2 v_x(\beta) + \delta_1  z_x(\beta)+\delta_2\eta_x(\beta,1)\right||{z}(\beta) | + |\lambda|^{-1}\left|\kappa_2 v_x(\alpha) + \delta_1  z_x(\alpha)+\delta_2\eta_x(\alpha,1)\right||{z}(\alpha)|}
\nline

\leq\frac{|\lambda|^{-s_3-\frac{3}{2}}}{2r_3}\left(\left|\kappa_2 v_x(\beta) + \delta_1  z_x(\beta) +\delta_2\eta_x(\beta,1)\right|^2+\left|\kappa_2 v_x(\alpha) + \delta_1  z_x(\alpha)+\delta_2\eta_x(\alpha,1) \right|^2\right)
\nline+ \frac{r_3\, |\lambda|^{s_3-\frac{1}{2}}}{2} \left(|z(\alpha)|^2+|z(\beta)|^2\right).
\end{array}
\end{equation*}
Inserting  \eqref{NG-E(3.41)} and \eqref{NG-E(3.42)} in the above estimation, we obtain
\begin{equation*}
\begin{array}{lll}

\displaystyle{
|\lambda|^{-1}\left|\kappa_2 v_x(\beta) + \delta_1  z_x(\beta)+\delta_2\eta_x(\beta,1)\right||{z}(\beta) | + |\lambda|^{-1}\left|\kappa_2 v_x(\alpha) + \delta_1  z_x(\alpha)+\delta_2\eta_x(\alpha,1)\right||{z}(\alpha)|}
\nline

\leq \frac{1}{2}\left(\frac{|\lambda|^{-s_3-s_2}}{r_2r_3}+\frac{r_3|\lambda|^{-s_1+s_3}}{r_1}+r_3C_{h'}|\lambda|^{s_3-\frac{1}{2}} \right)\int_\alpha^\beta|z|^2dx\nline

\ \ \ +\frac{\max(K_7,K_8)  }{2}\, R_{3,\lambda}\,\left(1+\lambda^{-2}\right)\left(\left\|F\right\|_{\mathcal{H}_2}\left\|U\right\|_{\mathcal{H}_2}+\left\|F\right\|_{\mathcal{H}_2}^2\right),

\end{array}
\end{equation*}
where
\begin{equation*}
R_{3,\lambda} =C_{h}^2 |\lambda|^{-1 }(r_1 r_3 |\lambda|^{s_1+s_3}+r_2 r_3^{-1}|\lambda|^{s_2-s_3 })+ r_3^{-1}|\lambda|^{-s_3-\frac{3}{2}} \left(C_{h'}+C_{h}\right).
\end{equation*}
Finally, inserting the above equation in \eqref{NG-E(3.55)}, we get
\begin{equation*}
\begin{array}{lll}

\left[
1- \frac{1}{2}\left(\frac{|\lambda|^{-s_3-s_2}}{r_2r_3}+\frac{r_3|\lambda|^{-s_1+s_3}}{r_1}+r_3C_{h'}|\lambda|^{s_3-\frac{1}{2}} \right)\right]\int_\alpha^\beta|z|^2dx

\nline
\leq

\frac{\max\left(K_3+K_6+2,\max(K_7,K_8) \right)}{2}\left(R_{3,\lambda}+|\lambda|^{-1}\right)\left(1+\lambda^{-2}\right)\left(\left\|F\right\|_{\mathcal{H}_2}\left\|U\right\|_{\mathcal{H}_2}+\left\|F\right\|_{\mathcal{H}_2}^2\right),

\end{array}
\end{equation*}
 hence we get \eqref{NG-E(3.52)}.
\end{proof}
\noindent \textbf{Proof of Proposition  \ref{NG-P-3.3}.} We now divide the proof in two steps: \\[0.1in]
\textbf{Step 1.} In this step, we prove the  asymptotic behavior estimate of \eqref{NG-E(3.34)}.  Taking $s_3=s_1=-s_2=\frac{1}{2}$, $r_1=\frac{1}{C_{h'}},\ r_2=9 C_{h'}$ and $r_3=\frac{1}{3 C_{h'}}$ in Lemma \ref{NG-L-3.7},  we get
\begin{equation*}
\left\{
\begin{array}{ll}
\frac{1}{2}\int_\alpha^\beta|z|^2dx\leq K_{13}\lambda^{-4 }\left( \frac{C_{h}^2}{3C_{h'}^2}\lambda^{2 }+|\lambda| +3 C_{h'}\left(9 C_{h}^2\, C_{h'}+C_{h}+C_{h'}\right)\right)\left(\lambda^{2}+1\right)\left(\left\|F\right\|_{\mathcal{H}_2}\left\|U\right\|_{\mathcal{H}_2}+\left\|F\right\|_{\mathcal{H}_2}^2\right),\nline

\left\|{U}\right\|_{{\mathcal{H}_2}}^2 \leq 2K_{11}\left(\lambda^{2}+1\right)\int_\alpha^\beta |z|^2dx +3K_{12}\lambda^{-2} \left(\lambda^{2}+1\right)\left(\left\|F\right\|_{\mathcal{H}_2}\left\|U\right\|_{\mathcal{H}_2}+\left\|F\right\|_{\mathcal{H}_2}^2\right).

\end{array}
\right.
\end{equation*}
In the above equation, using the fact that 
\begin{equation*}
\begin{array}{ll}

\frac{C_{h}^2}{3C_{h'}^2}\lambda^{2 }+|\lambda| +3 C_{h'}\left(9 C_{h}^2\, C_{h'}+C_{h}+C_{h'}\right)\leq \max\left(\frac{C_{h}^2}{3C_{h'}^2},3 C_{h'}\left(9 C_{h}^2\, C_{h'}+C_{h}+C_{h'}\right),1\right)\left(\lambda^2+|\lambda|+1\right)\nline
\hspace{6.6cm}\leq \max\left(\frac{C_{h}^2}{3C_{h'}^2},3 C_{h'}\left(9 C_{h}^2\, C_{h'}+C_{h}+C_{h'}\right),1\right)\left(|\lambda|+1\right)^2
\end{array}
\end{equation*}
and 
\begin{equation*}
\lambda^{2}+1\leq \left(|\lambda|+1\right)^2,
\end{equation*}
we get
\begin{equation}\label{NG-E(3.57)}
\int_\alpha^\beta|z|^2dx\leq K_{14} \lambda^{-4 }\left( |\lambda|+1\right)^4\left(\left\|F\right\|_{\mathcal{H}_2}\left\|U\right\|_{\mathcal{H}_2}+\left\|F\right\|_{\mathcal{H}_2}^2\right)
\end{equation}
and
\begin{equation}\label{NG-E(3.58)}
\left\|{U}\right\|_{{\mathcal{H}_2}}^2 \leq 2K_{11}\left(|\lambda|+1\right)^2\int_\alpha^\beta |z|^2dx +3K_{12} \lambda^{-2}\left(|\lambda|+1\right)^2\left(\left\|F\right\|_{\mathcal{H}_2}\left\|U\right\|_{\mathcal{H}_2}+\left\|F\right\|_{\mathcal{H}_2}^2\right),
\end{equation}
where 
\begin{equation*}
K_{14}=2K_{13} \max\left(\frac{C_{h}^2}{3C_{h'}^2},3 C_{h'}\left(9 C_{h}^2\, C_{h'}+C_{h}+C_{h'}\right),1\right).
\end{equation*}
Inserting \eqref{NG-E(3.57)} in \eqref{NG-E(3.58)}, we get
\begin{equation*}
\begin{array}{ll}

\left\|{U}\right\|_{{\mathcal{H}_2}}^2 &\leq
 \left(2K_{11} K_{14}\left( |\lambda|+1\right)^4 +3K_{12} \lambda^{2 }\right) \lambda^{-4}\left(|\lambda|+1\right)^2\left(\left\|F\right\|_{\mathcal{H}_2}\left\|U\right\|_{\mathcal{H}_2}+\left\|F\right\|_{\mathcal{H}_2}^2\right),\nline 
 &\leq  \left(2K_{11} K_{14}+3K_{12}\right) 
  \lambda^{-4}\left(|\lambda|+1\right)^6\left(\left\|F\right\|_{\mathcal{H}_2}\left\|U\right\|_{\mathcal{H}_2}+\left\|F\right\|_{\mathcal{H}_2}^2\right),
 \end{array}
\end{equation*}
hence we get \eqref{NG-E(3.34)}.\\[0.1in]
\textbf{Step 2.} In this step, we prove the  asymptotic behavior estimate of \eqref{NG-E(3.35)}. Let $M\in\mathbb{R}^*$  such that $|\lambda|\geq M>0$. In this case,  taking $s_1=s_2=s_3=0$, $r_1=\frac{3\sqrt{M}}{2C_{h'}}$, $r_2=\frac{3C_{h'}}{\sqrt{M}}$  and $r_3=\frac{\sqrt{M}}{2C_{h'}}$ in Lemma \ref{NG-L-3.7},  we get
\begin{equation}\label{NG-E(3.60)}
\begin{array}{lll}
\left\|{U}\right\|_{{\mathcal{H}_2}}^2 \leq

K_{11}|\lambda|^{\frac{3}{2}}\left(1+|\lambda|^{-1}+|\lambda|^{-\frac{3}{2}}\right)\int_\alpha^\beta |z|^2dx 

\nline

\hspace{1.2cm}
+K_{12}| \lambda|^{\frac{1}{2}}\left(1+|\lambda|^{-\frac{1}{2}}+|\lambda|^{-1}\right)\left(1+\lambda^{-2}\right)\left(\left\|F\right\|_{\mathcal{H}_2}\left\|U\right\|_{\mathcal{H}_2}+\left\|F\right\|_{\mathcal{H}_2}^2\right)

\end{array}
\end{equation}
and
\begin{equation}\label{NG-E(3.61)}
\begin{array}{lll}

\frac{1}{2}\left(1- \frac{\sqrt{M}}{2|\lambda|^{\frac{1}{2}}} \right)\int_\alpha^\beta|z|^2dx

\nline\leq

K_{13}|\lambda|^{-1 }\left[1+\frac{3 C_{h}^2\, M }{4 C_{h'}^2}+\frac{6 C_{h}^2\, C_{h'}^2}{M}+\frac{2 C_{h'}\left(C_{h}+C_{h'}\right)|\lambda|^{-\frac{1}{2}}}{\sqrt{M}}\right]\left(1+\lambda^{-2}\right)\left(\left\|F\right\|_{\mathcal{H}_2}\left\|U\right\|_{\mathcal{H}_2}+\left\|F\right\|_{\mathcal{H}_2}^2\right).

\end{array}
\end{equation}
From  the fact that $|\lambda|\geq M$, we get
\begin{equation*}
\frac{1}{2}\left(1-\frac{\sqrt{M}}{2|\lambda|^{\frac{1}{2}}}\right)\geq \frac{1}{4}>0 .
\end{equation*}
 Therefore, from the above inequality and \eqref{NG-E(3.61)}, we get
\begin{equation}\label{K0g-122}
\int_\alpha^\beta | z|^2\, dx \leq 
 K_{15}|\lambda|^{-1 }\left(1+ M+\frac{1}{M}+\frac{|\lambda|^{-\frac{1}{2}}}{\sqrt{M}}\right)\left(1+\lambda^{-2}\right)\left(\left\|F\right\|_{\mathcal{H}_2}\left\|U\right\|_{\mathcal{H}_2}+\left\|F\right\|_{\mathcal{H}_2}^2\right).
\end{equation}
where 
\begin{equation*}
K_{15}=4K_{13}\max\left[1,\ \frac{3 C_{h}^2 }{4 C_{h'}^2},\ 6 C_{h}^2\, C_{h'}^2,\ 2 C_{h'}\left(C_{h}+C_{h'}\right)\right].
\end{equation*}
In Estimation \eqref{NG-E(3.61)}, using the fact that
 \begin{equation*}
  \left\{
 \begin{array}{ll}
1+ M+\frac{1}{M}\leq \left(\sqrt{M}+\frac{1}{\sqrt{M}}\right)^2,\quad
\frac{1}{\sqrt{M}}\leq \left(\sqrt{M}+\frac{1}{\sqrt{M}}\right)^2,\nline
1+\lambda^{-2}\leq \left(1+|\lambda|^{-\frac{1}{2}}\right)^{4}.
 \end{array}
 \right.
 \end{equation*}
 we get
\begin{equation}\label{NG-E(3.62)}
\int_\alpha^\beta | z|^2\, dx \leq 
K_{15}|\lambda|^{-1}\left(\sqrt{M}+\frac{1}{\sqrt{M}}\right)^2\left(1+ |\lambda|^{-\frac{1}{2} }\right)^5\left(\left\|F\right\|_{\mathcal{H}_2}\left\|U\right\|_{\mathcal{H}_2}+\left\|F\right\|_{\mathcal{H}_2}^2\right).
\end{equation}
Inserting \eqref{NG-E(3.62)} in \eqref{NG-E(3.60)}, then using the fact that
 \begin{equation*}
1+|\lambda|^{-1}+|\lambda|^{-\frac{3}{2}}\leq \left( 1+|\lambda|^{-\frac{1}{2} }\right)^3,\quad
 \left(1+|\lambda|^{-\frac{1}{2} }+|\lambda|^{-1}\right)\left(
  1+\lambda^{-2}\right)\leq \left(1+|\lambda|^{-\frac{1}{2}}\right)^{6}\leq \left(1+|\lambda|^{-\frac{1}{2}}\right)^{8},
 \end{equation*}
 we get
\begin{equation*}
 \begin{array}{ll}
\left\|{U}\right\|_{{\mathcal{H}_2}}^2 \leq
\max\left(K_{11} K_{15},K_{12}\right)|\lambda|^{\frac{1}{2}}\left(1+ |\lambda|^{-\frac{1}{2} }\right)^8\left[\left(\sqrt{M}+\frac{1}{\sqrt{M}}\right)^2+1\right]\left(\left\|F\right\|_{\mathcal{H}_2}\left\|U\right\|_{\mathcal{H}_2}+\left\|F\right\|_{\mathcal{H}_2}^2\right)\nline
\hspace{1.1cm}\leq
2\max\left(K_{11} K_{15},K_{12}\right)|\lambda|^{\frac{1}{2}}\left(1+ |\lambda|^{-\frac{1}{2} }\right)^8\left(\sqrt{M}+\frac{1}{\sqrt{M}}\right)^2\left(\left\|F\right\|_{\mathcal{H}_2}\left\|U\right\|_{\mathcal{H}_2}+\left\|F\right\|_{\mathcal{H}_2}^2\right),
 \end{array}
\end{equation*}
hence we get estimation of  \eqref{NG-E(3.35)}. The proof is thus complete. \xqed{$\square$}$\\[0.1in]$
\noindent \textbf{Proof of Theorem \ref{NG-T-3.2}.} First, we will prove condition   \eqref{NG-E(3.24)}. Remark that  it has been proved in Proposition \ref{NG-P-3.1} that $0\in \rho(\mathcal{A}_2).$ Now, suppose \eqref{NG-E(3.24)} is not true, then there exists $\omega\in\mathbb{R}^*$ such that $i\omega\not\in\rho(\mathcal{A}_2)$.  According to Lemma \ref{NG-L-A.3} and Remark \ref{NG-R-A.4},  there exists 
 \begin{equation*}
 \left\{(\lambda_n,{U}_n:=\left(u_n,v_n,w_n ,y_n,z_n,\phi_n,\eta_n(\cdot,\cdot)\right))\right\}_{n\geq 1}\subset \mathbb{R}^*\times D\left(\mathcal{A}_2\right),
 \end{equation*}
with $\lambda_n \to  \omega$ as $n\to\infty,$ $|\lambda_n|<|\omega|$ and $\left\|{U}_n\right\|_{{\mathcal{H}_2}} = 1$, such that
\begin{equation*}
 (i \lambda_n I - \mathcal{A}_2){U}_n ={F}_n:=(f_{1,n},f_{2,n},f_{3,n},f_{4,n},f_{5,6},f_{6,n},f_{7,n}(\cdot,\cdot)) \to  0 \ \textrm{in} \ {\mathcal{H}_2},\qquad\text{as }n\to\infty.
\end{equation*}
We will check condition \eqref{NG-E(3.24)}  by finding a contradiction with $\left\|{U}_n\right\|_{{\mathcal{H}}_2} = 1$ such as $\left\| {U}_n\right\|_{\mathcal{H}_2} \to0.$  According to  Equation \eqref{NG-E(3.34)}  in Proposition \ref{NG-P-3.3} with $U=U_n,\ F=F_n$ and $\lambda=\lambda_n$,  we obtain 
\begin{equation*}
0\leq \left\|{U}_n\right\|_{{\mathcal{H}_2}}^2 \leq K_1|\lambda_n|^{-4} \left(|\lambda_n|+1\right)^6\left(\left\|F_n\right\|_{\mathcal{H}_2}\left\|U_n\right\|_{\mathcal{H}_2}+\left\|F_n\right\|_{\mathcal{H}_2}^2\right),
\end{equation*}
as $n\to\infty,$ we get $\left\|{U}_n\right\|_{{\mathcal{H}_2}}^2\to0,$ which contradicts $\left\|{U}_n\right\|_{{\mathcal{H}}_2} = 1$. Thus,  condition  \eqref{NG-E(3.24)} is holds true. Next, we will prove condition \eqref{NG-E(3.25)}  by a contradiction argument. Suppose  there exists 
 \begin{equation*}
 \left\{(\lambda_n,{U}_n:=\left(u_n,v_n,w_n ,y_n,z_n,\phi_n,\eta_n(\cdot,\cdot)\right))\right\}_{n\geq 1}\subset \mathbb{R}^*\times D\left(\mathcal{A}_2\right),
 \end{equation*}
with $|\lambda_n| \geq 1$ without affecting the result, such that $|\lambda_n| \to  + \infty,$  and  $\left\|{U}_n\right\|_{{\mathcal{H}_2}} = 1$ and there exists a sequence ${G}_n:=(g_{1,n},g_{2,n},g_{3,n},g_{4,n},g_{5,6},g_{6,n},g_{7,n}(\cdot,\cdot)) \in \mathcal{H}_2$, such that
\begin{equation*}
 (i \lambda_n I - \mathcal{A}_2){U}_n =\lambda_n^{-\frac{1}{2}} G_n \to  0 \ \textrm{in} \ {\mathcal{H}_2}.
\end{equation*}
We will check condition  \eqref{NG-E(3.25)}  by finding a contradiction with $\left\|{U}_n\right\|_{{\mathcal{H}}_2} = 1$ such as $\left\| {U}_n\right\|_{\mathcal{H}_2} =o(1).$  According to Equation \eqref{NG-E(3.35)}  in Proposition \ref{NG-P-3.3} with $U=U_n,\ F=\lambda^{-\frac{1}{2}}G_n,\ \lambda=\lambda_n$ and $M=1$, we get
\begin{equation*}
\left\|{U}_n\right\|_{{\mathcal{H}_2}}^2\leq 4K_2 \left(1+ |\lambda_n|^{-\frac{1}{2} }\right)^8\left(\left\|G_n\right\|_{\mathcal{H}_2}\left\|U_n\right\|_{\mathcal{H}_2}+|\lambda_n|^{-\frac{1}{2}}\left\|G_n\right\|_{\mathcal{H}_2}^2\right),
\end{equation*}
as $|\lambda_n|\to\infty,$ we get $\left\|{U}_n\right\|_{{\mathcal{H}_2}}^2 = o(1),$  which contradicts $\left\|{U}_n\right\|_{{\mathcal{H}}_2} = 1$. Thus,  condition  \eqref{NG-E(3.25)}  is holds true.  The result follows from Theorem  \ref{NG-T-A.5} (part (ii)). The proof is thus complete. \xqed{$\square$}
\begin{Remark}\label{NG-R-3.8} 
In the case that $\alpha=0$ and $\beta\neq L$ or $\beta=L$ and $\alpha\neq0$, we can proceed similar to the proof of Theorem \ref{NG-T-3.2}  to check that the energy of System \eqref{NG-E(3.1)}-\eqref{NG-E(3.7)} decays  polynomially of order $t^{-4}$. \xqed{$\square$}
 \end{Remark}
\appendix
\section{Notions of stability and theorems used}\label{NG-S-A}
\noindent We introduce here the notions of stability that we encounter in this work.
\begin{defi}\label{NG-D-A.1}
{Assume that $A$ is the generator of a C$_0$-semigroup of contractions $\left(e^{tA}\right)_{t\geq0}$  on a Hilbert space  $H$. The  $C_0$-semigroup $\left(e^{tA}\right)_{t\geq0}$  is said to be
\begin{enumerate}
\item[1.]  strongly stable if 
\begin{equation*}
\lim_{t\to +\infty} \|e^{tA}x_0\|_{H}=0, \quad\forall \ x_0\in H;
\end{equation*}
\item[2.]  exponentially (or uniformly) stable if there exist two positive constants $M$ and $\epsilon$ such that
\begin{equation*}
\|e^{tA}x_0\|_{H} \leq Me^{-\epsilon t}\|x_0\|_{H}, \quad
\forall\  t>0,  \ \forall \ x_0\in {H};
\end{equation*}
\item[3.] polynomially stable if there exists two positive constants $C$ and $\alpha$ such that
\begin{equation*}
 \|e^{tA}x_0\|_{H}\leq C t^{-\alpha}\|Ax_0\|_{H},  \quad\forall\ 
t>0,  \ \forall \ x_0\in D\left(A\right).
\end{equation*}
\xqed{$\square$}
\end{enumerate}}
\end{defi}
\noindent For proving the strong stability of the $C_0$-semigroup $\left(e^{tA}\right)_{t\geq0}$, we will recall two methods, the first result obtained by Arendt and Batty in \cite{Arendt01}. 
\begin{Theorem}[Arendt and Batty in \cite{Arendt01}]\label{NG-T-A.2}
{Assume that $A$ is the generator of a C$_0-$semigroup of contractions $\left(e^{tA}\right)_{t\geq0}$  on a Hilbert space $H$. If $A$ has no pure imaginary eigenvalues and  $\sigma\left(A\right)\cap i\mathbb{R}$ is countable,
where $\sigma\left(A\right)$ denotes the spectrum of $A$, then the $C_0$-semigroup $\left(e^{tA}\right)_{t\geq0}$  is strongly stable.}\xqed{$\square$}
\end{Theorem}
The second one is a classical method based on Arendt and Batty  theorem and the contradiction argument (see  page 25 in \cite{LiuZheng01}).
\begin{lemma}\label{NG-L-A.3}
 Assume that $A$ is the generator of a C$_0-$semigroup of contractions $\left(e^{tA}\right)_{t\geq0}$  on a Hilbert space $H$. Furthermore,  Assume that $0\in \rho(A).$ If there exists $\omega\in \mathbb{R}^*$, such that $i\omega\not\in \rho(A)$, then
 \begin{equation}\label{NG-E(A.1)}
 \left\{i\lambda\text{ such that }\lambda\in\mathbb{R}^* \text{ and } |\lambda|<\left\|A^{-1}\right\|^{-1}\leq|\omega|\right\}\subset\rho\left(A\right) \ \ \ \text{and} \ \ \ \sup_{|\lambda|<\left\|A^{-1}\right\|^{-1}\leq|\omega|}\left\|\left(i\lambda-A\right)^{-1}\right\|=\infty.
 \end{equation}
\end{lemma}
\begin{proof}
Since $0\in\rho(A)$, for any real number $\lambda$ with $|\lambda|< \left\|A^{-1}\right\|^{-1}$, we deduce from the contraction mapping theorem that operator $i\lambda I-A=-A^{-1} \left(I-i\lambda A\right)$ is invertible. Therefore,  we get
\begin{equation*}
\left\{i\lambda\text{ such that }\lambda\in\mathbb{R}^* \text{ and } |\lambda|<\left\|A^{-1}\right\|^{-1}\right\}\subset\rho\left(A\right).
\end{equation*}
In addition, if there exists $\omega\in \mathbb{R}^*$, such that $i\omega\not\in \rho(A)$, then $\left\|A^{-1}\right\|^{-1}\leq |w|$, hence we get the first estimation of \eqref{NG-E(A.1)}. Next, since $i\omega \not\in\rho(A)$, then for all $|\lambda|<\left\|A^{-1}\right\|^{-1}\leq |\omega|$ the operator 
\begin{equation*}
i\omega I- A=\left(i\lambda I-A\right)\left(I+i(\omega-\lambda)\left(i\lambda I-A\right)^{-1}\right)
\end{equation*}
is not invertible,  hence from the contraction mapping theorem we deduce 
\begin{equation*}
\left\|\left(i\lambda I-A\right)^{-1}\right\|\geq \frac{1}{|\omega-\lambda|},\quad \forall\ |\lambda|<\left\|A^{-1}\right\|^{-1}\leq |\omega|,
\end{equation*}
therefore
\begin{equation*}
\sup_{|\lambda|<\left\|A^{-1}\right\|^{-1}\leq|\omega|}\left\|\left(i\lambda I -A\right)^{-1}\right\|\geq \sup_{|\lambda|<\left\|A^{-1}\right\|^{-1}\leq|\omega|}\frac{1}{|\omega-\lambda|}=\infty, 
\end{equation*}
hence we get second  estimation of \eqref{NG-E(A.1)}.   
\end{proof}
\begin{Remark}\label{NG-R-A.4} Condition  \eqref{NG-E(A.1)}  turns out that there exists  $\left\{(\lambda_n,{U}_n)\right\}_{n\geq 1}\subset \mathbb{R}^*\times D\left(A\right),$ with $\lambda_n \to  \omega$ as $n\to\infty,$ $|\lambda_n|<|\omega|$ and $\left\|{U}_n\right\|_{{\mathcal{H}_2}} = 1$, such that
\begin{equation*}
 (i \lambda_n I - A){U}_n ={F}_n \to  0 \ \textrm{in} \ {H},\qquad\text{as }n\to\infty.
\end{equation*}
Then, we will check condition  $i\mathbb{R}\subset\rho(A)$  by finding a contradiction with $\left\|{U}_n\right\|_{H} = 1$ such as $\left\| {U}_n\right\|_{H} =o(1).$\xqed{$\square$}
\end{Remark}
\noindent We now recall the following standard result which is stated in a comparable way (see  \cite{Huang01,pruss01} for part (i) and \cite{Batty01,Borichev01,Batty01} for  part (ii)).
\begin{Theorem}\label{NG-T-A.5} Assume that $A$ is the generator of a strongly continuous semigroup of contractions $\left(e^{tA}\right)_{t\geq0}$  on $H$. Assume that $i\mathbb{R}\subset  \rho(A).$ Then;
   \begin{enumerate}
\item[(i)] The semigroup $e^{tA}$ is exponentially stable if and only if 
\begin{equation*}
\overline{\lim_{\lambda\to \infty}}\left\|\left(i\lambda I-A\right)^{-1}\right\|<\infty.
\end{equation*}
\item[(ii)] The semigroup $e^{tA}$ is polynomially stable of order $\alpha>0$ if and only if
\begin{equation*}
\overline{\lim_{\lambda\to \infty}}|\lambda|^{-\frac{1}{\alpha}}\left\|\left(i\lambda-A\right)^{-1}\right\|<\infty.
\end{equation*}
\xqed{$\square$}
\end{enumerate}
\end{Theorem}

\end{document}